\documentclass[12pt,titlepage]{article}

\usepackage{natbib}
\newcommand {\ctn}{\citet} 
\usepackage{graphicx,subfigure,amsmath,latexsym,amssymb}
\usepackage{float,epsfig,multirow,rotating,times}
\usepackage{upgreek,wrapfig}
\usepackage{comment}

\newtheorem{theorem}{Theorem}

\newtheorem{corollary}[theorem]{Corollary}

\newtheorem{lemma}[theorem]{Lemma}

\newtheorem{remark}[theorem]{Remark}

\newenvironment{proof}[1][Proof]{\textbf{#1.} }{\ \rule{0.5em}{0.5em}}

\numberwithin{equation}{section}
\numberwithin{algo}{section}
\numberwithin{table}{section}
\numberwithin{figure}{section}

\usepackage{setspace}
\usepackage[mathscr]{euscript}
\usepackage[margin=1in]{geometry}
\begin{document}


\title{\vspace{-0.8in}
Bayes Meets Riemann -- Bayesian Characterization of Infinite Series
with Application to Riemann Hypothesis}
\author{Sucharita Roy and Sourabh Bhattacharya\thanks{
Sucharita Roy is an Assistant Professor in St. Xavier's College, Kolkata, pursuing PhD 
in Interdisciplinary Statistical Research Unit, Indian Statistical Institute, 203, B. T. Road, Kolkata 700108.
Sourabh Bhattacharya is an Associate Professor in Interdisciplinary Statistical Research Unit, Indian Statistical
Institute, 203, B. T. Road, Kolkata 700108.
Corresponding e-mail: sourabh@isical.ac.in.}}
\date{\vspace{-0.5in}}
\maketitle%

\begin{abstract}

In the classical literature on infinite series there are various tests to determine if a given
infinite series converges, diverges, or oscillates. But unfortunately, for very many infinite series
all the existing tests can fail to provide definitive answers. 
In this article we propose a novel Bayesian theory for assessment of convergence properties of any given
infinite series. Remarkably, this theory attempts to provide conclusive answers to the question
of convergence even where all the existing tests of convergence fail. 
We apply our ideas to seven different examples, obtaining very encouraging results.
Importantly, we also apply our ideas
to investigate the Riemann Hypothesis, and obtain results that do not completely support the conjecture.

We also extend our ideas to develop a Bayesian theory on oscillating series, where we allow
even infinite number of limit points. Analysis of Riemann Hypothesis using Bayesian multiple limit points
theory yielded almost identical results as the Bayesian theory of convergence assessment.
\\[2mm]
\end{abstract}

\tableofcontents

\pagebreak

\section{Introduction}
\label{sec:intro}

Determination of convergence, divergence or oscillation of infinite series has a very rich
tradition in mathematics, and a large number of tests exist for the purpose. Unfortunately, 
there does not seem to exist any universal test that provides conclusive answers to all infinite series; see, 
for example, \ctn{Ilyin82}, \ctn{Knopp90}, \ctn{Bou12}. Attempts to resolve the issue as much as possible
using hierarchies of tests, with the successive tests in the hierarchy providing conclusive answers
to successively larger ranges of infinite series, are provided by \ctn{Knopp90}, \ctn{Bromwich05}, 
\ctn{Bou11} and \ctn{Lif11}. These tests are based on the Kummer approach for positive series
and the chain of the Ermakov tests for positive monotone series. 
The hierarchy of tests provided in \ctn{Bou12} are based on \ctn{Bromwich05} and are related to the well-known Cauchy's test
(see, for example, \ctn{Fich70}, \ctn{Rudin76}, \ctn{Spivak94}).
Below we briefly discuss the approach of \ctn{Bou12}, who consider positive series. It is important to
remark at the outset that positive series is not a requirement for the approaches that we propose and develop
in this article. 

\subsection{Hierarchical tests of convergence}
\label{subsec:hierarchy}
The tests of \ctn{Bou12} are based on the following theorem, 
which is a refinement of a result of \ctn{Bromwich05}.
\begin{theorem}[\ctn{Bou12}]
Let $\sum_{i=1}^{\infty}F'(i)$ be a divergent series where $F(x)>0$, $F'(x)>0$ and $F'(x)$ is decreasing.
If $\sum_{i=1}^{\infty}X_i$ is a positive series, then denoting 
$\frac{\log\left\{\frac{F'(i)}{X_i}\right\}}{\log F(i)}=W_i$, the following hold:
\begin{align}
&\mbox{If}~\underset{i\rightarrow\infty}{\lim\inf}~W_i>1,~\mbox{then}~\sum_{i=1}^{\infty}X_i~\mbox{converges};\notag\\
&\mbox{If}~\underset{i\rightarrow\infty}{\lim\sup}~W_i<1,~\mbox{then}~\sum_{i=1}^{\infty}X_i~\mbox{diverges}.\notag
\end{align}
\end{theorem}
Letting $F(z)=z$ in the above theorem, \ctn{Bou12} obtain their first test, which we provide below. 
\begin{theorem}[Test $T_1$ of \ctn{Bou12}]
Consider a positive series $\sum_{i=1}^{\infty}X_i$ and let $T_{1,i}=\frac{i}{\log i}\left(1-X^{\frac{1}{i}}_i\right)$.
Then
\begin{align}
&\mbox{If}~\underset{i\rightarrow\infty}{\lim\inf}~T_{1,i}>1,~\mbox{then}~\sum_{i=1}^{\infty}X_i~\mbox{converges};\notag\\
&\mbox{If}~\underset{i\rightarrow\infty}{\lim\sup}~T_{1,i}<1,~\mbox{then}~\sum_{i=1}^{\infty}X_i~\mbox{diverges}.\notag
\end{align}
\end{theorem}
This result is the same as that of \ctn{Bromwich05}, but a proof was not supplied in that work. 

Now choosing $F(z)=\log z$, \ctn{Bou12} form their second test of the hierarchy; we provide the result below.
Again, the result has been formulated by \ctn{Bromwich05}, but a proof was not given.
\begin{theorem}[Test $T_2$ of \ctn{Bou12}]
Consider a positive series $\sum_{i=1}^{\infty}X_i$ and let $T_{2,i}=\frac{\log i}{\log\log i}\left(T_{1,i}-1\right)$.
Then
\begin{align}
&\mbox{If}~\underset{i\rightarrow\infty}{\lim\inf}~T_{2,i}>1,~\mbox{then}~\sum_{i=1}^{\infty}X_i~\mbox{converges};\notag\\
&\mbox{If}~\underset{i\rightarrow\infty}{\lim\sup}~T_{2,i}<1,~\mbox{then}~\sum_{i=1}^{\infty}X_i~\mbox{diverges}.\notag
\end{align}
\end{theorem}
Setting $F(z)=\log\log z$, the following result has been proved by \ctn{Bou12}:
\begin{theorem}[Test $T_3$ of \ctn{Bou12}]
Consider a positive series $\sum_{i=1}^{\infty}X_i$ and let $T_{3,i}=\frac{\log i}{\log\log i}\left(T_{2,i}-1\right)$.
Then
\begin{align}
&\mbox{If}~\underset{i\rightarrow\infty}{\lim\inf}~T_{3,i}>1,~\mbox{then}~\sum_{i=1}^{\infty}X_i~\mbox{converges};\notag\\
&\mbox{If}~\underset{i\rightarrow\infty}{\lim\sup}~T_{3,i}<1,~\mbox{then}~\sum_{i=1}^{\infty}X_i~\mbox{diverges}.\notag
\end{align}
\end{theorem}
Successively selecting $F(z)=\log\log\log z$, $F(z)=\log\log\log\log z$, etc. successively more refined tests
$T_4$, $T_5$, etc. can be constructed, with each test having wider scope compared to the preceding test with regard 
to obtaining conclusive decision on convergence or divergence of the underlying series.  

However, if, say, at stage $k$, $\underset{i\rightarrow\infty}{\lim\inf}~T_{k,i}<1<
\underset{i\rightarrow\infty}{\lim\sup}~T_{k,i}$ so that $T_k$ is inconclusive, then all the subsequent tests 
will also fail to provide any conclusion.
Thus, in spite of the above developments, conclusion regarding the series can still be elusive. For instance, an example
considered in \ctn{Bou12} is the following series:
\begin{equation}
S_1=\sum_{i=3}^{\infty}\left(1-\frac{\log i}{i}-\frac{\log\log i}{i}
\left\{\cos^2\left(\frac{1}{i}\right)\right\}\left(a+(-1)^ib\right)\right)^i,
\label{eq:inconclusive1}
\end{equation}
where $a\geq 0$ and $b\geq 0$. For $a=b=1$, $\underset{i\rightarrow\infty}{\lim\inf}~T_{2,i}=0<1<2=
\underset{i\rightarrow\infty}{\lim\sup}~T_{2,i}$. Hence, the hierarchy of tests 
$\left\{T_k;k\geq 1\right\}$ fails to provide definitive answer to the question of convergence of the above series. 

In fact, we can generalize the series (\ref{eq:inconclusive1}) such that the hierarchy of tests fails for the general class
of series. Indeed, consider
\begin{equation}
S_2=\sum_{i=3}^{\infty}\left(1-\frac{\log i}{i}-\frac{\log\log i}{i}
f(i)\left(a+(-1)^ib\right)\right)^i,
\label{eq:inconclusive2}
\end{equation}
where $0\leq f(i)\leq 1$ for all $i=1,2,3,\ldots$, and $f(i)\rightarrow 1$ as $i\rightarrow\infty$.
Such a function can be easily constructed as follows. Let $g(i)$ be positive and monotonically increase
to $c$, where $c>0$. Then let $f(i)=g(i)/c$, for $i=1,2,3,\ldots$. A simple example of such a function
$g$ is $g(i)=c-\frac{1}{i}$; $g(i)=\cos^2\left(\frac{1}{i}\right)$ is another example, showing the generality
of (\ref{eq:inconclusive2}) compared to (\ref{eq:inconclusive1}).

\subsection{Riemann Hypothesis and series convergence}
\label{subsec:RH_series}
It is well-known that the famous Riemann Hypothesis is equivalent to convergence of an infinite
series on a certain interval. A brief introduction to the problem, along with
the necessary background, is provided in Section \ref{sec:RH}. Studying the relevant infinite
series, if at all possible, is then the most challenging problem of mathematics. The existing
mathematical literature, however, does not seem to be able to provide any directions in this regard.  
Hence, innovative theories and methods for analyzing infinite series should be particularly welcome.

Note that direct and successive evaluation of sums of consecutive terms of the deterministic series of interest need not even provide any insight into the
convergence behaviour of the series. This is because if the said sum seems to have approximately stabilized after a large number of successive evaluations, a further large number
of evaluations may reveal a slow increase of the sums. On the other hand, even though initially the sums might exhibit an increasing nature, eventually they might stabilize. 
To combat such problems, it would be worthwhile to create some appropriate transformation of the sums such that convergence of the series may be indicated if the transformed
sums approach a certain pre-defined value (say, 1), and divergence would be anticipated if the transformed sums approach another pre-defined value (say, 0), in a large number
of evaluations. Although these two pre-defined values and the progress of the transformed sums towards these values in a large, but finite number of evaluations do not, in any way, 
formally settle the question of convergence of the underlying series, strong evidence regarding the convergence behaviour may be gained, when the number of evaluations is considerably large.

In this paper, our approach of characterization of convergence properties of infinite series is based on the aforementioned intuition, which we formalize rigorously through 
a novel Bayesian procedure. We subsequently extend the idea and formalism to infinite series with multiple or even infinite
number of limit points. The main motivation and the idea of Bayesian formalism is illustrated in Section \ref{sec:key_concept}. 

\section{The key concept}
\label{sec:key_concept}

Let us assume that the terms $\left\{x_1,x_2,\ldots\right\}$ of any deterministic infinite series of the form $\sum_{i=1}^{\infty}x_i$ of interest is a realization
of some stochastic process $\left\{X_i:i=1,2,\ldots\right\}$, so that $\sum_{i=1}^{\infty}x_i$ is a realization of the corresponding  
random infinite series 
\begin{equation}
S_{1,\infty}=\sum_{i=1}^{\infty}X_i.
\label{eq:S}
\end{equation}
In the above, we do not assume any distributional form for $\left\{X_i:i=1,2,\ldots\right\}$, signifying the nonparametric nature of our problem.
Let $p\in[0,1]$ denote the probability of convergence the sum $S_{1,\infty}$. 
In particular, if $\left\{X_i:i=1,2,\ldots\right\}$ are independent,
then by Kolmogorov's 0-1 law (see, for example, \ctn{Stroock99}), $p$ is either $0$ or $1$, where $0$ 
stands for divergence of almost all realizations of $S_{1,\infty}$ and $1$ is associated with convergence of almost all realizations of $S_{1,\infty}$. 
Kolmogorov's three series theorem (see, for example, \ctn{Stroock99}) helps determine in this case if $p=0$ or $p=1$. However, the three series theorem requires parametric 
specification of the distributions of
$\left\{X_i:i=1,2,\ldots\right\}$, and specific choices of the parameters determine if $p=0$ or $p=1$. 
Since our goal is to determine the convergence behaviour of the deterministic series $\sum_{i=1}^{\infty}x_i$, interpreted as a realization of the specified stochastic process,
different choices of the parameters would lead to convergence and divergence of the same series, along with almost all other realizations of the stochastic process. In other words,
Kolmogorov's three series theorem is inappropriate when it comes to determination of convergence behaviour of deterministic series. 

If the random variables are not independent, then it may happen that some of the realizations of $S_{1,\infty}$ are convergent, some are divergent and the rest are oscillatory.
Since the above argument regarding Kolmogorov's three series theorem 
shows that it is inappropriate to assume parametric forms of the distributions of the random variables, we do not assume any particular distributional form of $\left\{X_i:i=1,2,\ldots\right\}$. 
It then follows that the value of $p$ is unknown, so that from the Bayesian perspective, one must acknowledge uncertainty about $p$ in the form of some appropriate prior.   

Now, specifying a prior directly on $p$ associated with the entire infinite series and computing the posterior given $\left\{X_i:i=1,2,\ldots\right\}$, is not a valid proposition, as
computing the likelihood would require evaluation of infinite number of terms associated with the infinite series, which amounts to knowing the convergence behaviour of the
series of interest. 
Instead, it makes sense to specify priors on the probabilities associated with the finite partial sums of the form 
$\sum_{i=m}^nX_i$, for $m\leq n$. Indeed, let
$$
P\left(\left|\sum_{i=m}^nX_i\right|\leq c_{m,n}\right)=p_{m,n},$$ 
where 
$c_{m,n}$ are non-negative quantities satisfying $c_{m,n}\downarrow 0$ as $m,n\rightarrow\infty$. 
Thus, the probability depends on how large $m$ and $n$ are. 

Now note that, as $m,n\rightarrow\infty$, 
$$\mathbb I_{\left\{\left|\sum_{i=m}^nX_i\right|\leq c_{m,n}\right\}}
\rightarrow \mathbb I_{\left\{\underset{m,n\rightarrow\infty}{\lim}~\left|\sum_{i=m}^nX_i\right|=0\right\}}$$
almost surely, so that uniform integrability leads to 
\begin{align}
\underset{m,n\rightarrow\infty}{\lim}~p_{m,n}
&=\underset{m,n\rightarrow\infty}{\lim}~P\left(\left|\sum_{i=m}^nX_i\right|\leq c_{m,n}\right)\notag\\
&=\underset{m,n\rightarrow\infty}{\lim}~E\left(\mathbb I_{\left\{\left|\sum_{i=m}^nX_i\right|\leq c_{m,n}\right\}}\right)
=E\left(\mathbb I_{\left\{\underset{m,n\rightarrow\infty}{\lim}~\left|\sum_{i=m}^nX_i\right|=0\right\}}\right)\notag\\
&=P\left(\underset{m,n\rightarrow\infty}{\lim}~\left|\sum_{i=m}^nX_i\right|=0\right)
=\underset{m,n\rightarrow\infty}{\lim}~p_{m,n}
=p,
\end{align}
so that it is sufficient to deal with $p_{m,n}$ associated with the partial sums rather than $p$. It is only required to ensure that the priors on $p_{m,n}$ are built such that 
given any realization $\left\{x_i:i=1,2,\ldots\right\}$ of the stochastic process $\left\{X_i:i=1,2,\ldots\right\}$ associated with the corresponding series of interest $\sum_{i=1}^{\infty}x_i$, 
the posterior corresponding to the prior of $p_{m,n}$, which we denote by 
$\pi_{m,n}\left(\cdot\Big |\left|\sum_{i=m}^nx_i\right|\right)$, converges
to $\pi\left(\cdot\Big |\underset{m,n\rightarrow\infty}{\lim}~\left|\sum_{i=m}^nx_i\right|\right)$, the posterior corresponding to the prior of $p$. 
Since the latter posterior is based on some given, single realization of the underlying stochastic process, the overall probability of convergence $p$ is informed with respect
to the conditioned single realization only. Consequently, the overall probability of convergence, given the series of interest, admits interpretation as the probability of convergence
of the series of interest. Hence, it is reasonable to require that, 
$\pi\left(\cdot\Big|\underset{m,n\rightarrow\infty}{\lim}~\left|\sum_{i=m}^nx_i\right|\right)=\delta_{\left\{z\right\}}(\cdot)$, the point mass at $z$, 
where $z=1$ or $z=0$ accordingly as $\underset{m,n\rightarrow\infty}{\lim}~\left|\sum_{i=m}^nx_i\right|$ is zero or positive, that is, 
accordingly as $\sum_{i=1}^{\infty}x_i$ is convergent or divergent. Thus, it is required to construct the priors on $p_{m,n}$ such that
$\pi_{m,n}\left(\cdot\Big |\left|\sum_{i=m}^nx_i\right|\right)\rightarrow\delta_{\left\{z\right\}}(\cdot)$ in some appropriate sense, as $m,n\rightarrow\infty$,
for any realization of the stochastic process.

It is important to appreciate that for another realization $\left\{\tilde x_i:i=1,2,\ldots\right\}$ of the underlying stochastic process, the corresponding infinite sum
$\sum_{i=1}^{\infty}\tilde x_i$ may have different convergence behaviour than $\sum_{i=1}^{\infty}x_i$. For instance, $\sum_{i=1}^{\infty}\tilde x_i$ may be divergent while 
$\sum_{i=1}^{\infty}x_i$ may be convergent.
Hence, the corresponding posteriors based on the partial sums of $\sum_{i=1}^{\infty}\tilde x_i$ will converge to $0$, while those associated with 
$\sum_{i=1}^{\infty}x_i$ will converge to $1$. Since
$p$ is the probability that $S_{1,\infty}$ converges, at first glance such discrepant posteriors may create the impression that the Bayesian inference procedure regarding $p$ is inconsistent.
However, as discussed above, given only the series of interest, the overall probability of convergence $p$ admits interpretability as the probability of convergence of the series at hand. 
This is exactly what is desired, since our goal is to study the convergence properties of the series of our interest only, not to learn about $p$. As an aside, note that it is of course possible 
to learn about $p$ via its posterior distribution which may be obtained by conditioning on adequate number of realizations (instead of
a single realization) of the stochastic process as in the usual Bayesian inference problems of learning about unknown parameters. 


In Section \ref{sec:recursive1} we devise a recursive Bayesian methodology that achieves the goal discussed above.
It is important to remark that no restrictive assumption is necessary for the development of our ideas, 
not even independence of $X_i$.
With this methodology, we then characterize convergence and divergence of infinite series 
in Section \ref{sec:characterization}, illustrating in Section \ref{sec:illustrations} our theory and methods with seven
examples. In Section \ref{sec:RH} we apply our ideas to Riemann Hypothesis, obtaining results that 
are not in complete favour of the conjecture. 
We also extend our theory and methods to infinite series with multiple or infinite number
of limit points; details are provided in Section S-3 of the supplement. 
Illustrations of our Bayesian multiple limit point theory are provided in Sections S-4 and S-5
of the supplement, 
the latter section detailing the application to Riemann
Hypothesis in order to vindicate our results obtained in Section \ref{sec:RH}. Finally, we make
concluding remarks in Section \ref{sec:conclusion}.

\section{A recursive Bayesian procedure for studying infinite series}
\label{sec:recursive1}

Since we view $X_i$ as realizations from some random process, we first formalize the notion
in terms of the relevant probability space.
Let $(\Omega,\mathcal A,\mu)$ be a probability space, where $\Omega$ is the sample space,
$\mathcal A$ is the Borel $\sigma$-field on $\Omega$, and $\mu$ is some probability measure.
Let, for $i=1,2,3,\ldots$, $X_i:\Omega\mapsto\mathbb R$ be real valued random variables
measurable with respect to the Borel $\sigma$-field $\mathcal B$ on $\mathbb R$.
As in \ctn{Schervish95}, we can then define a $\sigma$-field of subsets of $\mathbb R^{\infty}$ with
respect to which $X=(X_1,X_2,\ldots)$ is measurable. Indeed, let us define $\mathbb B^{\infty}$ to be
the smallest $\sigma$-field containing sets of the form 
\begin{align}
B&=\left\{X:X_{i_1}\leq r_1,X_{i_2}\leq r_2,\ldots,X_{i_p}\leq r_p,~\mbox{for some}~p\geq 1,\right.\notag\\
&\quad\quad\left.~\mbox{some integers}~
i_1,i_2,\ldots,i_p,~\mbox{and some real numbers}~r_1,r_2,\ldots,r_p\right\}.\notag
\end{align}
Since $B$ is an intersection of finite number of sets  
of the form $\left\{X:X_{i_j}\leq r_j\right\}$; $j=1,\ldots,p$, all of which belong to $\mathcal A$ (since
$X_{i_j}$ are measurable)
it follows that $X^{-1}(B)\in\mathcal A$, so that $X$ is measurable with respect to 
$(\mathbb R^{\infty},\mathbb B^{\infty},P)$, where $P$ is the probability measure induced by $\mu$.

Alternatively, note that it is possible to represent any stochastic process $\{X_i:i\in \mathfrak I\}$, for fixed
$i$ as a random variable $\omega\mapsto X_i(\omega)$, where $\omega\in\Omega$;
$\Omega$ being the set of all functions from $\mathfrak I$ into $\mathbb R$. 
Also, fixing $\omega\in\Omega$, the function $i\mapsto X_i(\omega);~i\in \mathfrak I$,
represents a path of $X_i;~i\in\mathfrak I$. Indeed, we can identify $\omega$ with the function
$i\mapsto X_i(\omega)$ from $\mathfrak I$ to $\mathbb R$; see, for example, \ctn{Oksendal00}, for
a lucid discussion.

This latter identification will be convenient for our purpose, and we adopt this in this article.
Note that the $\sigma$-algebra $\mathcal F$ induced by $X$
is generated by sets of the form
\[
\left\{\omega:\omega(i_1)\in B_1,\omega(i_2)\in B_2,\ldots,\omega(i_k)\in B_k\right\},
\]
where $B_j\subset\mathbb R;~j=1,\ldots,k$, are Borel sets in $\mathbb R$.   

\subsection{Development of the stage-wise likelihoods}
\label{subsec:Bayesian_method}

For $j=1,2,3,\ldots$, let
\begin{equation}
S_{j,n_j}=\sum_{i=\sum_{k=0}^{j-1}n_k+1}^{\sum_{k=0}^jn_k}X_i,
\label{eq:S_j_n}
\end{equation}
where $n_0=0$ and $n_j\geq 1$ for all $j\geq 1$.
Also let $\{c_j\}_{j=1}^{\infty}$ be a non-negative decreasing sequence and
\begin{equation}
Y_{j,n_j}=\mathbb I_{\left\{\left|S_{j,n_j}\right|\leq c_j\right\}}.
\label{eq:Y_j_n}
\end{equation}
Let, for $j\geq 1$,
\begin{equation}
P\left(Y_{j,n_j}=1\right)=p_{j,n_j}.
\label{eq:p_j_n}
\end{equation}
Hence, the likelihood of $p_{j,n_j}$, given $y_{j,n_j}$, is given by
\begin{equation}
L\left(p_{j,n_j}\right)=p^{y_{j,n_j}}_{j,n_j}\left(1-p_{j,n_j}\right)^{1-y_{j,n_j}}
\label{eq:likelihood}
\end{equation}
It is important to relate $p_{j,n_j}$ to convergence or divergence of the underlying series.
Note that $p_{j,n_j}$ is the probability that $|S_{j,n_j}|$ falls below $c_j$. Thus, 
$p_{j,n_j}$ can be interpreted as the probability that the
series $S_{1,\infty}$ is convergent when the data observed is $S_{j,n_j}$. 
If $S_{1,\infty}$ is convergent, then it is to be expected {\it a posteriori}, that  
\begin{equation}
p_{j,n_j}\rightarrow 1\quad\mbox{as}~j\rightarrow\infty.
\label{eq:convergent_p}
\end{equation}
Note that the above is expected to hold even for $n_j=n$ for all $j\geq 1$, and for all $n\geq 1$. This is related
to Cauchy's criterion of convergence of partial sums: for every $\epsilon>0$ there exists a positive
integer $N$ such that for all $n\geq m\geq N$, $|\sum_{i=m}^nX_i|<\epsilon$. 
Indeed, as we will formally show, condition (\ref{eq:convergent_p}) 
is both necessary and sufficient for convergence of the series. 

On the other hand, if the series is divergent, then there exist $j_0\geq 1$ 
such that for every $j>j_0$ there exists $n_j\geq 1$ satisfying $|S_{j,n_j}|>c_j$. Here we expect, 
{\it a posteriori}, that
\begin{equation}
p_{j,n_j}\rightarrow 0\quad\mbox{as}~j\rightarrow\infty.
\label{eq:divergent_p}
\end{equation}
Again, we will prove formally that the above condition is both necessary and sufficient for divergence.


In this work we call the series $S_{1,\infty}$ oscillating if the sequence 
$\left\{S_{1,n};~n=1,2,\ldots\right\}$ has more than one limit points.
Thus, these are non-convergent series, and so, the probability of convergence of these series
must tend to zero in our Bayesian framework, which is in fact ensured by our theoretical developments.
But it is also important to be able to categorize and learn about the limit points. 
A general theory, which encompasses finite as well as infinite number of limit points,
with perhaps unequal frequencies of occurrences, is developed in Section S-3 of the supplement. 

In what follows we shall first construct a recursive Bayesian methodology that formally characterizes
convergence and divergence in terms of formal posterior convergence related to (\ref{eq:convergent_p})
and (\ref{eq:divergent_p}).

\subsection{Development of recursive Bayesian posteriors}
\label{subsec:recursive_posteriors}


We assume that $\left\{y_{j,n_j};j=1,2,\ldots\right\}$ is observed successively at stages indexed by $j$.
That is, we first observe $y_{1,n_1}$, and based on our prior belief regarding the first stage probability, 
$p_{1,n_1}$, compute the posterior distribution of $p_{1,n_1}$ given $y_{1,n_1}$, which we denote by
$\pi(p_{1,n_1}|y_{1,n_1})$.
Based on this posterior we construct a prior for the second stage, and compute the posterior
$\pi(p_{2,n_2}|y_{1,n_1},y_{2,n_2})$. We continue this procedure for as many stages as we desire.
Details follow.

Consider the sequences $\left\{\alpha_j\right\}_{j=1}^{\infty}$ and $\left\{\beta_j\right\}_{j=1}^{\infty}$,
where $\alpha_j=\beta_j=1/j^2$ for $j=1,2,\ldots$.
At the first stage of our recursive Bayesian algorithm, that is, when $j=1$, 
let us assume that the prior is given by
\begin{equation}
\pi(p_{1,n_1})\equiv Beta(\alpha_1,\beta_1),
\label{eq:prior_stage_1}
\end{equation}
where, for $a>0$ and $b>0$, $Beta(a,b)$ denotes the Beta distribution with mean $a/(a+b)$
and variance $(ab)/\left\{(a+b)^2(a+b+1)\right\}$.
Combining this prior with the
likelihood (\ref{eq:likelihood}) (with $j=1$), we obtain the following posterior of $p_{1,n_1}$ given $y_{1,n_1}$:
\begin{equation}
\pi(p_{1,n_1}|y_{1,n_1})\equiv Beta\left(\alpha_1+y_{1,n_1},\beta_1+1-y_{1,n_1}\right).
\label{eq:posterior_stage_1}
\end{equation}
At the second stage (that is, for $j=2$), for the prior of $p_{2,n_2}$ we consider the posterior
of $p_{1,n_1}$ given $y_{1,n_1}$ associated with the $Beta(\alpha_1+\alpha_2,\beta_1+\beta_2)$ prior.
That is, our prior on $p_{2,n_2}$ is given by:
\begin{equation}
\pi(p_{2,n_2})\equiv Beta\left(\alpha_1+\alpha_2+y_{1,n_1},\beta_1+\beta_2+1-y_{1,n_1}\right).
\label{eq:prior_stage_2}
\end{equation}
The reason for such a prior choice is that the uncertainty regarding convergence of the series
is reduced once we obtain the posterior at the first stage, so that at the second stage the uncertainty 
regarding the prior is expected to be lesser compared to the first stage posterior. With our choice, it 
is easy to see that the prior variance at the second stage, given by 
$$\left\{(\alpha_1+\alpha_2+y_{1,n_1})(\beta_1+\beta_2+1-y_{1,n_1})\right\}/\left\{(\alpha_1+\alpha_2+\beta_1+\beta_2+1)^2
(\alpha_1+\alpha_2+\beta_1+\beta_2+2)\right\},$$ 
is smaller than the first stage posterior variance, given by
$$\left\{(\alpha_1+y_{1,n_1})(\beta_1+1-y_{1,n_1})\right\}/\left\{(\alpha_1+\beta_1+1)^2
(\alpha_1+\beta_1+2)\right\}.$$ 

The posterior of $p_{2,n_2}$ given $y_{2,n_2}$ is then obtained by combining the second stage prior
(\ref{eq:prior_stage_2}) with (\ref{eq:likelihood}) (with $j=2$). The form of the posterior
at the second stage is thus given by
\begin{equation}
\pi(p_{2,n_2}|y_{2,n_2})\equiv Beta\left(\alpha_1+\alpha_2+y_{1,n_1}+y_{2,n_2},\beta_1+\beta_2+2-y_{1,n_1}-y_{2,n_2}\right).
\label{eq:posterior_stage_2}
\end{equation}

Continuing this way, at the $k$-th stage, where $k>1$, we obtain the following posterior of $p_{k,n_k}$:
\begin{equation}
\pi(p_{k,n_k}|y_{k,n_k})\equiv Beta\left(\sum_{j=1}^k\alpha_j+\sum_{j=1}^ky_{j,n_j},
k+\sum_{j=1}^k\beta_j-\sum_{j=1}^ky_{j,n_j}\right).
\label{eq:posterior_stage_k}
\end{equation}

It follows from (\ref{eq:posterior_stage_k}) that
\begin{align}
E\left(p_{k,n_k}|y_{k,n_k}\right)&=\frac{\sum_{j=1}^k\alpha_j
+\sum_{j=1}^ky_{j,n_j}}{k+\sum_{j=1}^k\alpha_j+\sum_{j=1}^k\beta_j};
\label{eq:postmean_p_k}\\
Var\left(p_{k,n_k}|y_{k,n_k}\right)&=
\frac{(\sum_{j=1}^k\alpha_j+\sum_{j=1}^ky_{j,n_j})(k+\sum_{j=1}^k\beta_j-\sum_{j=1}^ky_{j,n_j})}
{(k+\sum_{j=1}^k\alpha_j+\sum_{j=1}^k\beta_j)^2(1+k+\sum_{j=1}^k\alpha_j+\sum_{j=1}^k\beta_j)}.
\label{eq:postvar_p_k}
\end{align}
Since $\sum_{j=1}^k\alpha_j=\sum_{j=1}^k\beta_j=\sum_{j=1}^k\frac{1}{j^2}$, (\ref{eq:postmean_p_k})
and (\ref{eq:postvar_p_k}) admit the following simplifications:
\begin{align}
E\left(p_{k,n_k}|y_{k,n_k}\right)&=\frac{\sum_{j=1}^k\frac{1}{j^2}+\sum_{j=1}^ky_{j,n_j}}
{k+2\sum_{j=1}^k\frac{1}{j^2}};
\label{eq:postmean_p_k_2}\\
Var\left(p_{k,n_k}|y_{k,n_k}\right)&=
\frac{(\sum_{j=1}^k\frac{1}{j^2}+\sum_{j=1}^ky_{j,n_j})(k+\sum_{j=1}^k\frac{1}{j^2}-\sum_{j=1}^ky_{j,n_j})}
{(k+2\sum_{j=1}^k\frac{1}{j^2})^2(1+k+2\sum_{j=1}^k\frac{1}{j^2})}.
\label{eq:postvar_p_k_2}
\end{align}

\section{Characterization of convergence properties of the underlying infinite series}
\label{sec:characterization}
Based on our recursive Bayesian theory we have the following theorem that characterizes 
convergence of $S_{1,\infty}$ in terms of the limit of the posterior
probability of $p_{k,n_k}$, as $k\rightarrow\infty$.
Note that the sample space of $S_{1,\infty}$ is also given by $\mathfrak S$.
We also assume, for the sake of generality, that for any $\omega\in\mathfrak S\cap\mathfrak N^c$, where
$\mathfrak N~(\subset\mathfrak S)$ has zero probability measure, the non-negative monotonically 
decreasing sequence $\{c_j\}_{j=1}^{\infty}$
depends upon $\omega$, so that we shall denote the sequence by $\{c_j(\omega)\}_{j=1}^{\infty}$.
In other words, we allow $\left\{c_j(\omega)\right\}_{j=1}^{\infty}$ to depend upon the corresponding series 
$S_{1,\infty}(\omega)$. 
Note that if $S_{1,\infty}(\omega)<\infty$, then 
the sequence $\left\{|S_{j,n_j}(\omega)|\right\}_{j=1}^{\infty}$ is uniformly bounded, 
for all sequences $\{n_j\}_{j=1}^{\infty}$, 
and converges to zero for all sequences $\{n_j\}_{j=1}^{\infty}$, which implies that there exists a 
monotonically decreasing sequence $\left\{c_j(\omega)\right\}_{j=1}^{\infty}$ independent of the choice of
$\{n_j\}_{j=1}^{\infty}$ such that for some $j_0(\omega)\geq 1$,
\begin{equation}
|S_{j,n_j}(\omega)|\leq c_j(\omega),~\mbox{for}~j\geq j_0(\omega). 
\label{eq:bound_S}
\end{equation}
Indeed, in most of our illustrations presented in this paper, including the Riemann Hypothesis, we
choose $\left\{c_j(\omega)\right\}_{j=1}^{\infty}$ in a way that depends upon the infinite series at hand.
\begin{theorem}
\label{theorem:convergence}
For any $\omega\in\mathfrak S\cap\mathfrak N^c$, where $\mathfrak N$ is some null set having probability measure zero, 
$S_{1,\infty}(\omega)<\infty$ if and only if 
there exists a non-negative monotonically decreasing sequence
$\left\{c_j(\omega)\right\}_{j=1}^{\infty}$ such that 
for any choice of the sequence $\{n_j\}_{j=1}^{\infty}$,
\begin{equation}
\pi\left(\mathcal N_1|y_{k,n_k}(\omega)\right)\rightarrow 1,
\label{eq:consistency_at_1}
\end{equation}
as $k\rightarrow\infty$, 
where $\mathcal N_1$ is any neighborhood of 1 (one).
\end{theorem}
\begin{proof}
Let, for $\omega\in\mathfrak S\cap\mathfrak N^c$, 
$S_{1,\infty}(\omega)$ be convergent. 
Then, by (\ref{eq:bound_S}), $|S_{j,n_j}(\omega)|\leq c_j(\omega)$ for all $n_j$, so that $y_{j,n_j}(\omega)=1$ 
for all $j>j_0(\omega)$, for all $n_j$. 
Hence, in this case, $\sum_{j=1}^ky_{j,n_j}(\omega)=k-k_0(\omega)$,
where $k_0(\omega)\geq 0$. Also, $\sum_{j=1}^k\frac{1}{j^2}\rightarrow\frac{\pi^2}{6}$, as $k\rightarrow\infty$. 
Consequently, it is easy to see that
\begin{align}
\mu_k=E\left(p_{k,n_k}|y_{k,n_k}(\omega)\right)&\sim\frac{\frac{\pi^2}{6}+k-k_0(\omega)}{k+\frac{\pi^2}{3}}
\rightarrow 1,~\mbox{as}~k\rightarrow\infty,~\mbox{and},
\label{eq:postmean_p_k_limit}\\
\sigma^2_k=Var\left(p_{k,n_k}|y_{k,n_k}(\omega)\right)&\sim
\frac{(\frac{\pi^2}{6}+k)(\frac{\pi^2}{6})}{(k+\frac{\pi^2}{3})^2(1+k+\frac{\pi^2}{3})}
\rightarrow 0~\mbox{as}~k\rightarrow\infty.
\label{eq:postvar_p_k_limit}
\end{align}
In the above, for any two sequences $\left\{a_k\right\}_{k=1}^{\infty}$ and $\left\{b_k\right\}_{k=1}^{\infty}$,
$a_k\sim b_k$ indicates $\frac{a_k}{b_k}\rightarrow 1$, as $k\rightarrow\infty$. 
Now let $\mathcal N_1$ denote any neighborhood of 1, and let $\epsilon>0$ be sufficiently small such that
$\mathcal N_1\supseteq\left\{1-p_{k,n_k}<\epsilon\right\}$. Combining (\ref{eq:postmean_p_k_limit})
and (\ref{eq:postvar_p_k_limit}) with Chebychev's inequality ensures
that (\ref{eq:consistency_at_1}) holds. 

Now assume that (\ref{eq:consistency_at_1}) holds. 
Then for any given $\epsilon>0$, 
\begin{equation}
\pi\left(p_{k,n_k}>1-\epsilon|y_{k,n_k}(\omega)\right)\rightarrow 1,~\mbox{as}~k\rightarrow\infty.
\label{eq:post1}
\end{equation}
Hence, it can be seen, using Markov's inequality, that
\begin{align}
E\left(p_{k,n_k}|y_{k,n_k}(\omega)\right)&\rightarrow 1;
\label{eq:postmean1}\\
Var\left(p_{k,n_k}|y_{k,n_k}(\omega)\right)&\rightarrow 0,
\label{eq:postvar1}
\end{align}
as $k\rightarrow\infty$.
If $S_{1,\infty}(\omega)$ does not converge then there exists $j_0(\omega)$ such that for each 
$j\geq j_0(\omega)$, there exists $n_j(\omega)$ satisfying 
$\left|S_{j,n_j(\omega)}(\omega)\right|>c_j(\omega)$, for any choice of non-negative sequence 
$\{c_j(\omega)\}_{j=1}^{\infty}$ monotonically converging to zero. 
Hence, in this situation, 
$0\leq \sum_{j=1}^ky_{j,n_j(\omega)}(\omega)\leq j_0(\omega)$.
Substituting this in (\ref{eq:postmean_p_k_2}) and (\ref{eq:postvar_p_k_2}), it is easy to see that,
as $k\rightarrow\infty$,
\begin{align}
E\left(p_{k,n_k(\omega)}|y_{k,n_k(\omega)}(\omega)\right)\rightarrow 0;
\label{eq:postmean_div}\\
Var\left(p_{k,n_k(\omega)}|y_{k,n_k(\omega)}(\omega)\right)\rightarrow 0,
\label{eq:postvar_div}
\end{align}
so that (\ref{eq:postmean1}) is contradicted.

\end{proof}

We now prove the following theorem that provides necessary and sufficient conditions for
divergence of $S_{1,\infty}(\omega)$ in terms of the limit of the posterior
probability of $p_{k,n_k(\omega)}$, as $k\rightarrow\infty$.
\begin{theorem}
\label{theorem:divergence}
For any $\omega\in\mathfrak S\cap\mathfrak N^c$, where $\mathfrak N$ is some null set having probability measure zero, 
$S_{1,\infty}(\omega)$ is divergent if and only if 
there exists a sequence 
$\{n_j(\omega)\}_{j=1}^{\infty}$ such that
\begin{equation}
\pi\left(\mathcal N_0|y_{k,n_k(\omega)}(\omega)\right)\rightarrow 1,
\label{eq:consistency_at_0}
\end{equation}
$k\rightarrow\infty$, 
where $\mathcal N_0$ is any neighborhood of 0 (zero).
\end{theorem}
\begin{proof}
Assume that $S_{1,\infty}(\omega)$ is divergent. Then 
then there exist $j_0(\omega)\geq 1$ such that for every $j\geq j_0(\omega)$, one can find
$n_j(\omega)$ satisfying
$\left|S_{j,n_j(\omega)}(\omega)\right|>c_j(\omega)$, for any choice of non-negative sequence $\{c_j(\omega)\}_{j=1}^{\infty}$
monotonically converging to zero. 
From the proof of the sufficient condition of Theorem \ref{theorem:convergence} it follows that
(\ref{eq:postmean_div}) and (\ref{eq:postvar_div}) hold.
Let $\epsilon>0$ be small enough so that $\mathcal N_0\supseteq\left\{p_{k,n_k(\omega)}<\epsilon\right\}$. Then
combining Chebychev's inequality with (\ref{eq:postmean_div}) and (\ref{eq:postvar_div})
it is easy to see that (\ref{eq:consistency_at_0}) holds.

Now assume that (\ref{eq:consistency_at_0}) holds. 
Then for any given $\epsilon>0$, 
\begin{equation}
\pi\left(p_{k,n_k(\omega)}<\epsilon|y_{k,n_k(\omega)}(\omega)\right)\rightarrow 1,~\mbox{as}~k\rightarrow\infty.
\label{eq:post2}
\end{equation}
It can be seen, now using Markov's inequality with respect to $1-p_{k,n_k(\omega)}$, that 
\begin{align}
E\left(p_{k,n_k(\omega)}|y_{k,n_k(\omega)}(\omega)\right)&\rightarrow 0; 
\label{eq:postmean2}\\
Var\left(p_{k,n_k(\omega)}|y_{k,n_k(\omega)}\right)&\rightarrow 0,
\label{eq:postvar2}
\end{align}
as $k\rightarrow\infty$.  

If $S_{1,\infty}(\omega)$ is convergent, then by Theorem \ref{theorem:convergence}, 
$\pi\left(\mathcal N_1|y_{k,n_k}(\omega)\right)\rightarrow 1$ as $k\rightarrow\infty$, for
all sequences $\{n_j\}_{j=1}^{\infty}$, so that
$E\left(p_{k,n_k(\omega)}|y_{k,n_k(\omega)}(\omega)\right)\rightarrow 1$, which is a contradiction to (\ref{eq:postmean2}). 

\end{proof}

Note that Theorem \ref{theorem:divergence} encompasses even oscillatory series. For instance, if for some
$\omega\in\mathfrak S\cap\mathfrak N^c$, $S_{1,\infty}(\omega)=\sum_{i=1}^{\infty}\left(-1\right)^{i-1}$, then
the sequence $n_j(\omega)=1+3(j-1)$ ensures that $|S_{j,n_j}(\omega)|>c_j(\omega)$ for all $j\geq j_0(\omega)$,
for some $j_0(\omega)\geq 1$, for any monotonically decreasing non-negative sequence $\{c_j(\omega)\}_{j=1}^{\infty}$. 
This of course forces declaration of divergence of this particular series, 
as per Theorem \ref{theorem:divergence}.
We show in Section S-4.1 of the supplement, 
with the help of our Bayesian idea of studying
oscillatory series, how to identify the number and proportions of the limit points of this oscillatory series.

\subsection{Characterization of infinite series using non-recursive Bayesian posteriors}
\label{subsec:non_recursive}
Observe that it is not strictly necessary for the prior at any stage to depend upon the previous stage.
Indeed, we may simply assume that $\pi\left(p_{j,n_j}\right)\equiv Beta\left(\alpha_j,\beta_j\right)$,
for $j=1,2,\ldots$. In this case, the posterior of $p_{k,n_k}$ given $y_{k,n_k}$ is simply
$Beta\left(\alpha_k+y_{k,n_k},1+\beta_k-y_{k,n_k}\right)$. The posterior mean and variance are then given by
\begin{align}
E\left(p_{k,n_k}|y_{k,n_k}(\omega)\right)&=\frac{\alpha_k+y_{k,n_k}(\omega)}
{1+\alpha_k+\beta_k};
\label{eq:postmean_p_k_3}\\
Var\left(p_{k,n_k}|y_{k,n_k}(\omega)\right)&=
\frac{(\alpha_k+y_{k,n_k}(\omega))(1+\beta_k-y_{k,n_k}(\omega))}
{(1+\alpha_k+\beta_k)^2(2+\alpha_k+\beta_k)}.
\label{eq:postvar_p_k_3}
\end{align}
Since $y_{k,n_k}(\omega)$ converges to $1$ or $0$ as $k\rightarrow\infty$, accordingly as
$S_{1,\infty}(\omega)$ is convergent or divergent, it is easily seen, provided that $\alpha_k\rightarrow 0$
and $\beta_k\rightarrow 0$ as $k\rightarrow\infty$, that (\ref{eq:postmean_p_k_3}) converges to $1$ (respectively, $0$) 
if and only if $S_{1,\infty}(\omega)$ is convergent (respectively, divergent). 

Thus, characterization of convergence or divergence of infinite series is possible even with the non-recursive approach.
Indeed, note that the prior parameters $\alpha_k$ and $\beta_k$ are more flexible compared to those
associated with the recursive approach. This is because, in the non-recursive approach 
we only require $\alpha_k\rightarrow 0$ and $\beta_k\rightarrow 0$ as 
$k\rightarrow\infty$, so that convergence of the series $\sum_{j=1}^{\infty}\alpha_j$ and $\sum_{j=1}^{\infty}\beta_j$
are not necessary, unlike the recursive approach. However, choosing $\alpha_k$ and $\beta_k$  to be of sufficiently
small order ensures much faster convergence of the posterior mean and variance as compared to the recursive approach.

Unfortunately, an important drawback of the non-recursive approach is that it does not admit extension
to the case of general oscillatory series with multiple limit points, where blocks of partial sums can not be used; see Section
S-3 of the supplement. On the other hand, as we show in Section S-3 of the supplement, the principles of our recursive theory   
can be easily adopted to develop a Bayesian characterization of oscillating series, which also includes
the characterization of non-oscillating series as a special case. In other words, the recursive
approach seems to be more powerful from the perspective of development of a general characterization theory. 
Moreover, as our examples on convergent and divergent series demonstrate, the recursive posteriors converge 
sufficiently fast to the correct degenerate distributions, obviating the need to consider the non-recursive approach.
Consequently, we do not further pursue the non-recursive approach in this article but reserve the topic for further
investigation in the future.


\begin{remark}
An important issue associated with our characterization results is that the terms $\left\{x_1,x_2,\ldots\right\}$ of the underlying deterministic series of interest $\sum_{i=1}^{\infty}x_i$
is assumed to lie in the complement of the null set. For appropriately specified stochastic processes this need not be difficult to verify. However, for
the sake of sufficient generality we have not assumed any specific form of the underlying stochastic process, which makes the question of null sets
relevant in our case. The solution is that, even if $\left\{x_1,x_2,\ldots\right\}$ falls in some null set, we can still compute a pseudo-posterior distribution of $p_{k,n_k}$
conditional on $\left\{x_1,x_2,\ldots\right\}$,
which has exactly the same form as before. This pseudo-posterior may not admit interpretability as a {\it bona fide} posterior distribution, but characterizes the 
convergence property of $\sum_{i=1}^{\infty}x_i$ in exactly the same way as before. In other words, interestingly and very importantly, all our results of characterization 
presented in our paper hold for {\it all} $\omega\in\mathfrak S$.
\end{remark}

\section{Illustrations}
\label{sec:illustrations}
We now illustrate our ideas with seven examples. These seven examples can be categorized into three
categories in terms of construction of the upper bound $c_{j}$. With the first example
we demonstrate that it may sometimes be easy to devise an appropriate upper bound. In Examples
2 -- 5, we show that usually simple bounds such as that in Example 1, are not adequate in practice,
but appropriate bounds may be constructed if convergence and divergence of the series in question
is known for some values of the parameters; the resultant bounds can be utilized to learn about
convergence or divergence of the series for the remaining values of the parameters. In Examples
6 and 7, the series in question are stand-alone in the sense they are not defined
by parameters with known convergence/divergence for some of their values which might have aided our
construction of $c_{j}$. However, we show that these series can be embedded into appropriately
parameterized series, facilitating similar analysis as Examples 2 -- 5.

For these examples, we consider $n_j=n$ for $j=1,\ldots,K$, with $n=10^6$ and $K=10^5$. 
Since $n$ seems to be sufficiently large, in the case of divergence we expect $|S_{j,n}|$ to exceed 
the monotonically decreasing $c_j$ for all $j\geq j_0$, for sufficiently large $j_0$. Our experiments
demonstrate that this is indeed the case. For further justification we conducted some experiments with larger values
of $n$, but the results remained unchanged. Hence, for relative computational ease we set $n=10^6$ for
the illustrations in this work.

Since we needed to sum $10^6$ terms at each step of $10^5$ stages, 
the associated computation is extremely demanding. For the purpose of efficiency, we parallelized
the computation of the sums of $10^6$ terms, splitting the job on many processors, using the
Message Passing Interface (MPI) protocol. In more details, 
we implemented our parallelized codes, written in C, in VMware consisting of 60 double-threaded, 64-bit physical cores, 
each running at 2793.269 MHz. Parallel computation of our methods 
associated with Examples 1 to 5 take, respectively, 1 minute, 4 minutes, 7 minutes, 6 minutes, and 9 minutes. 
Examples 6 and 7 require about 6 minutes and 4 minutes of computational time.

For space issues we present our applications to the first four examples here; the applications of the
remaining examples are provided in Section S-2 of the supplement.

\subsection{Example 1}
\label{subsec:example1}
In their first example \ctn{Bou12} study the following divergent series with their methods:
\begin{equation}
S=\sum_{i=2}^{\infty}\frac{1}{\log(i)}.
\label{eq:example1}
\end{equation}
We test our Bayesian idea on this series choosing the monotonically decreasing sequence as $c_{j,n}=1/\sqrt{nj}$,
where we represent $c_j$ as $c_{j,n}$ to reflect dependence on $n$. 
Figure \ref{fig:example1}, a plot of the posterior means of $\left\{p_{k,n};k=1,\ldots,10^5\right\}$,  
clearly and correctly indicates that the series is divergent.
We also constructed approximate 95\% highest posterior density credible intervals at each recursive step; however,
thanks to very less variances at each stage, the intervals turned out to be too small to be clearly distinguishable
from the plot of the stage-wise posterior means. 
\begin{figure}
\centering
\includegraphics[width=6cm,height=5cm]{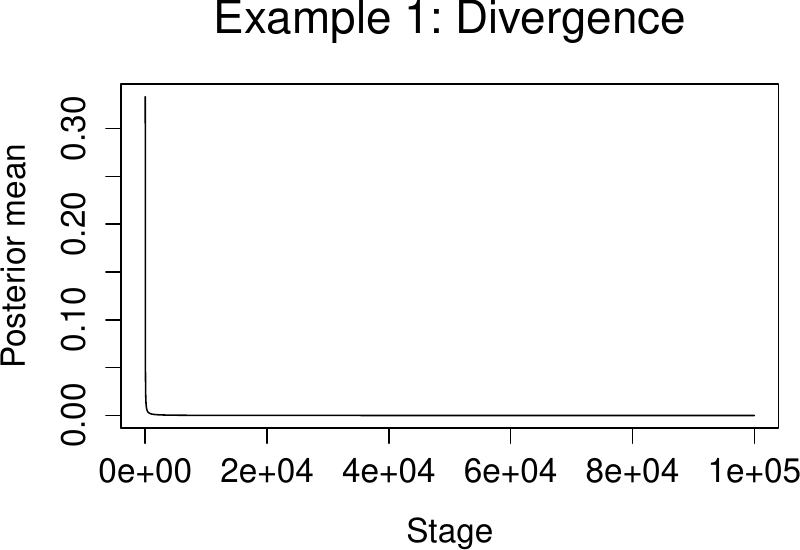}
\caption{Example 1: The series (\ref{eq:example1}) is divergent.}
\label{fig:example1}
\end{figure}

\subsection{Example 2}
\label{subsec:example2}
Example 2 of \ctn{Bou12} deals with the following series:
\begin{equation}
S^a=\sum_{i=2}^{\infty}\left(1-\left\{\frac{\log(i)}{i}\right\}-a\frac{\log\log(i)}{i}\right)^i,
\label{eq:example2}
\end{equation}
where $a\in\mathbb R$. \ctn{Bou12} prove that the series converges for $a>1$ and diverges for $a\leq 1$.

\subsubsection{Choice of $c_{j,n}$}
\label{subsubsec:example2_c}
Now, however, selecting the monotone sequence as $c_{j,n}=1/\sqrt{nj}$ turn out to be inappropriate
for this series, the behaviour of which is quite sensitive to the parameter $a$, particularly around $a=1$. 
Hence, any appropriate sequence $\left\{c_{j,n}\right\}_{j=1}^{\infty}$ must depend on the parameter $a$
of the series (\ref{eq:example2}). 

Denoting $c_{j,n}$ by $c^a_{j,n}$ to reflect the dependence on $a$
as well, we first set 
\begin{equation}
u^a_{j,n}=S^{a_0}_{j,n}+\frac{(a-1-9\times 10^{-11})}{\log(j+1)}, 
\label{eq:u_example2}
\end{equation}
and then let 
\begin{equation}
c^a_{j,n}=\left\{\begin{array}{ccc} u^a_{j,n}, &\mbox{if}~u^a_{j,n}>0;\\
S^{a_0}_{j,n}, & \mbox{otherwise.}\end{array}\right.
\label{eq:example2_c}
\end{equation}
where $a_0=1+10^{-10}$. The reason behind such a choice of $c^a_{j,n}$ is provided below. 

Let, for $\epsilon>0$, 
\begin{equation}
\tilde S=\sup\left\{S^a:a\geq 1+\epsilon\right\}.
\label{eq:tilde_S_example2}
\end{equation}
Thus, $\tilde S$ may be interpreted as the convergent series which is closest to divergence
given the convergence criterion $a\geq 1+\epsilon$.
Since $S^a$ is decreasing in $a$, it easily follows that equality of (\ref{eq:tilde_S_example2})
is attained at $a_0=1+\epsilon$.


Since the terms of the series $S^a$ are decreasing in $i$, it follows that $S^{a_0}_{j,n}$ in (\ref{eq:example2_c})
is decreasing in $j$.
We assume that $\epsilon$ is chosen to be so small that convergence properties
of the series for $\left\{a\leq 1\right\}\cup\left\{a\geq 1+\epsilon\right\}$ are only desired. 
Indeed, since $\left(1-\left\{\frac{\log(i)}{i}\right\}-a\frac{\log\log(i)}{i}\right)^i$ is decreasing in $a$
for any given $i\geq 3$, our method of constructing $c^a_{j,n}$ need not
be able to correctly identify the convergence properties of the series for $1<a<1+\epsilon$.

For the purpose of illustrations we choose $\epsilon=10^{-10}$.
Note that for $a>1$ the term $\frac{(a-1-9\times 10^{-11})}{\log(j+1)}$ inflates $c^a_{j,n}$
making $S^a_{j,n}$ more likely to fall below $c^a_{j,n}$ for increasing $a$, thus paving the way for  
diagnosing convergence. The same term also 
ensures that for $a\leq 1$, $c^a_{j,n}<S^{a_0}_{j,n}$, so that
$S^a_{j,n}$ is likely to exceed $c^a_{j,n}$, thus providing an inclination towards divergence. The term 
$-9\times 10^{-11}$ is an adjustment for the case $a=1+10^{-10}$, ensuring that $c^a_{j,n}$ marginally exceeds $S^a_{j,n}$
to ensure convergence.
The scaling factor $\log(j+1)$ ensures that the part $\frac{(a-1-9\times 10^{-11})}{\log(j+1)}$ of 
(\ref{eq:example2_c}) tends to zero at a slow rate so that $c^a_{j,n}$ is decreasing with $j$ and $n$ even if
$a-1-9\times 10^{-11}$ is negative.

Figure \ref{fig:example2}, depicting our Bayesian results for this series, is in agreement
with the results of \ctn{Bou12}. In fact, we have applied our methods to many more values of $a\in A_{\epsilon}$
with $\epsilon=10^{-10}$, and in every case the correct result is vindicated.
\begin{figure}
\centering
\subfigure [Divergence: $a = 1-10^{-10}$.]{ \label{fig:example1_a_less_1}
\includegraphics[width=6cm,height=5cm]{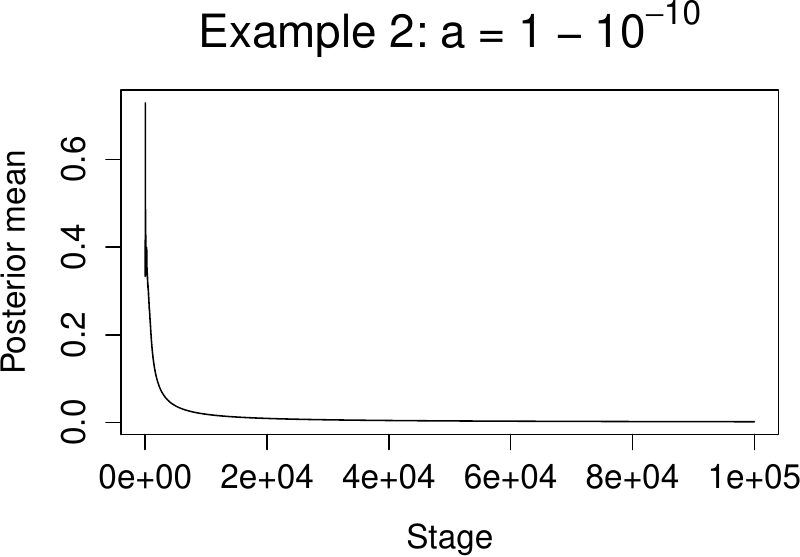}}
\hspace{2mm}
\subfigure [Divergence: $a=1$.]{ \label{fig:example2_a_1}
\includegraphics[width=6cm,height=5cm]{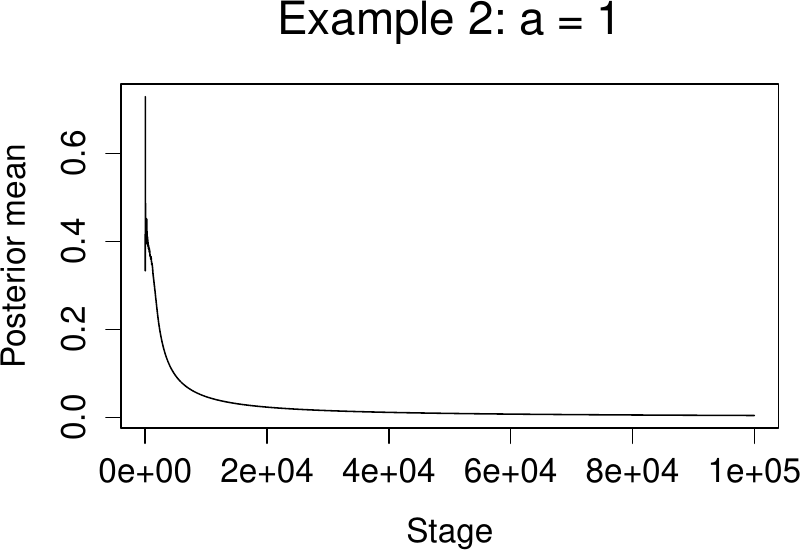}}\\
\subfigure [Convergence: $a=1+10^{-10}$.]{ \label{fig:example2_a_greater_1}
\includegraphics[width=6cm,height=5cm]{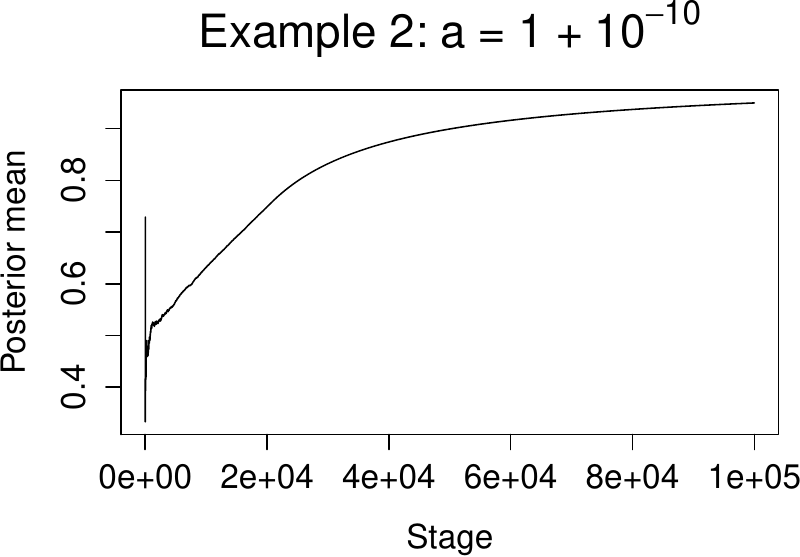}}
\hspace{2mm}
\subfigure [Convergence: $a=1+20^{-10}$.]{ \label{fig:example2_a_greater_1_2}
\includegraphics[width=6cm,height=5cm]{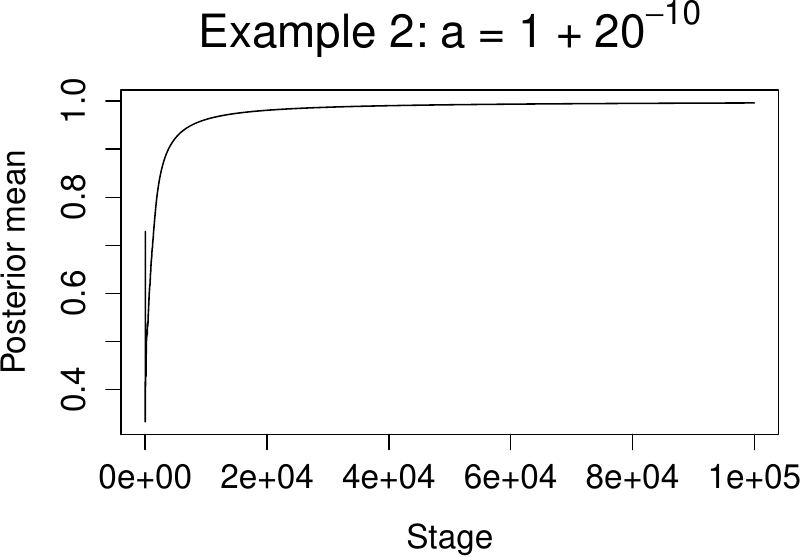}}\\
\subfigure [Divergence: $a=-1$.]{ \label{fig:example2_a_minus_1}
\includegraphics[width=6cm,height=5cm]{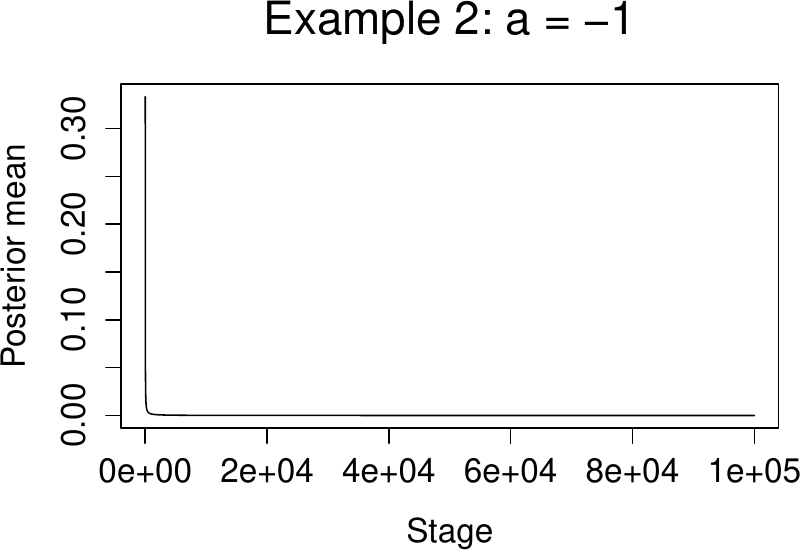}}\\
\caption{Example 2: The series (\ref{eq:example2}) converges for $a>1$ and diverges
for $a\leq 1$.}
\label{fig:example2}
\end{figure}

\subsection{Example 3}
\label{subsec:example3}
Let us now consider the following series analysed by \ctn{Bou12}:
\begin{equation}
S=\sum_{i=3}^{\infty}\left(1-\left(\frac{\log(i)}{i}\right)a^{\frac{\log\log(i)}{\log(i)}}\right)^i,
\label{eq:example3}
\end{equation}
where $a>0$. As is shown by \ctn{Bou12}, the series converges for $a>e$
and diverges for $a\leq e$.

\subsubsection{Choice of $c_{j,n}$}
\label{subsubsec:example3_c}

Here we first set 
\begin{equation}
u^a_{j,n}=S^{a_0}_{j,n}+\frac{(a-e-9\times 10^{-11})}{\log(j+1)}, 
\label{eq:u_example3}
\end{equation}
and then let $c^a_{j,n}$ defined by (\ref{eq:example2_c}).
Again, it is easily seen that $S^{a_0}_{j,n}$ is decreasing in $j$.
In this example we set $a_0=e+10^{-10}$. The rationale behind the choice remains the same as detailed in 
Section \ref{subsubsec:example2_c}. 

As before, the results obtained by our Bayesian theory, as displayed in Figure \ref{fig:example3}, are
in complete agreement with the
results obtained by \ctn{Bou12}.
\begin{figure}
\centering
\subfigure [Divergence: $a=e-10^{-10}$.]{ \label{fig:example3_e_less_1}
\includegraphics[width=6cm,height=5cm]{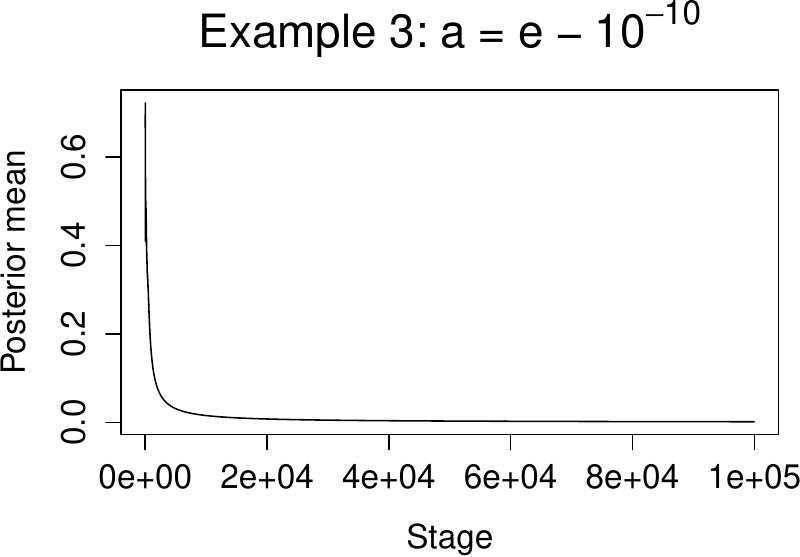}}
\hspace{2mm}
\subfigure [Divergence: $a=e$.]{ \label{fig:example3_a_e}
\includegraphics[width=6cm,height=5cm]{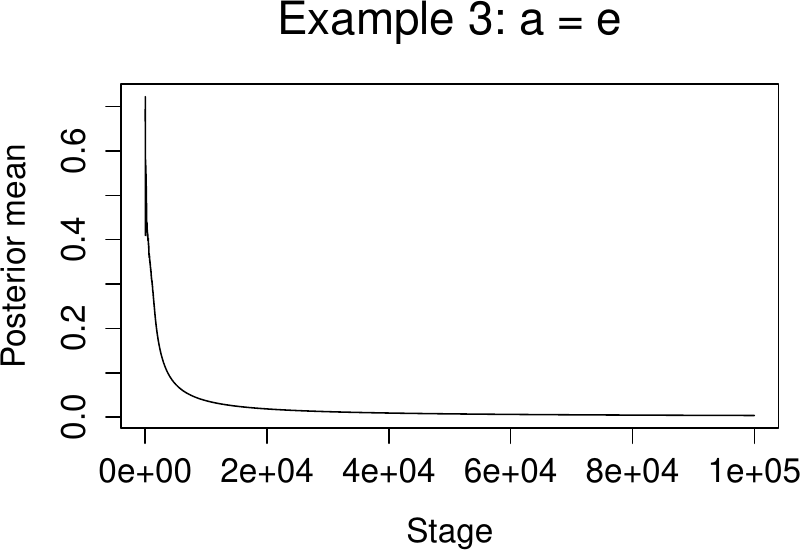}}\\
\subfigure [Convergence: $a=e+10^{-10}$.]{ \label{fig:example3_a_greater_e}
\includegraphics[width=6cm,height=5cm]{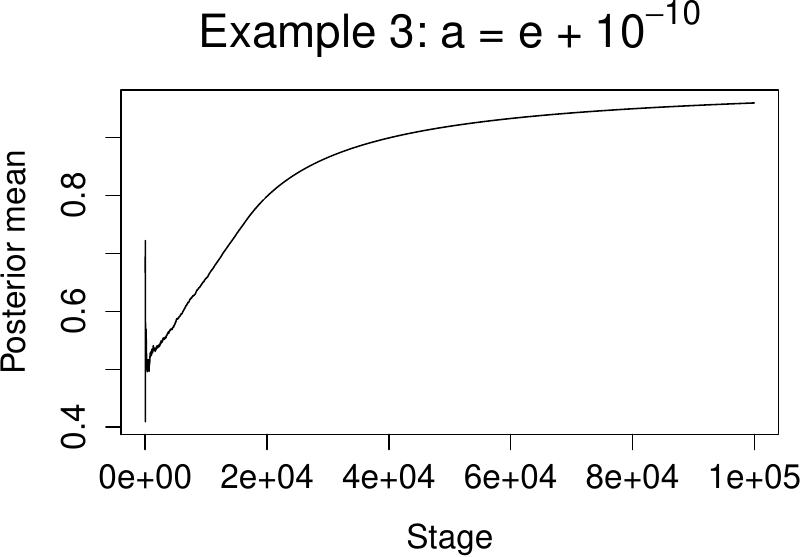}}
\subfigure [Convergence: $a=e+20^{-10}$.]{ \label{fig:example3_a_e2}
\includegraphics[width=6cm,height=5cm]{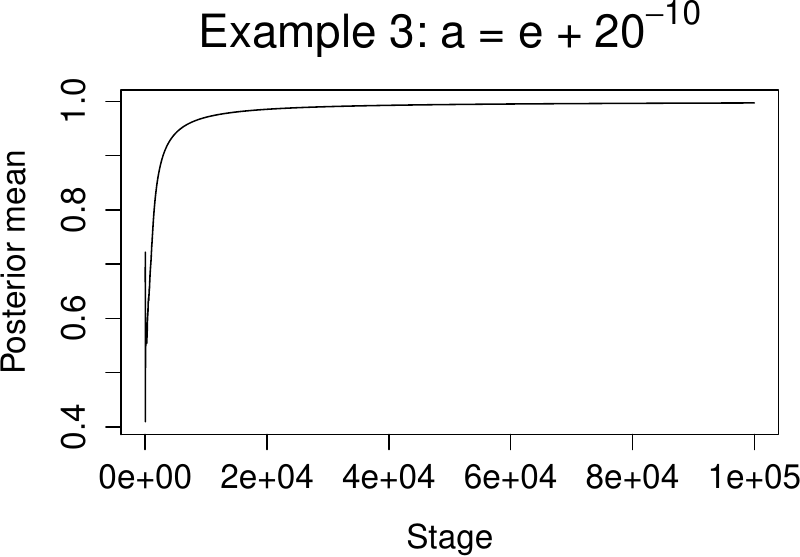}}
\caption{Example 3: The series (\ref{eq:example3}) converges for $a>e$ and diverges
for $a\leq e$.}
\label{fig:example3}
\end{figure}

\subsection{Example 4}
\label{subsec:example4}
We now consider series (\ref{eq:inconclusive1}). It has been proved by
\ctn{Bou12} that the series is convergent for $a-b>1$ and divergent for $a+b<1$.
As mentioned before, the hierarchy of tests of \ctn{Bou12} are
inconclusive for $a=b=1$. 

In this example we denote the partial sums by $S^{a,b}_{j,n}$ and the actual series $S$ by
$S^{a,b}$ to reflect the dependence on both
the parameters $a$ and $b$. 
\begin{equation}
S^{a,b}_{j,n}=\sum_{i=3+n(j-1)}^{3+nj-1}\left(1-\frac{\log i}{i}-\frac{\log\log i}{i}
\left\{\cos^2\left(\frac{1}{i}\right)\right\}\left(a+(-1)^ib\right)\right)^i,
\label{eq:S_example4}
\end{equation}
We then have the following lemma, the proof of which is
presented in Section S-1 of the supplement. 
\begin{lemma}
\label{lemma:example4}
For series (\ref{eq:inconclusive1}), for $j\geq 1$ and $n$ even, 
$S^{a,b}_{j,n}$ given by (\ref{eq:S_example4}) is decreasing in $a$ but increasing in $b$.
\end{lemma}

Since $S^{a,b}$ is just summation of the partial sums, it follows that
\begin{corollary}
\label{corollary:example4}
$S^{a,b}$ is decreasing in $a$ and increasing in $b$.
\end{corollary}

We let 
\begin{equation}
A_{\epsilon}=\left\{a:0\leq a\leq 1\right\}\cup\left\{a:a\geq 1+\epsilon\right\},
\label{eq:A_example4}
\end{equation}
and
\begin{equation}
\tilde S=\underset{a\in A_{\epsilon}}{\inf}\underset{b\geq 0}{\sup}~\left\{S^{a,b}:a-b>1\right\}.
\label{eq:tilde_S}
\end{equation}
It is easy to see in this case, due to Corollary \ref{corollary:example4} and the convergence
criterion $a-b>1$, that $\tilde S$ is attained at $a_0=1+\epsilon$ and $b_0=0$.
As before, we set $\epsilon=10^{-10}$.
Hence, arguments similar to those in Section \ref{subsubsec:example2_c} lead to the following choice of the
upper bound for $S^{a,b}_{j,n}$, which we denote in this example by $c^{a,b}_{j,n}$:
\begin{equation}
c^{a,b}_{j,n}=\left\{\begin{array}{ccc} u^{a,b}_{j,n}, &\mbox{if}~u^{a,b}_{j,n}>0;\\
S^{a_0,b_0}_{j,n}, & \mbox{otherwise},\end{array}\right.
\label{eq:example4_c}
\end{equation}
where $a_0=1+10^{-10}$, $b_0=0$, and 
\begin{equation}
u^{a,b}_{j,n}=S^{a_0,b_0}_{j,n}+\frac{(a-1-b-9\times 10^{-11})}{\log(j+1)}. 
\label{eq:u_example4}
\end{equation}
As before, it is easily seen that $S^{a_0,b_0}_{j,n}$ is decreasing in $j$. Also note that $-b$ 
in (\ref{eq:u_example4}) takes account of the fact that the partial sums are increasing in $b$,
thus favouring divergence for increasing $b$.

Setting aside panel (c) of Figure \ref{fig:example4b}, observe that the remaining panels of
Figures \ref{fig:example4a} and \ref{fig:example4b} are in agreement with the results of \ctn{Bou12}, 
but in the case $a=b=1$, the tests of \ctn{Bou12} turned out to be 
inconclusive. Panel (c) of Figure \ref{fig:example4b} demonstrates that the series is divergent for $a=b=1$.
\begin{figure}
\centering
\subfigure [Convergence: $a=3,b=1$.]{ \label{fig:example4_a_3_b_1}
\includegraphics[width=6cm,height=5cm]{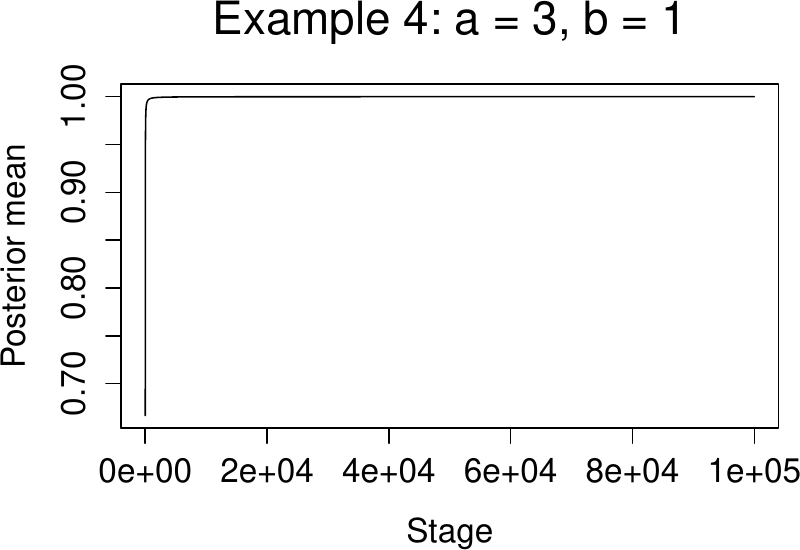}}
\hspace{2mm}
\subfigure [Convergence: $a=1+10^{-10},b=0$.]{ \label{fig:example4_a11_b00}
\includegraphics[width=6cm,height=5cm]{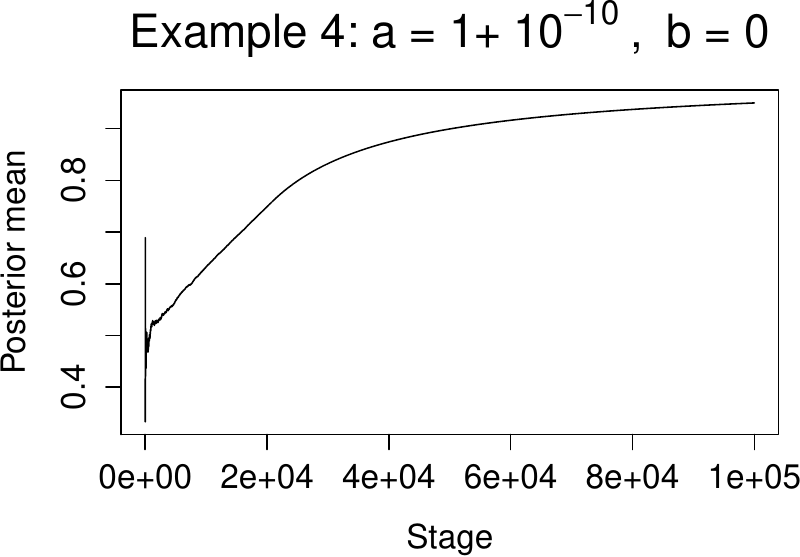}}\\
\subfigure [Convergence: $a=1+20^{-10},b=10^{-10}$.]{ \label{fig:example4_a_12_b_01}
\includegraphics[width=6cm,height=5cm]{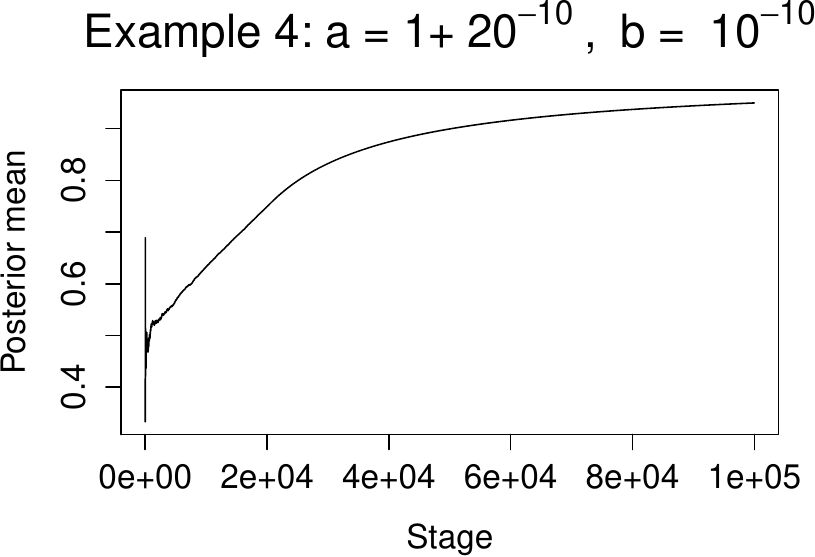}}
\hspace{2mm}
\subfigure [Divergence: $a=1/2,b=1/3$.]{ \label{fig:example4_a_1_2_b_1_3}
\includegraphics[width=6cm,height=5cm]{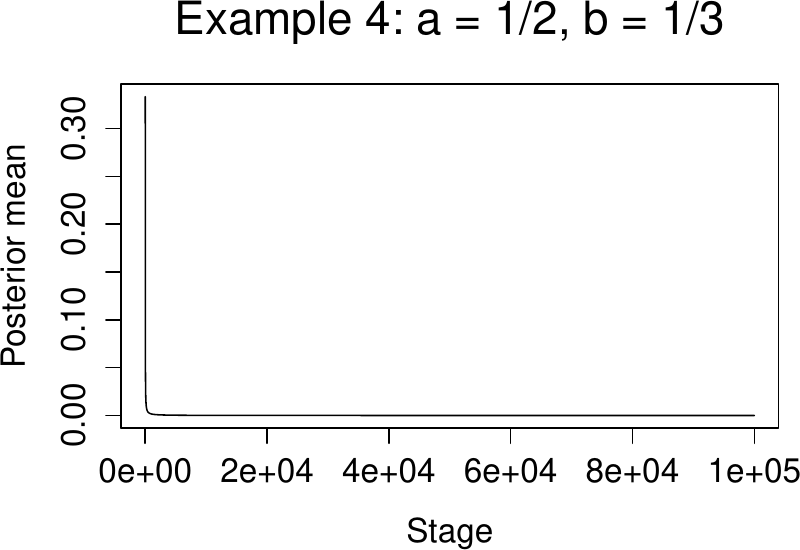}}
\caption{Example 4: The series (\ref{eq:inconclusive1}) converges for $(a=3,b=1)$, $\left(a=1+10^{-10},b=0\right)$,
$\left(a=1+20^{-10},b=10^{-10}\right)$ and diverges for $\left(a=1/2,b=1/3\right)$.}
\label{fig:example4a}
\end{figure}

\begin{figure}
\centering
\subfigure [Divergence: $a=\frac{1}{2}\left(1-10^{-11}\right),
b=\frac{1}{2}\left(1-10^{-11}\right)$.]{ \label{fig:example4_a+b_less_1}
\includegraphics[width=6cm,height=5cm]{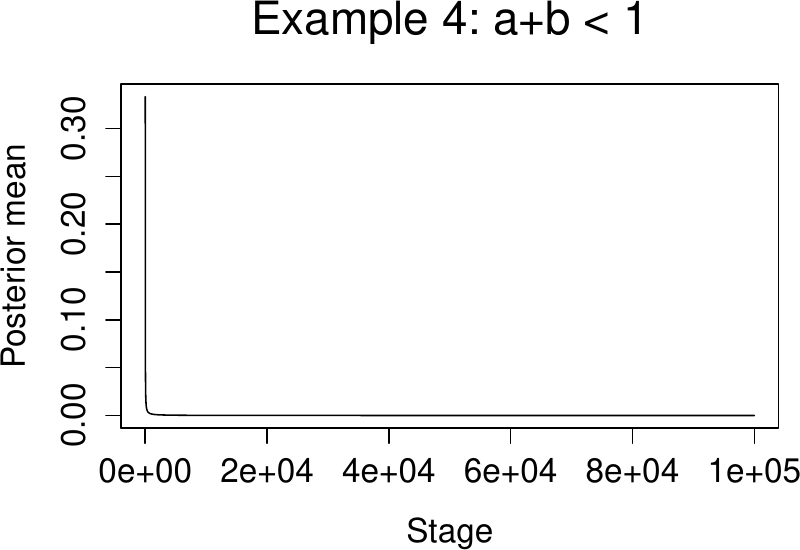}}
\hspace{2mm}
\subfigure [Divergence: $a=1,b=0$.]{ \label{fig:example4_a_1_b_0}
\includegraphics[width=6cm,height=5cm]{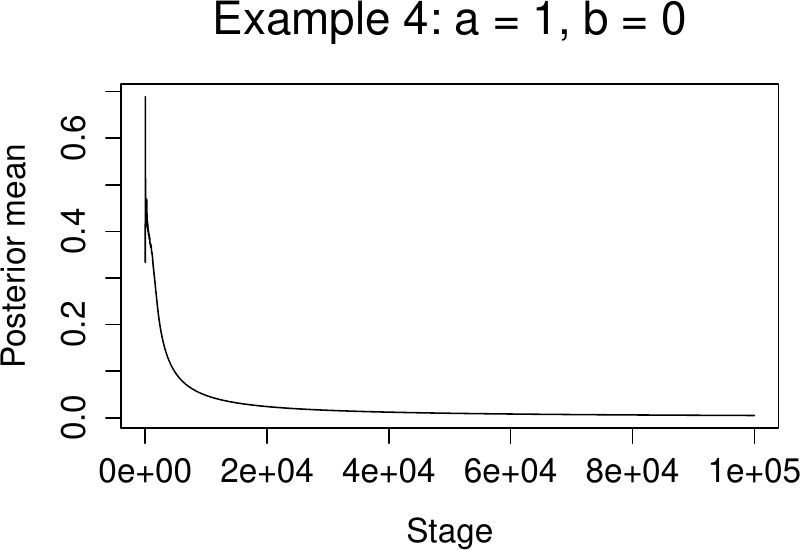}}
\hspace{2mm}
\subfigure [Divergence: $a=1,b=1$.]{ \label{fig:example4_a_1_b_1}
\includegraphics[width=6cm,height=5cm]{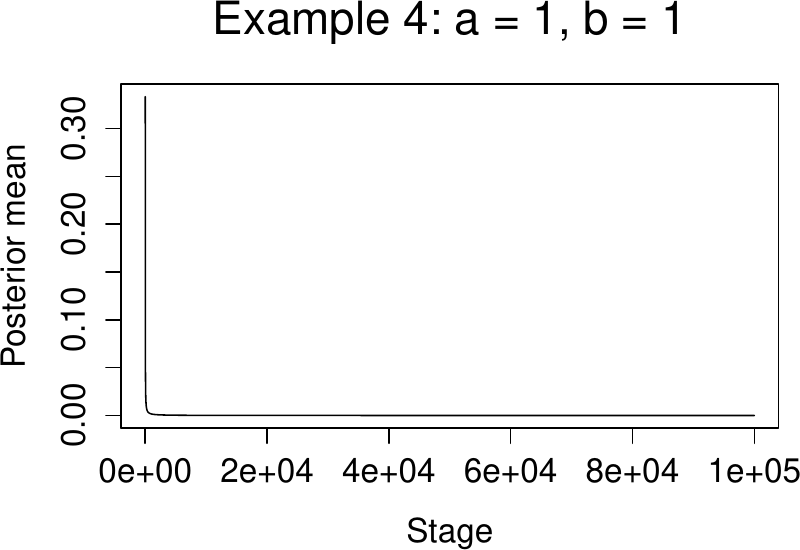}}
\caption{Example 4: The series (\ref{eq:inconclusive1}) diverges for 
$\left(a=\frac{1}{2}\left(1-10^{-11}\right),b=\frac{1}{2}\left(1-10^{-11}\right)\right)$, 
$(a=1,b=0)$ and $(a=1,b=1)$.}
\label{fig:example4b}
\end{figure}

\section{Application to Riemann Hypothesis}
\label{sec:RH}
\subsection{Brief background}
\label{subsec:brief_background}
Consider the Riemann zeta function given by
\begin{equation}
\zeta(a)=\frac{1}{1-2^{1-a}}\sum_{n=0}^{\infty}\frac{1}{2^{n+1}}\sum_{k=0}^n\left(-1\right)^k
\frac{n!}{k!(n-k)!}(k+1)^{-a},
\label{eq:Riemann_zeta}
\end{equation}
where $a$ is complex.
The above function is formed by first considering Euler's function
\begin{equation}
Z(a)=\sum_{n=1}^{\infty}\frac{1}{n^a},
\label{eq:Euler}
\end{equation}
then by multiplying both sides of (\ref{eq:Euler}) by $\left(1-\frac{2}{2^a}\right)$ to obtain
\begin{equation}
\left(1-\frac{2}{2^a}\right)Z(a)=\sum_{n=1}^{\infty}\frac{\left(-1\right)^{n+1}}{n^a},
\label{eq:Euler2}
\end{equation}
and then dividing the right hand side of (\ref{eq:Euler2}) by $\left(1-\frac{2}{2^a}\right)$.
The advantage of the function $\zeta(a)$ in comparison with the parent function $Z(a)$ is that,
$Z(a)$ is divergent if the real part of $a$, which we denote by $Re(a)$, is less than or equal to $1$,
while $\zeta(a)$ is convergent for all $a$ with $Re(a)>0$. Importantly, $\zeta(a)=Z(a)$ whenever $Z(a)$
is convergent.

Whenever $0<Re(a)<1$, $\zeta(a)$ satisfies the following identity:
\begin{equation}
\zeta(a)=2^a\pi^{a-1}\sin\left(\frac{\pi a}{2}\right)\Gamma(1-a)\zeta(1-a),
\label{eq:zeta_identity}
\end{equation}
where $\Gamma(\cdot)$ is the gamma function.
This can be extended to the set of complex numbers by defining a function with non-positive real part
by the right hand side of (\ref{eq:zeta_identity}); abusing notation, we denote the new function by
$\zeta(a)$. Because of the sine function, it follows that  
the trivial zeros of the above function occur when the values of $a$ are negative even integers.
Hence, the non-trivial zeros must satisfy $0<Re(a)<1$.

\ctn{Riemann1859} conjectured that all the non-trivial zeros have the real part $1/2$, which is
the famous Riemann Hypothesis.
For accessible account of the Riemann Hypothesis, see \ctn{Borwein06}, \ctn{Derbyshire04}.

One equivalent condition for the Riemann Hypothesis is related to sums of 
of the M\"{o}bius function, given by
\begin{equation}
\mu(n)=\left\{
\begin{array}{ccc}-1 & \mbox{if} & n~\mbox{is a square-free positive integer with an odd number of prime factors};\\
0 & \mbox{if} & n~\mbox{has a squared prime factor};\\
1 & \mbox{if} & n~\mbox{is a square-free positive integer with an even number of prime factors},
\end{array}
\right.
\label{eq:mobius2}
\end{equation}
where, by square-free integer we mean that the integer is not divisible by any perfect square other than $1$.
Specifically, the condition 
\begin{equation}
\sum_{n=1}^x\mu(n)=O\left(x^{\frac{1}{2}+\epsilon}\right)
\label{eq:Merten}
\end{equation}
for any $\epsilon>0$,
is equivalent to Riemann Hypothesis. This condition implies that the Dirichlet series for the M\"{o}bius function,
given by
\begin{equation}
M(a)=\sum_{n=1}^{\infty}\frac{\mu(n)}{n^a}=\frac{1}{\zeta(a)},
\label{eq:mobius}
\end{equation}
is analytic in $Re(a)>1/2$. This again ensures that $\zeta(a)$ is meromorphic in $Re(a)>1/2$ 
and that it has no zeros in this region. Using the functional equation (\ref{eq:zeta_identity}) it follows
that there are no zeros of $\zeta(a)$ in $0<Re(a)<1/2$ either. Hence, (\ref{eq:Merten}) implies
Riemann Hypothesis. The converse is also certainly true. 

The above arguments also imply that convergence of $M(a)$ in (\ref{eq:mobius}) 
for $Re(a)>1/2$ is equivalent to Riemann Hypothesis, and it is this criterion that is of our interest
in this paper.
Now, $M(a)$ converges absolutely for $Re(a)>1$; 
moreover, $M(1)=0$.
The latter is equivalent to the prime number
theorem stating that the number of primes below $x$ is asymptotically $x/\log(x)$, as $x\rightarrow\infty$
(\ctn{Landau06}).
Thus, $M(a)$ converges for $Re(a)\geq 1$. That $M(a)$ diverges for $Re(a)\leq 1/2$
can be seen as follows. Note that if $M(a)$ converged for any $a^*$ such that $Re(a^*)\leq 1/2$, then analytic
continuation for Dirichlet series of the form $M(a)$ would guarantee convergence of $M(a)$ for all $a$ with $Re(a)>Re(a^*)$.
But $\zeta(a)$ is not analytic on $0<Re(a)<1$ because of its non-trivial zeros on the strip. This would contradict
the analytic continuation leading to the identity $M(a)=1/\zeta(a)$ on the entire set of complex numbers.
Hence, $M(a)$ must be divergent for $Re(a)\leq 1/2$.

In this paper, we apply our ideas to particularly investigate convergence of $M(a)$ when $1/2<a<1$.

\subsection{Choice of the upper bound and implementation details}
\label{subsec:upper_bound_RH}
To form an idea of the upper bound we first plot the partial sums $S^a_{j,n}$, for
$j=1000$ and $n=10^6$, with respect to $a$. In this regard, panel (a) of Figure \ref{fig:plot_partial_sums}
shows the decreasing nature of the partial sums with respect to $a$, and panel (b) magnifies the plot
in the domain $1/2<a<1$ that we are particularly interested in. The latter shows that the partial sums
decrease sharply till about $0.7$, getting appreciably close to zero around that point,
after which the rate of decrease diminishes. Thus, one may expect a change point around $0.7$ regarding convergence.
Specifically, divergence may be expected below a point slightly larger than $0.7$ and convergence above it.
\begin{figure}
\centering
\subfigure [Plot of partial sums in the domain $(0,5)$.]{ \label{fig:plot_partial_sums_RH}
\includegraphics[width=6cm,height=5cm]{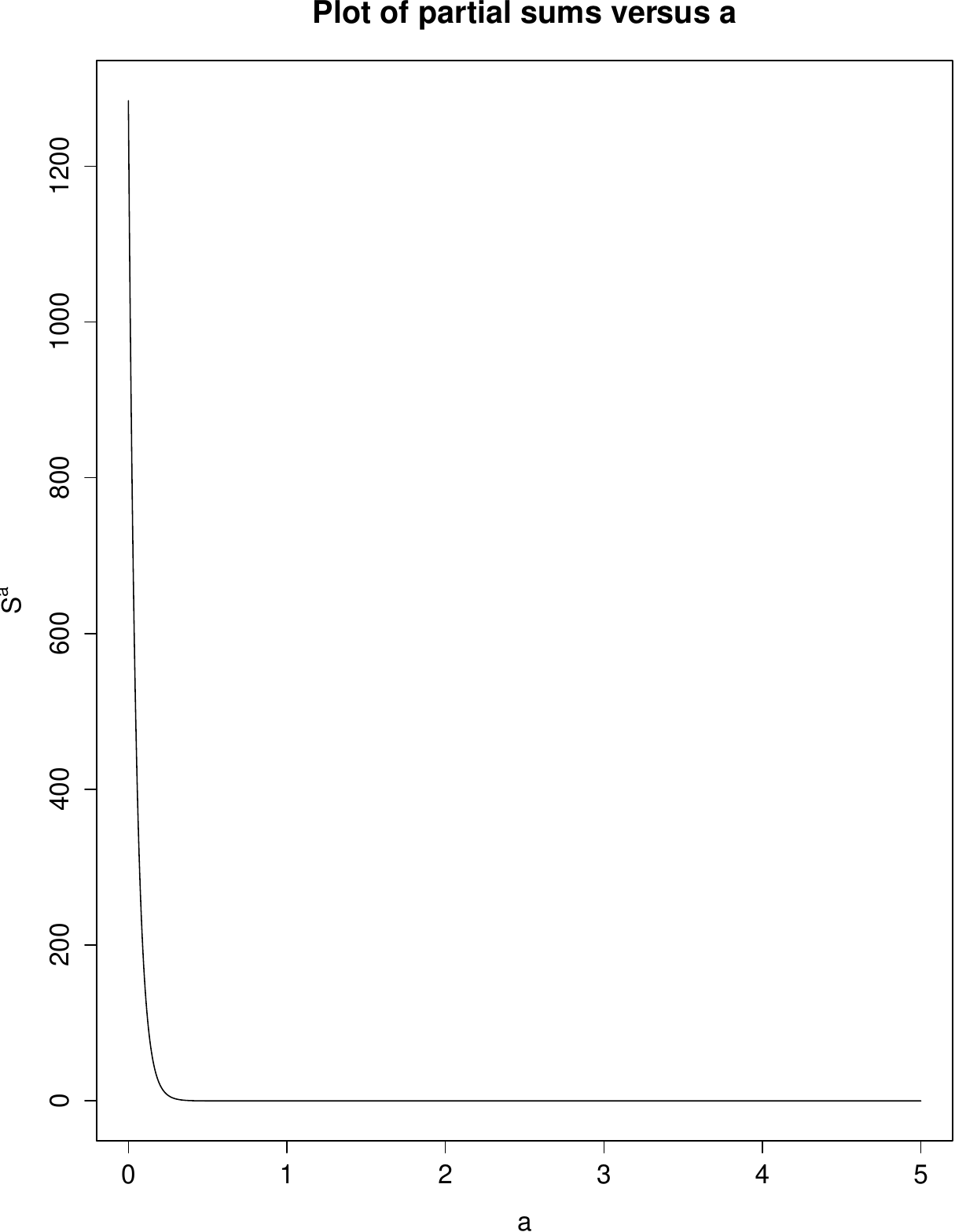}}
\hspace{2mm}
\subfigure [Plot of partial sums in the domain $(0.5,1)$.]{ \label{fig:plot_partial_sums_RH_magnify}
\includegraphics[width=6cm,height=5cm]{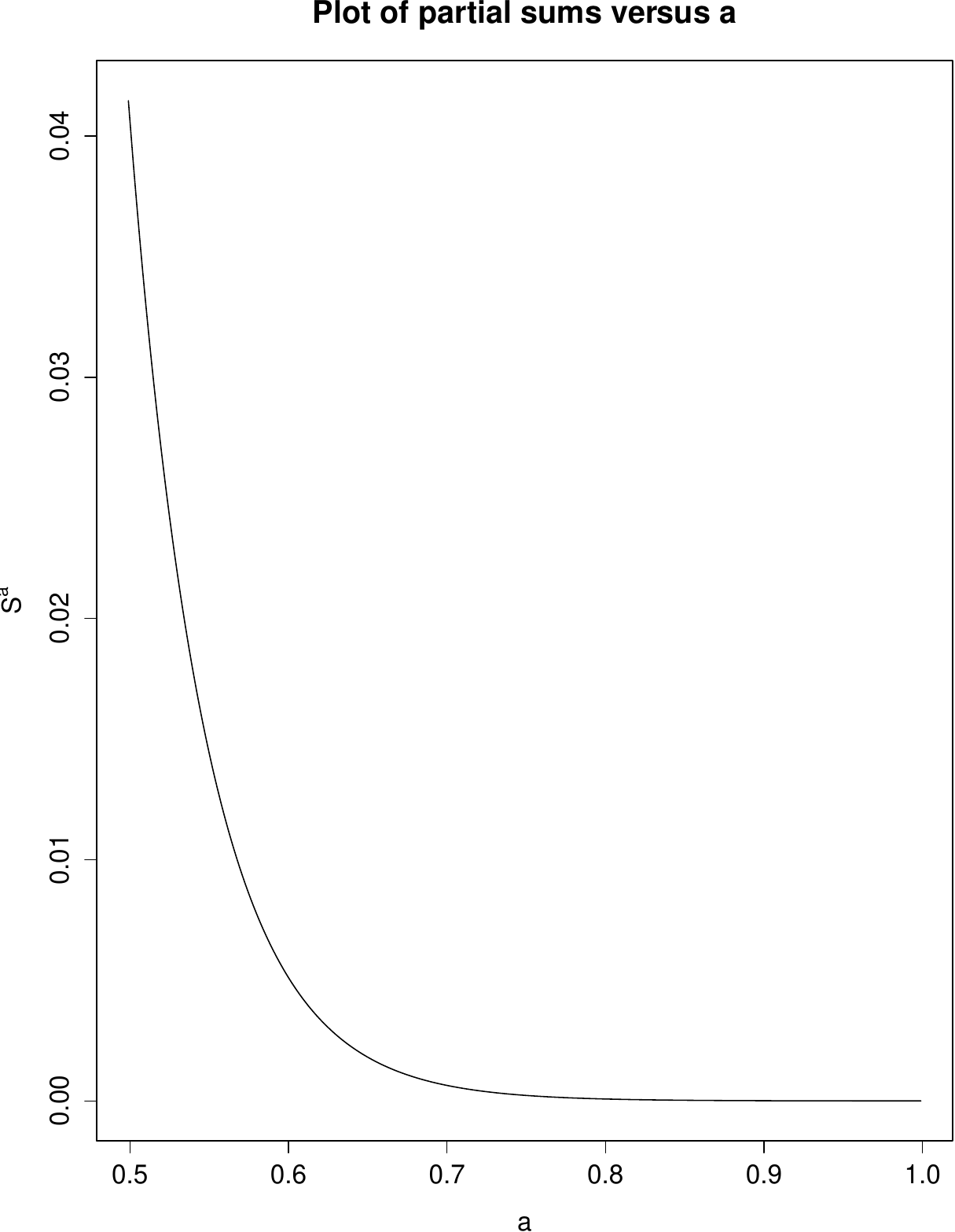}}
\caption{Plot of the partial sums $S^a_{1000,1000000}$ versus $a$. Panel (a) shows the plot in the domain
$[0,5]$ while panel (b) magnifies the same in the domain $(0.5,1)$.}
\label{fig:plot_partial_sums}
\end{figure}

Since $M(1)<\infty$, we consider this series as the basis for our upper bound,
with the value of $a$ also taken into account.
Specifically, we choose the upper bound as
\begin{equation}
c_{j,n}=\left|S^1_{j,n}+\frac{a}{j+1}\right|.
\label{eq:upper_bound_RH}
\end{equation}
Since Figure \ref{fig:plot_partial_sums} shows that the partial sums are of monotonically decreasing nature, 
the above choice of upper bound facilitates detection of convergence for relatively large values of $a$.
The part $\frac{a}{j+1}$, which tends to zero as $j\rightarrow\infty$, 
takes care of the fact that the series may be convergent if $a<1$, by slightly inflating $S^1_{j,n}$.

For our purpose, we compute the first $10^9$ values of the M\"{o}bius function using an efficient algorithm
proposed in \ctn{Lioen94}, which is based on the Sieve of Eratosthenes (\ctn{Eratosthenes}). We set $K=1000$ and $n=10^6$.
A complete analysis with our VMware with our parallel implementation takes about $2$ minutes.

\subsection{Results of our Bayesian analysis}
\label{subsec:results_RH}
Panels (a)--(e) of Figure \ref{fig:RH_1} and panels (d)--(f) of Figure \ref{fig:RH_2} show the $M(a)$ diverges for 
$a=0.1$, $0.2$, $0.3$, $0.4$, $0.5$, but converges for
$a=1+10^{-10}$, $2$ and $3$. 
In fact, for many other values
that we experimented with, $M(a)$ converged for $a>1$ and diverged for $a<1/2$,
demonstrating remarkable consistency with the known, existing results.

Certainly far more important are the results for $1/2<a<1$. Indeed,
panel (f) of Figure \ref{fig:RH_1} and panels (a)--(c) of Figure \ref{fig:RH_2} show that $M(a)$ diverged for
$a=0.6$ and $0.7$ and converged for $a=0.8$ and $0.9$. It thus appears that $M(a)$ diverges
for $a<a^*$ and converges for $a\geq a^*$, for some $a^*\in(0.7,0.8)$. Figure \ref{fig:RH_3} displays
results of our further experiments in this regard. Panels (a) and (b) of Figure \ref{fig:RH_3} show
the posterior means for the full set of iterations and the last $500$ iterations, respectively, for $a=0.71$.
Note that from panel (a), convergence seems to be attained, although towards the end, the plot seems to
be slightly tilted downwards. Panel (b) magnifies this, clearly showing divergence. Panels (c) and (d) of
Figure \ref{fig:RH_3} depict similar phenomenon for $a=0.715$, but as per panel (d), divergence seems to
ensue all of a sudden, even after showing signs of convergence for the major number of iterative stages. 
Convergence of $M(a)$ begins at $a=0.72$ (approximately); panels (e) and (f)
of Figure \ref{fig:RH_3} take clear note of this.

Thus, as per our methods, $M(a)$ diverges for $a<0.72$ and converges for $a\geq 0.72$.
This is remarkably in keeping with the wisdom gained from panel (b) of Figure \ref{fig:plot_partial_sums}
that convergence is expected to occur for values of $a$ exceeding $0.7$.
Note that neither the upper bound (\ref{eq:upper_bound_RH}), nor our methodology, is in any way biased towards
$a\approx 0.7$; hence, our result is perhaps not implausible.  

\subsection{Implications of our result}
\label{subsec:implications}
As per our results, $M(a)$ does not converge for all $a>1/2$, and hence does not completely
support Riemann Hypothesis. However, convergence of $M(a)$ fails only for the relatively small region
$0.5<a<0.72$, which perhaps is the reason why there exists much evidence in favour of Riemann Hypothesis.

\begin{figure}
\centering
\subfigure [Divergence: $a=0.1$.]{ \label{fig:RH_a_01}
\includegraphics[width=6cm,height=5cm]{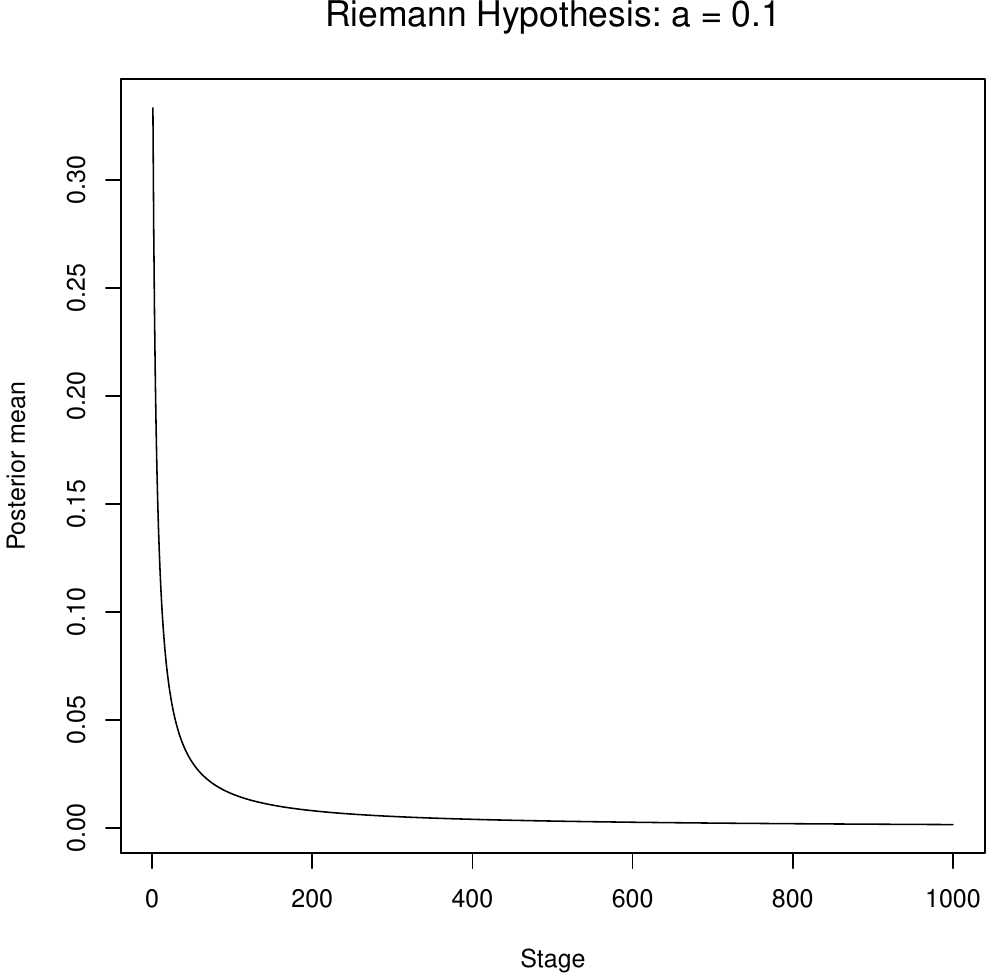}}
\hspace{2mm}
\subfigure [Divergence: $a=0.2$.]{ \label{fig:RH_a_02}
\includegraphics[width=6cm,height=5cm]{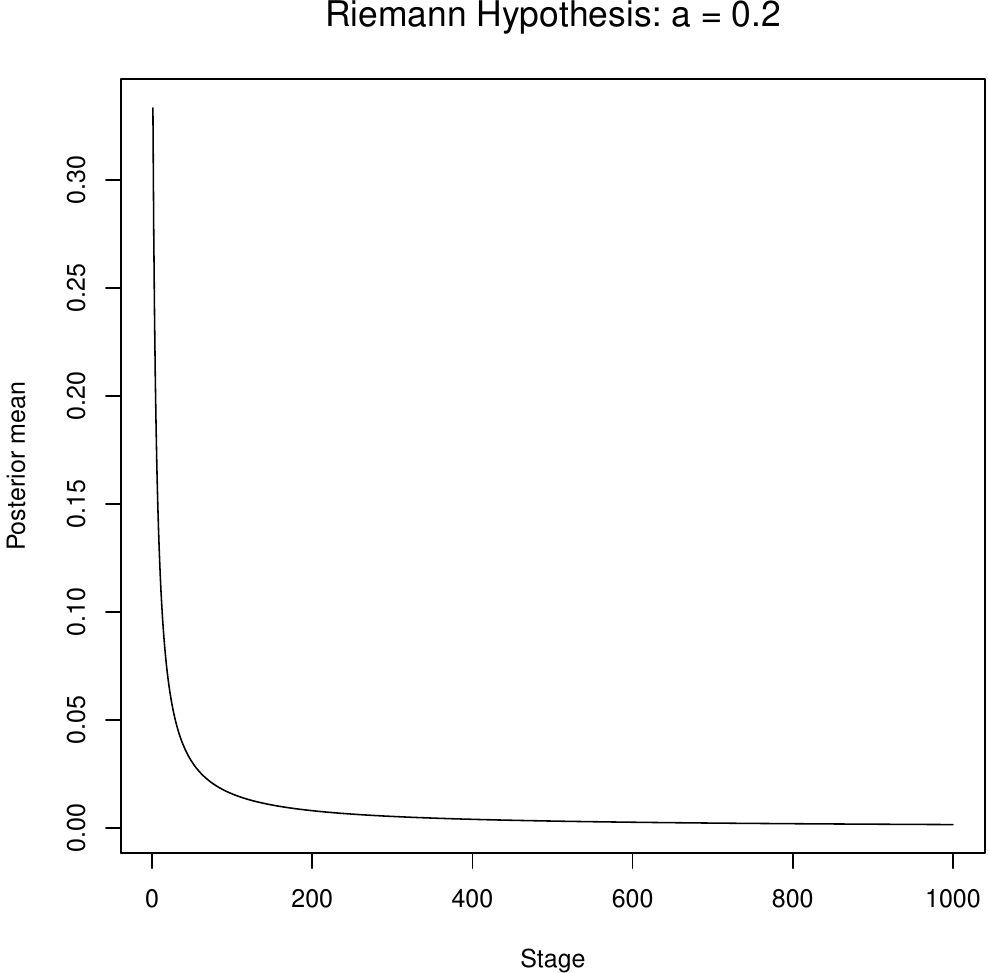}}\\
\subfigure [Divergence: $a=0.3$.]{ \label{fig:RH_a_03}
\includegraphics[width=6cm,height=5cm]{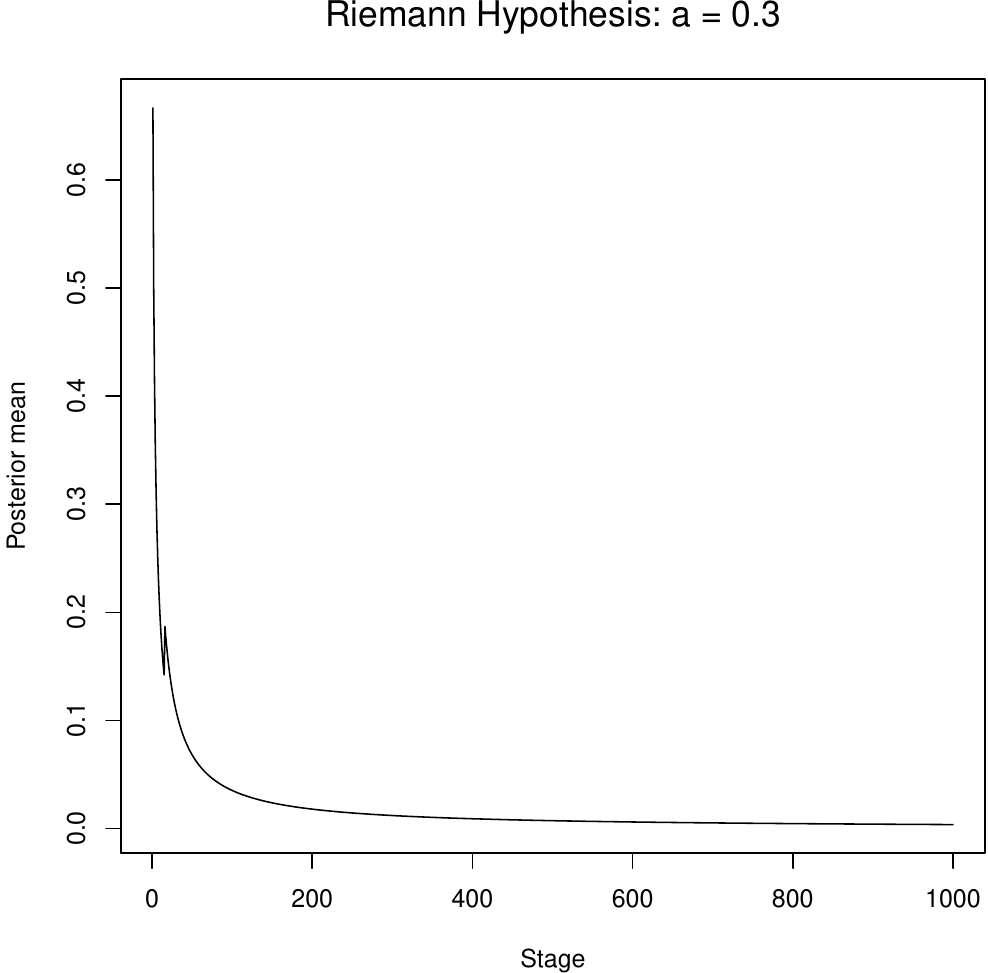}}
\hspace{2mm}
\subfigure [Divergence: $a=0.4$.]{ \label{fig:RH_a_04}
\includegraphics[width=6cm,height=5cm]{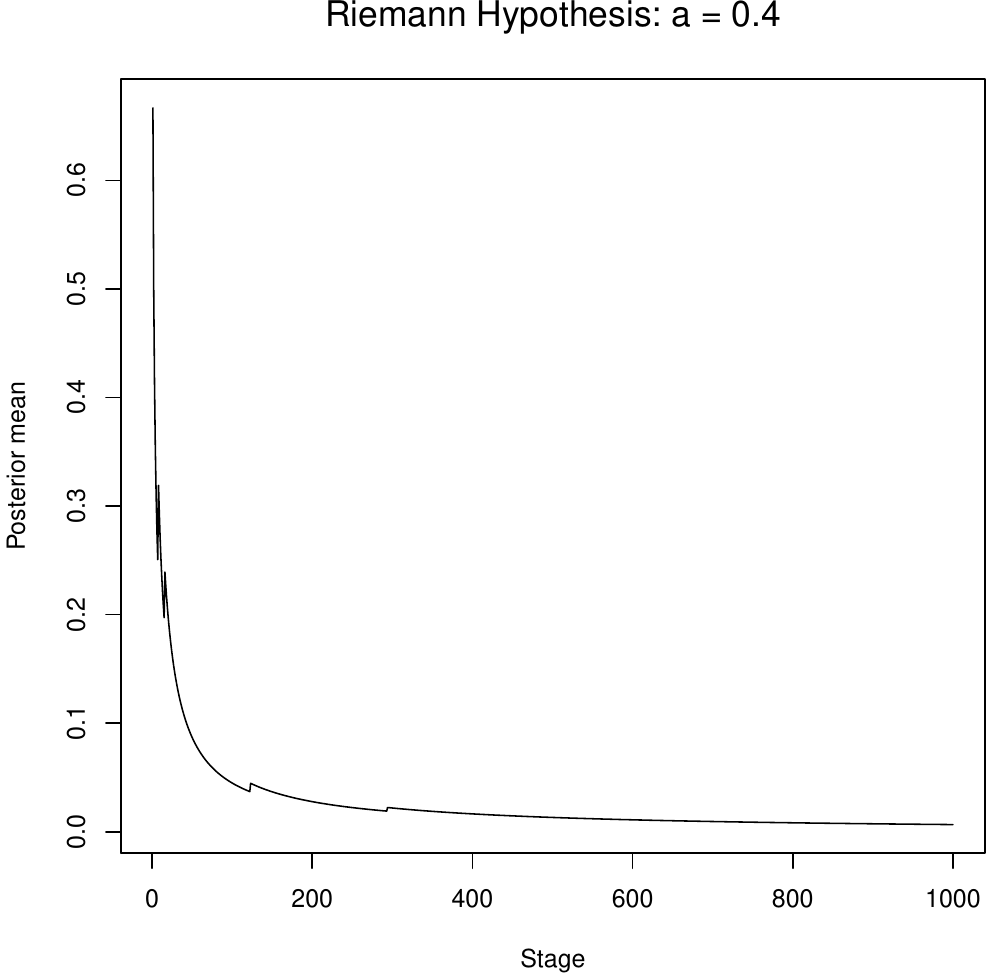}}\\
\subfigure [Divergence: $a=0.5$.]{ \label{fig:RH_a_05}
\includegraphics[width=6cm,height=5cm]{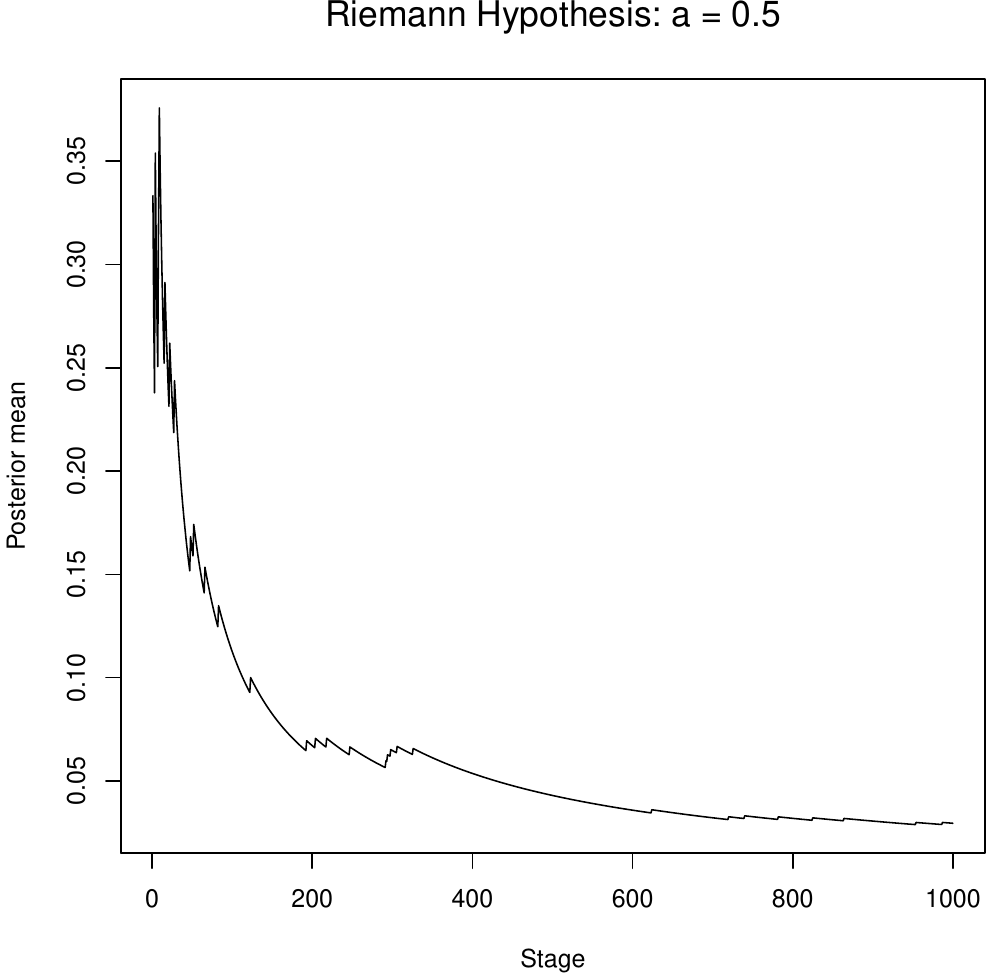}}
\hspace{2mm}
\subfigure [Divergence: $a=0.6$.]{ \label{fig:RH_a_06}
\includegraphics[width=6cm,height=5cm]{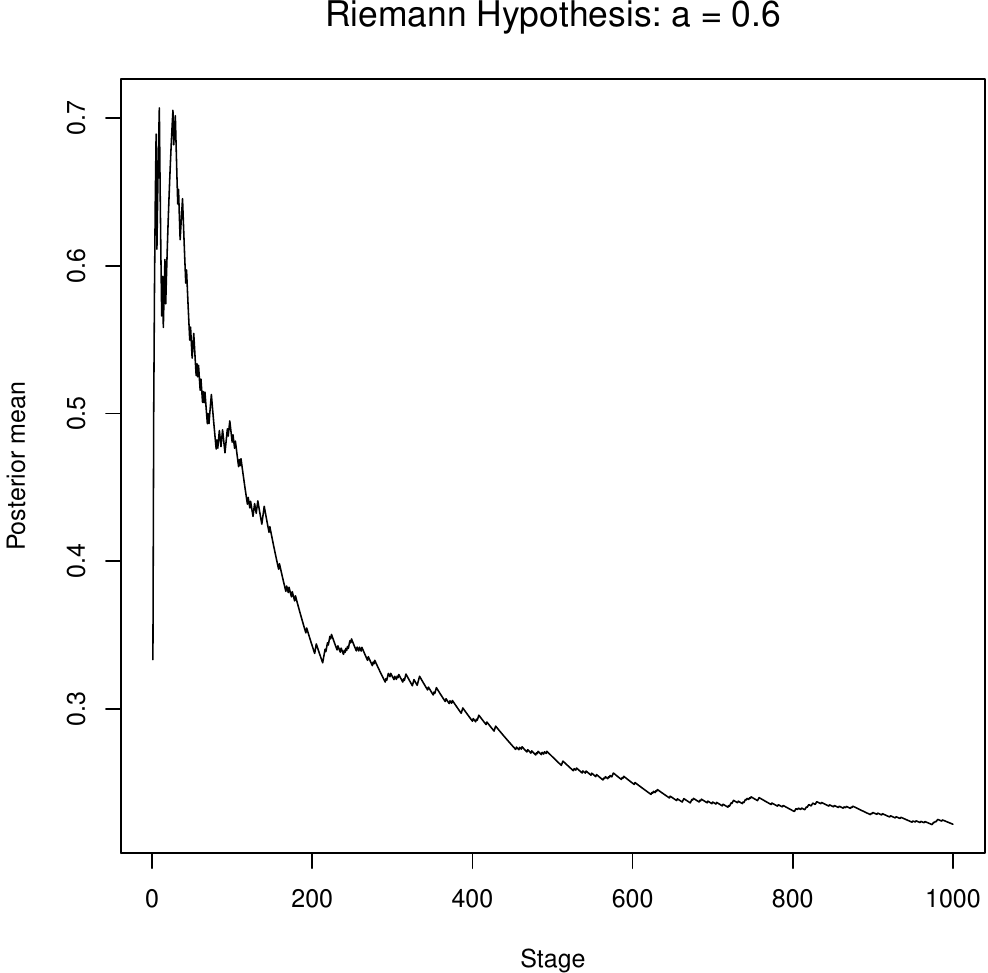}}
\caption{Riemann Hypothesis: The M\"{o}bius function based series diverges for 
$a=0.1$, $0.2$, $0.3$, $0.4$, $0.5$, $0.6$.}
\label{fig:RH_1}
\end{figure}

\begin{figure}
\centering
\subfigure [Divergence: $a=0.7$.]{ \label{fig:RH_a_07}
\includegraphics[width=6cm,height=5cm]{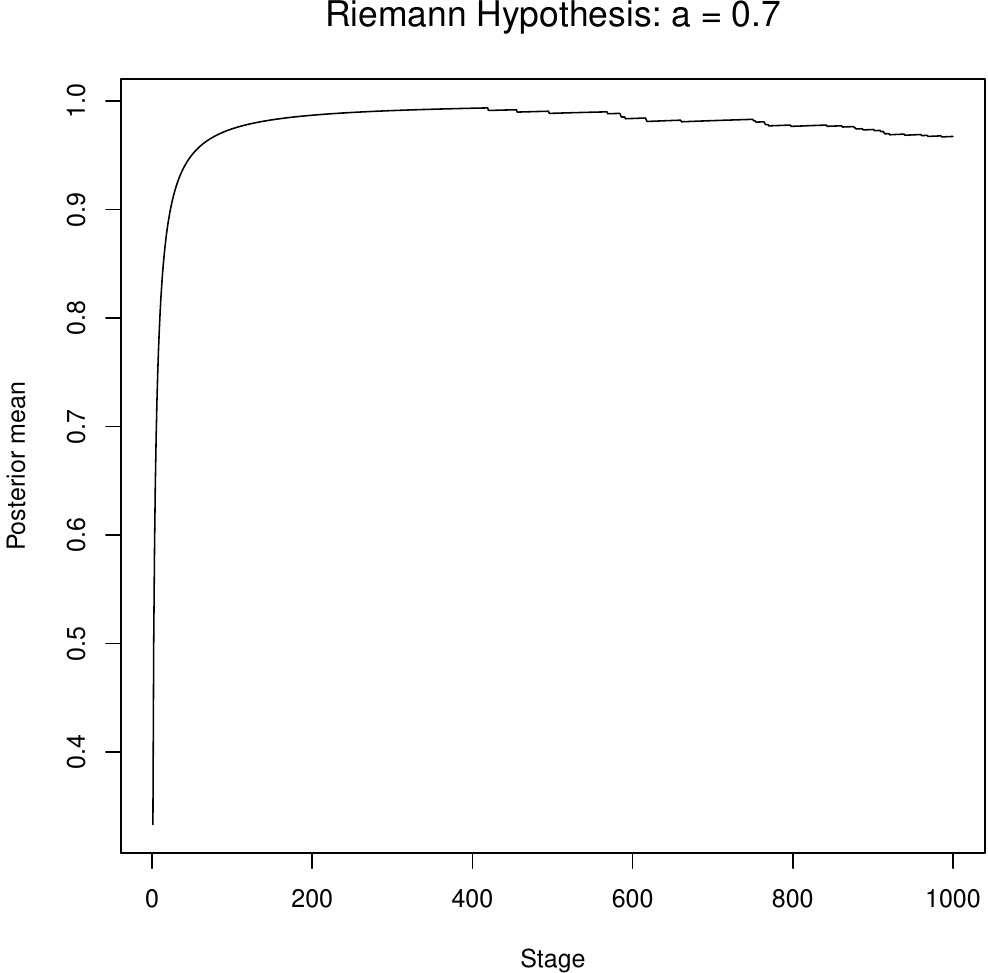}}
\hspace{2mm}
\subfigure [Convergence: $a=0.8$.]{ \label{fig:RH_a_08}
\includegraphics[width=6cm,height=5cm]{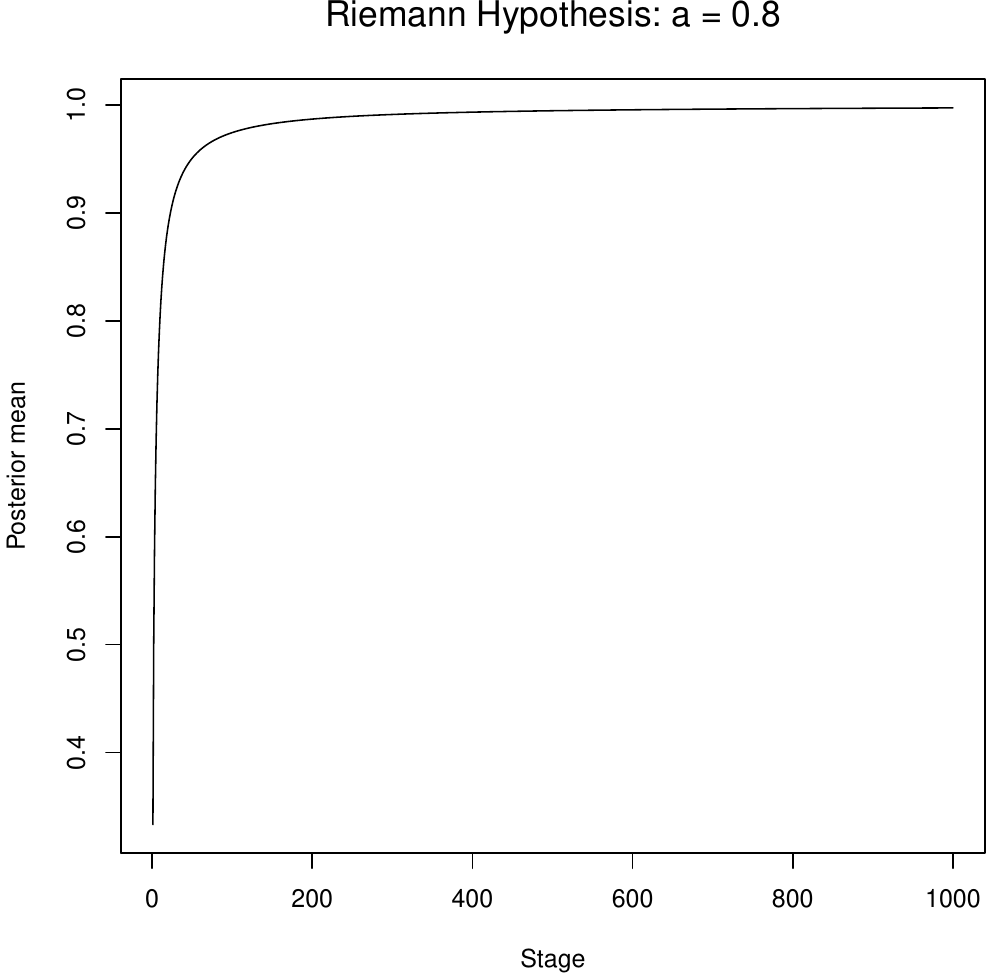}}\\
\subfigure [Convergence: $a=0.9$.]{ \label{fig:RH_a_09}
\includegraphics[width=6cm,height=5cm]{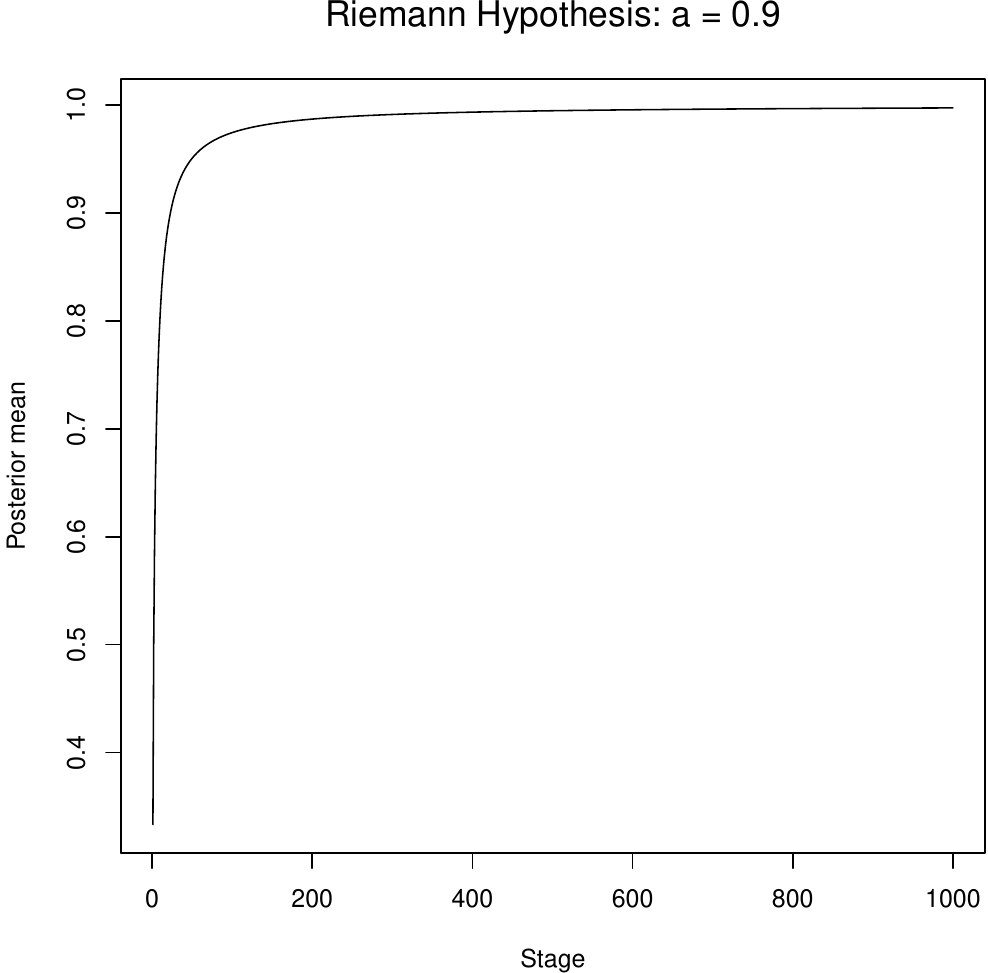}}
\hspace{2mm}
\subfigure [Convergence: $a=1+10^{-10}$.]{ \label{fig:RH_a_1_e}
\includegraphics[width=6cm,height=5cm]{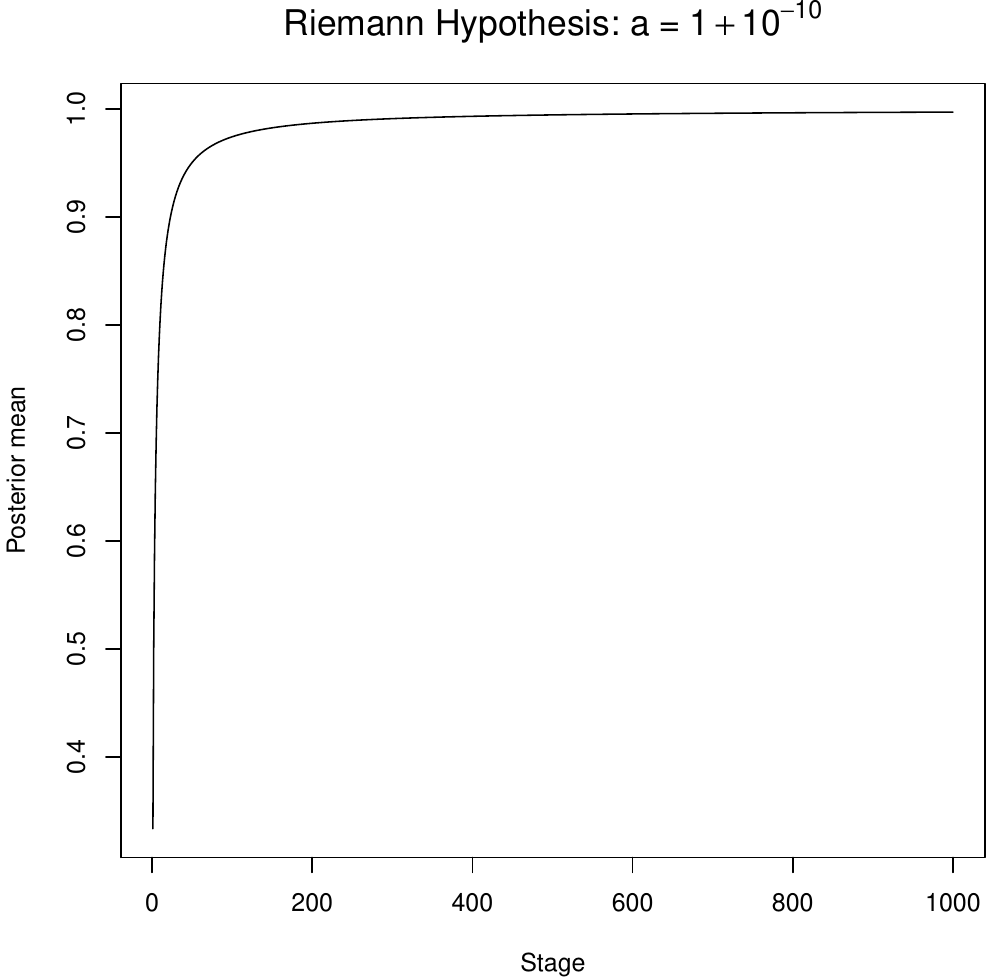}}\\
\subfigure [Convergence: $a=2$.]{ \label{fig:RH_a_2}
\includegraphics[width=6cm,height=5cm]{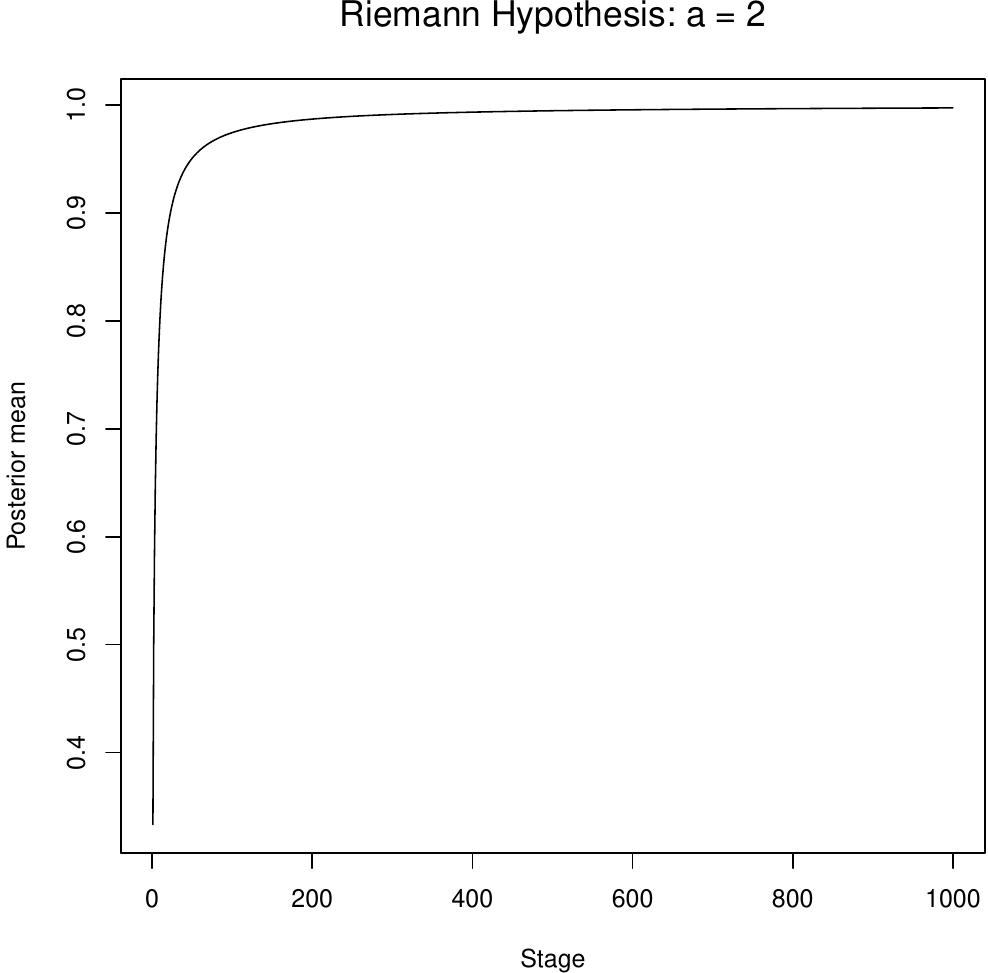}}
\hspace{2mm}
\subfigure [Convergence: $a=3$.]{ \label{fig:RH_a_3}
\includegraphics[width=6cm,height=5cm]{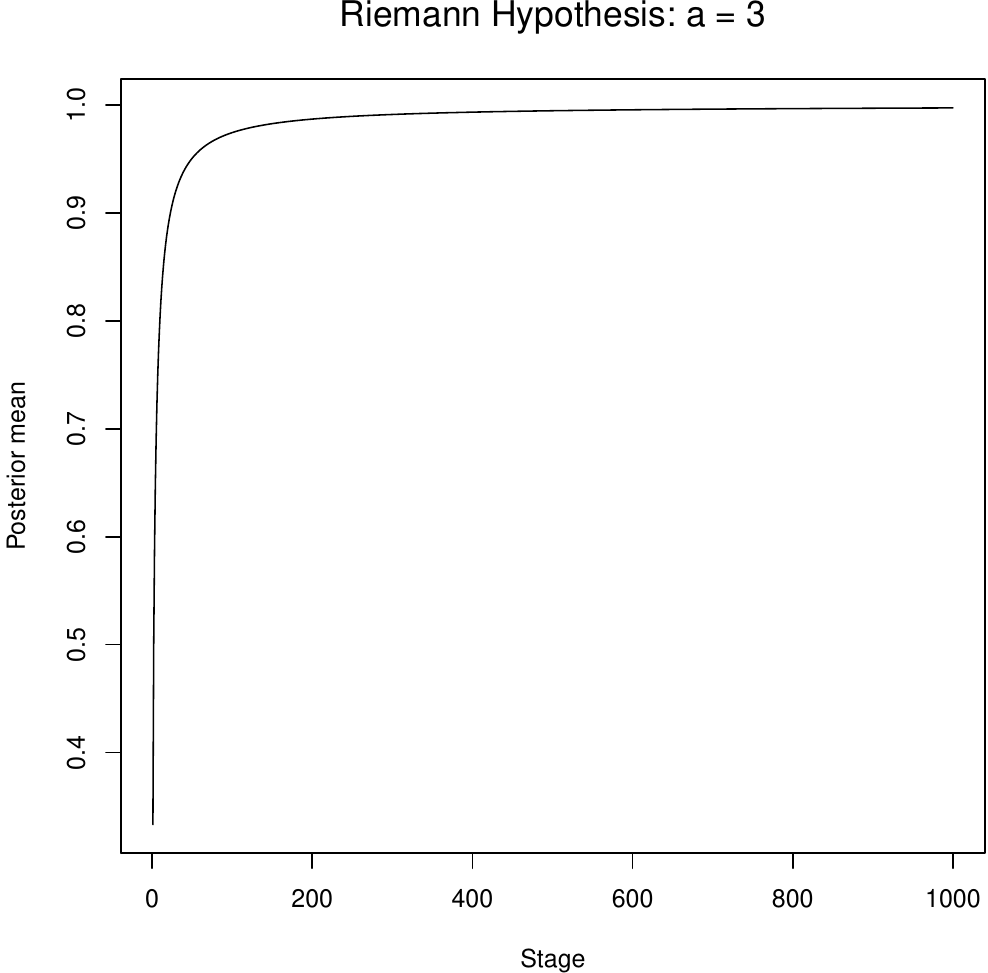}}
\caption{Riemann Hypothesis: The M\"{o}bius function based series diverges for 
$a=0.7$ but converges for $a=0.8$, $0.9$, $1+10^{-10}$, $2$, $3$.}
\label{fig:RH_2}
\end{figure}

\begin{figure}
\centering
\subfigure [Divergence: $a=0.71$.]{ \label{fig:RH_a_071}
\includegraphics[width=6cm,height=5cm]{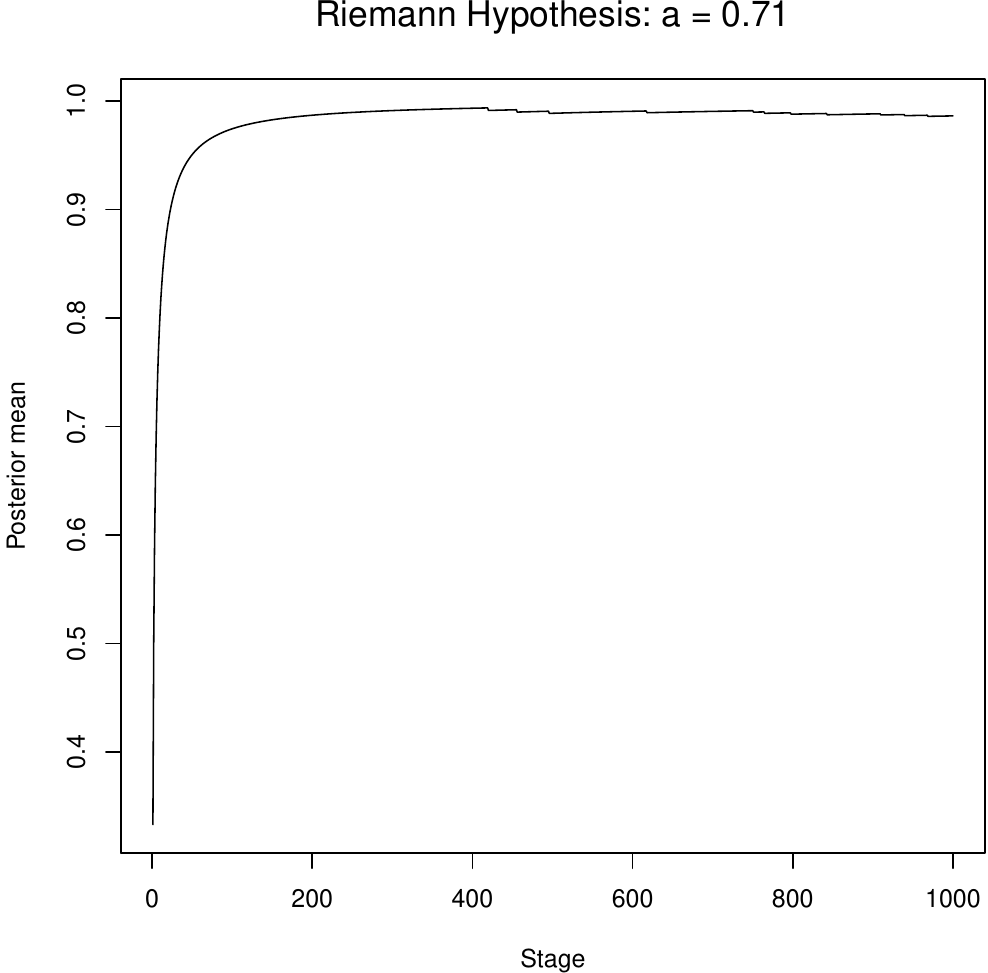}}
\hspace{2mm}
\subfigure [Divergence: $a=0.71$.]{ \label{fig:RH_a_071_partial}
\includegraphics[width=6cm,height=5cm]{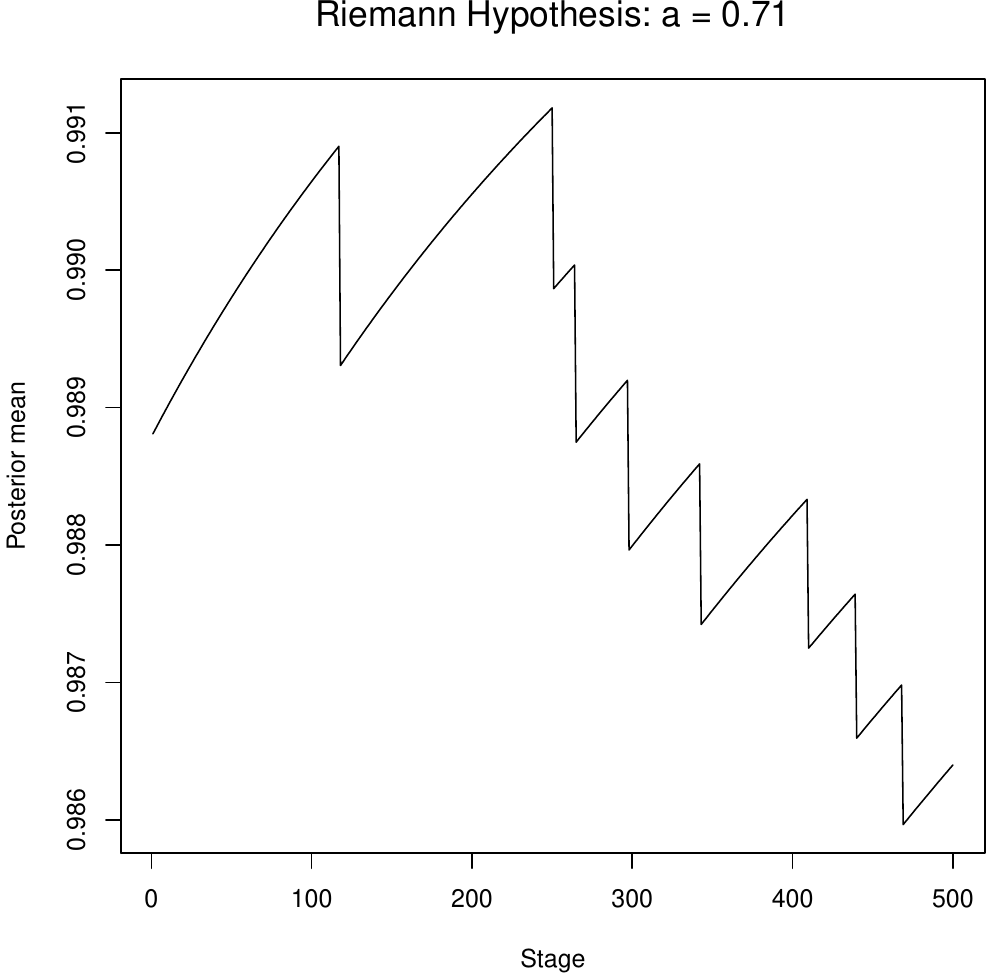}}\\
\subfigure [Divergence: $a=0.715$.]{ \label{fig:RH_a_0715}
\includegraphics[width=6cm,height=5cm]{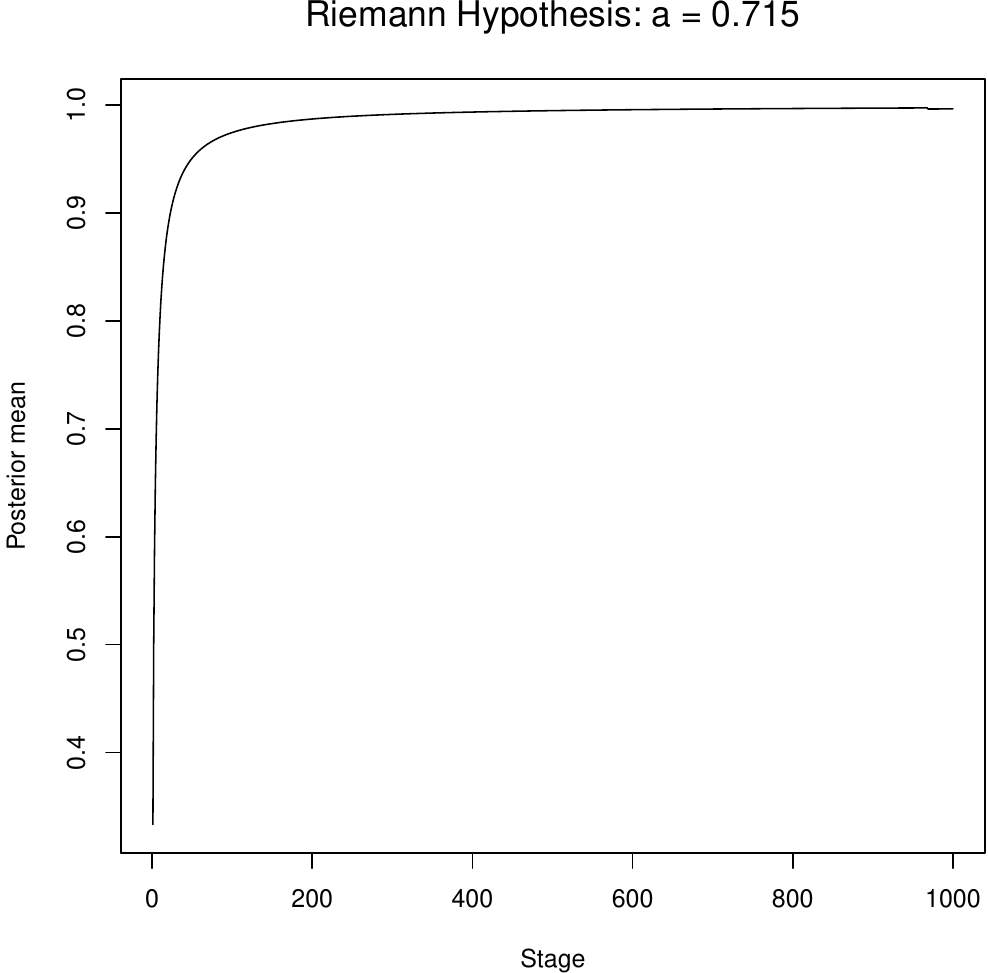}}
\hspace{2mm}
\subfigure [Divergence: $a=0.715$.]{ \label{fig:RH_a_0715_partial}
\includegraphics[width=6cm,height=5cm]{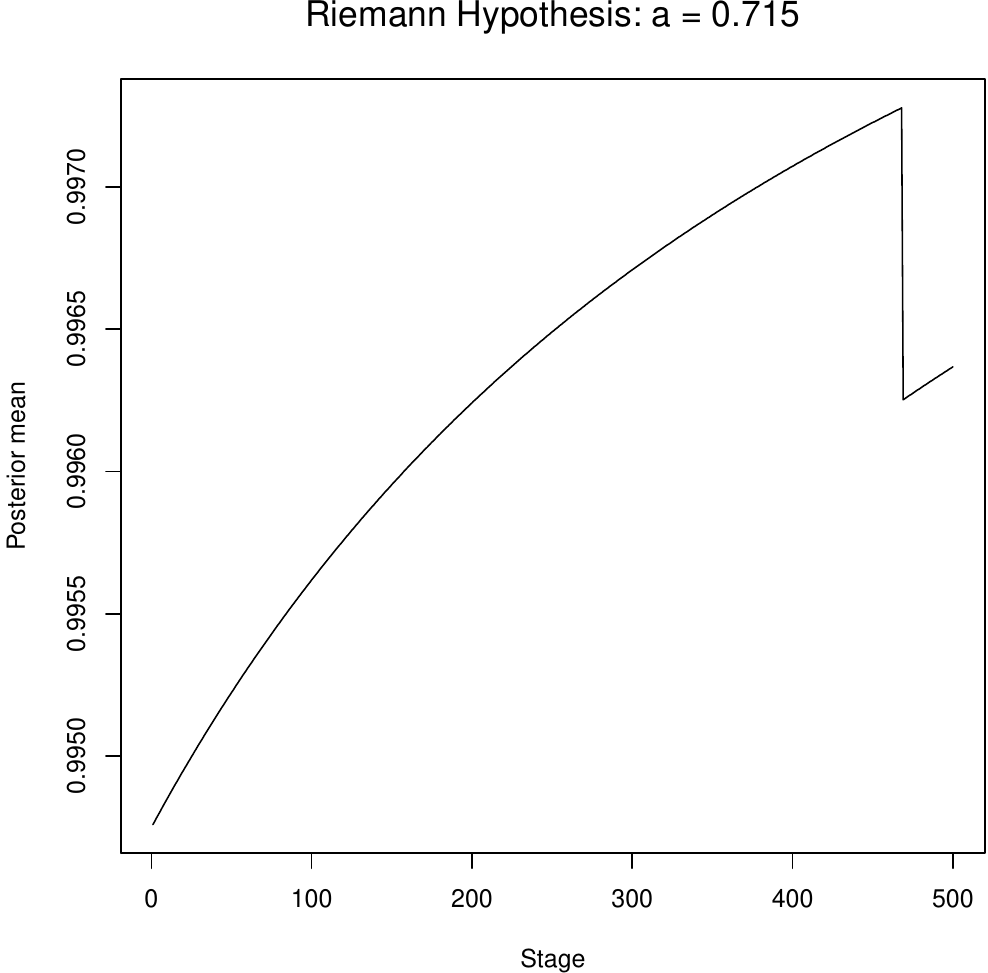}}\\
\subfigure [Convergence: $a=0.72$.]{ \label{fig:RH_a_072}
\includegraphics[width=6cm,height=5cm]{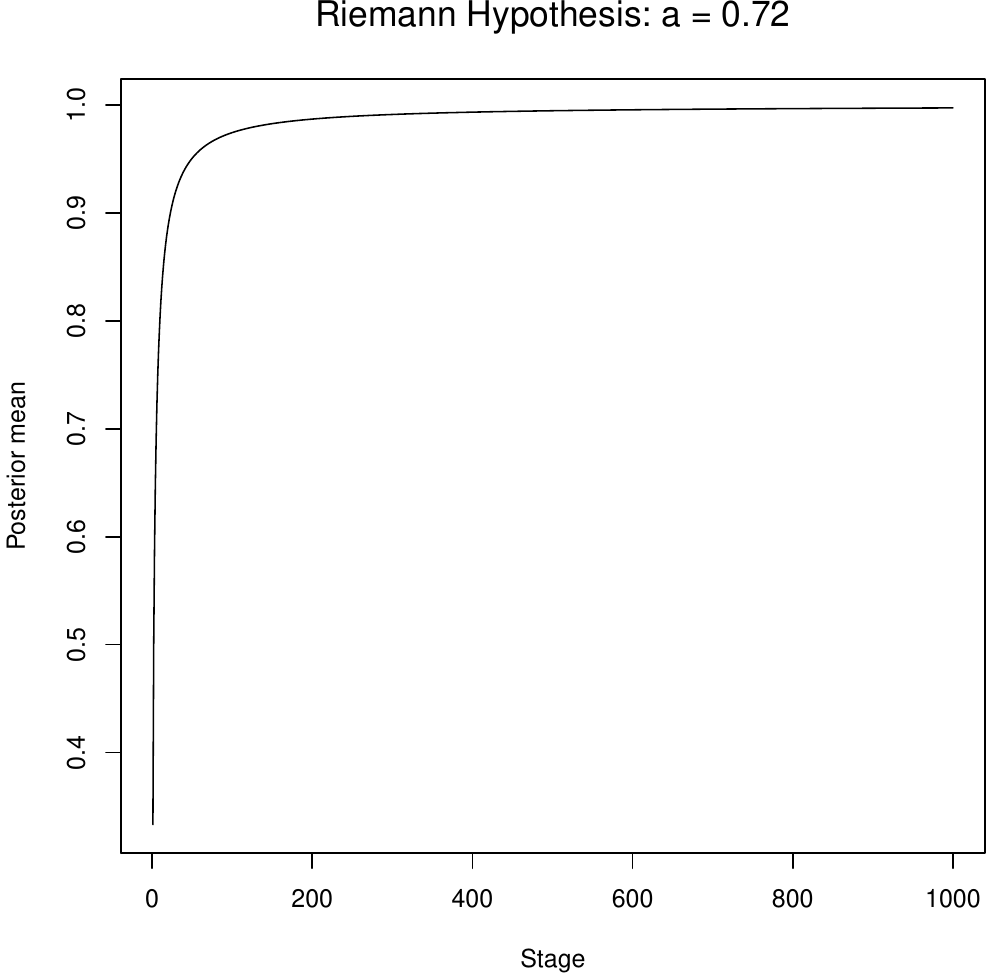}}
\hspace{2mm}
\subfigure [Convergence: $a=0.72$.]{ \label{fig:RH_a_072_partial}
\includegraphics[width=6cm,height=5cm]{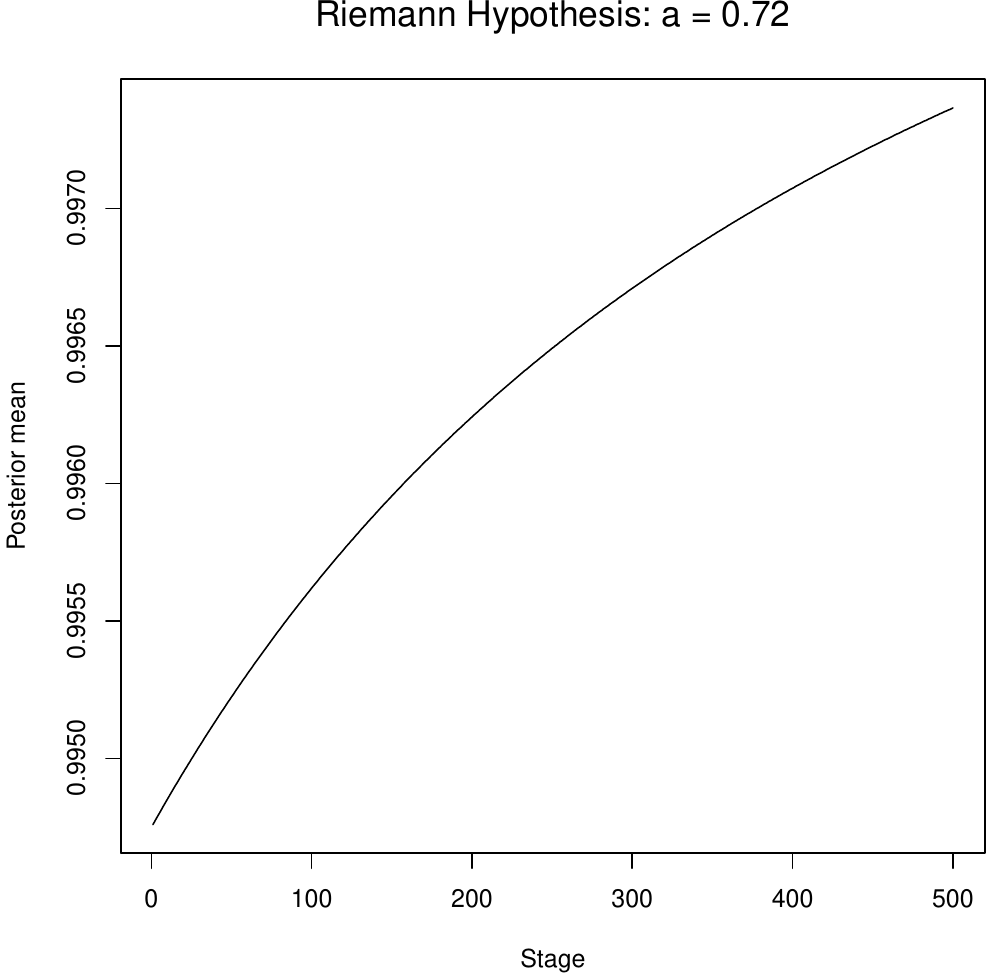}}
\caption{Riemann Hypothesis: The left panels show the posterior means for the full set of iterations,
while the right panels depict the posterior means for the last $500$ iterations, for $a=0.71$, $0.715$
and $0.72$. It is evident that the M\"{o}bius function based series diverges for $a=0.71$ and $0.715$ 
but converges for $a=0.72$.}
\label{fig:RH_3}
\end{figure}

\section{Summary and conclusion}
\label{sec:conclusion}
In this paper, we proposed and developed a novel Bayesian methodology for assessment of convergence
of infinite series; we further extended the theory to enable detection
of multiple or even infinite number of limit points of the underlying infinite series. 
Our developments do not require any restrictive assumption, not even independence of the elements
$X_i$ of the infinite series.

We demonstrated the reliability and efficiency of our
methods with varieties of examples, the most important one being associated with 
Riemann Hypothesis. 

Both methods proposed in this paper, namely the convergence assessment method and the
multiple limit points method are almost completely in agreement that the Riemann Hypothesis
can not be completely supported. Indeed, both the methods agree that there exists some
$a^*$ in the neighborhood of $0.7$ such that the infinite series based on the M\"{o}bius function
diverges for $a<a^*$ and converges for $a\geq a^*$. The results that we obtained by our Bayesian
analyses are also supported by informal plots of the partial sums depicted in Figure \ref{fig:plot_partial_sums}.
Further support of our Riemann hypothesis results can be obtained by exploiting the 
characterization of Riemann hypothesis by convergence of certain infinite series based on Bernoulli numbers;
the details are presented in Section S-6 of the supplement.

In fine, it is worth reminding the reader that although our work attempts to provide 
insights regarding Riemann hypothesis, 
we did not develop our Bayesian approach keeping Riemann hypothesis in mind. Indeed,
our primary objective is to develop Bayesian approaches to studying convergence properties of infinite series in general.
From this perspective, Riemann hypothesis is just an example where it makes sense to learn about convergence properties of 
a certain class of infinite series. Further development of our approach is of course in the cards.
Note that the theory that we developed for deterministic series remains valid for random series as well, but since the forms of the terms of random series are unknown, 
direct application of our methods is not possible. We are currently developing new theories and methods for random series where the terms have unknown distributions and/or not independent. 
We shall carry out a detailed investigation including
comparisons with existing theories on random infinite series. We then intend to extend these works
to complex infinite series, both deterministic and random.

\section*{Acknowledgment}
We thank Arun Kumar Kuchibhotla and Debapratim Banerjee for their very useful feedback
on the first draft of this paper.

\newpage

\renewcommand\thefigure{S-\arabic{figure}}
\renewcommand\thetable{S-\arabic{table}}
\renewcommand\thesection{S-\arabic{section}}

\setcounter{section}{0}
\setcounter{figure}{0}
\setcounter{table}{0}

\begin{center}
{\bf \Large Supplementary Material}
\end{center}

\section{Proof of Lemma 5.1}
\label{sec:proof_lemma5.1}
Since each term of the series (1) 
is decreasing in $a$, it is clear that $S^{a,b}_{j,n}$
is decreasing in $a$. We need to show that $S^{a,b}_{j,n}$ is increasing in $b$.  

Let, for $i\geq 3$,
\begin{equation}
g(i)=\left(1-\frac{\log i}{i}-\frac{\log\log i}{i}
\left\{\cos^2\left(\frac{1}{i}\right)\right\}\left(a+(-1)^ib\right)\right)^i.
\label{eq:g_example4}
\end{equation}

Observe that all our partial sums of the form
$S^{a,b}_{j,n}$ for $j\geq 3$ admit the form 
\begin{equation}
S^{a,b}_{j,n}=\sum_{i=r}^{r+n-1}g(i), 
\label{eq:S_partial_example4}
\end{equation}
where $r=3+n(j-1)$, which is clearly odd because $n$ is even.
Now,
\begin{equation}
\sum_{i=r}^{r+n-1}g(i)=\left\{g(r)+g(r+1)\right\}+\left\{g(r+2)+g(r+3)\right\}+\cdots +\left\{g(r+n-2)+g(r+n-1)\right\},
\label{eq:sum_example4}
\end{equation}
where the sums of the consecutive terms within the parentheses have the form
\begin{align}
&g(r+\ell)+g(r+\ell+1)\notag\\
&=\left(1-\frac{\log (r+\ell)}{r+\ell}-\frac{\log\log (r+\ell)}{r+\ell}
\left\{\cos^2\left(\frac{1}{r+\ell}\right)\right\}\left(a+(-1)^{(r+\ell)}b\right)\right)^{(r+\ell)}\notag\\
&\quad\quad+\left(1-\frac{\log (r+\ell+1)}{r+\ell+1}-\frac{\log\log (r+\ell+1)}{r+\ell+1}
\left\{\cos^2\left(\frac{1}{r+\ell+1}\right)\right\}\left(a+(-1)^{(r+\ell+1)}b\right)\right)^{(r+\ell+1)}.
\label{eq:pairwise_sum_example4}
\end{align}
Since $r$ is odd, and since the terms are represented pairwise in (\ref{eq:sum_example4}) it follows that 
in (\ref{eq:pairwise_sum_example4}), $r+\ell$ is odd and $r+\ell+1$ is even. That is, in (\ref{eq:pairwise_sum_example4}),
$a+(-1)^{(r+\ell)}b=a-b$ and $a+(-1)^{(r+\ell+1)}b=a+b$.
Since $\cos^2\left(\theta\right)$ is decreasing on $\left[0,\frac{\pi}{2}\right]$, and since
$\frac{1}{i}\leq \frac{\pi}{2}$ for $i\geq 3$, it follows that $\cos^2\left(\frac{1}{i}\right)$ is increasing in $i$.
Moreover, $\frac{\log\log i}{i}$ decreases in $i$ at a rate faster
than $\cos^2\left(\frac{1}{i}\right)$ increases, so that 
$\frac{\log\log i}{i}\times\cos^2\left(\frac{1}{i}\right)$ decreases in $i$.
It follows that 
\begin{equation}
\frac{\log\log (r+\ell)}{r+\ell}\cos^2\left(\frac{1}{r+\ell}\right)
>\frac{\log\log (r+\ell+1)}{r+\ell+1}\cos^2\left(\frac{1}{r+\ell+1}\right).
\label{eq:compare}
\end{equation}
Note that in $g(r+\ell)+g(r+\ell+1)$, $\frac{\log\log (r+\ell)}{r+\ell}\cos^2\left(\frac{1}{r+\ell}\right)$
is associated with $-b$ while $\frac{\log\log (r+\ell+1)}{r+\ell+1}\cos^2\left(\frac{1}{r+\ell+1}\right)$
involves $b$. Hence, increasing $b$ increases $g(r+\ell)$ but decreases $g(r+\ell+1)$,
and because of (\ref{eq:compare}), $g(r+\ell)+g(r+\ell+1)$ increases in $b$. 
This ensures that $\sum_{i=r}^{r+n-1}g(i)$,
given by (\ref{eq:sum_example4}), is increasing in $b$. In other words, partial sums of the form 
(\ref{eq:S_partial_example4}) are
increasing in $b$, proving Lemma 5.1 
when $n$ is even.

\section{Further examples on detection of series convergence and divergence using our Bayesian method}
\subsection{Example 5}
\label{subsec:example5}
Now consider the following series presented and analysed in \ctn{Bou12}:
\begin{equation}
S=\sum_{i=3}^{\infty}\left(1-\left(\frac{\log(i)}{i}\right)
\left(a\left(1+\sin^2\left(\sqrt{\left(\frac{\log\left(\log(i)\right)}{\log(i)}\right)}\right)\right)
+b\sin\left(\frac{i\pi}{4}\right)\right)\right)^i;~a>0,b>0.
\label{eq:example5}
\end{equation}
\ctn{Bou12} show that the series converges when $a-b>1$ and diverges when $a+b<1$.
Again, as in the case of Example 4, the following lemma holds in Example 5.
Note that for mathematical convenience we consider partial sums from the $5$-th term onwards.
We also assume $n$ to be a multiple of $4$.
\begin{lemma}
\label{lemma:example5}
For the series (\ref{eq:example5}), 
let 
\begin{equation}
S^{a,b}_{j,n}=\sum_{i=5+n(j-1)}^{5+nj-1}\left(1-\left(\frac{\log(i)}{i}\right)
\left(a\left(1+\sin^2\left(\sqrt{\left(\frac{\log\left(\log(i)\right)}{\log(i)}\right)}\right)\right)
+b\sin\left(\frac{i\pi}{4}\right)\right)\right)^i,
\label{eq:S_example5}
\end{equation}
for $j\geq 1$ and $n$, a multiple of $4$. Then $S^{a,b}_{j,n}$ is decreasing in $a$ and increasing in $b$.
\end{lemma}
\begin{proof}
That $S^{a,b}_{j,n}$ is decreasing in $a$ follows trivially since
each term of (\ref{eq:example5}) is decreasing in $a$. We need to show that $S^{a,b}_{j,n}$ is increasing in $b$.  

Let, for $i\geq 5$,
\begin{equation}
g(i)=\left(1-\left(\frac{\log(i)}{i}\right)
\left(a\left(1+\sin^2\left(\sqrt{\left(\frac{\log\left(\log(i)\right)}{\log(i)}\right)}\right)\right)
+b\sin\left(\frac{i\pi}{4}\right)\right)\right)^i.
\label{eq:g_example5}
\end{equation}
Now note that, with $r=5+n(j-1)$,
\begin{align}
\sum_{i=r}^{r+n-1}g(i)&=\sum_{m=1}^{\frac{n}{4}}Z_{r,m}\notag\\
&=\left\{Z_{r,1}+Z_{r,2}\right\}+\left\{Z_{r,3}+Z_{r,4}\right\}+
\cdots+\left\{Z_{r,\frac{n}{4}-1}+Z_{r,\frac{n}{4}}\right\},
\label{eq:sum_example5}
\end{align}
where
\begin{equation}
Z_{r,m}=\sum_{\ell=5+4(m-1)}^{5+4(m-1)+3}g(r+\ell).
\label{eq:Z_example5}
\end{equation}

Now, for any $\ell\geq 1$, observe that in $\left\{Z_{r,\ell}+Z_{r,\ell+1}\right\}$, the term
$Z_{r,\ell}$ consists of only negative signs of the sine-values, while in $Z_{r,\ell+1}$ the
corresponding signs are positive, although the magnitudes are the same. Since $\log(i)/i$ is decreasing in $i$,
it follows that $\left\{Z_{r,\ell}+Z_{r,\ell+1}\right\}$ is increasing in $b$ for $\ell\geq 1$. Hence,
it follows that (\ref{eq:sum_example5}), and $S^{a,b}_{j,n}$, defined by (\ref{eq:S_example5}),
are increasing in $b$ for $j\geq 1$ and $n$, a multiple of $4$, proving Lemma \ref{lemma:example5}.
\end{proof}

The following corollary with respect to $S^{a,b}$ again holds:
\begin{corollary}
\label{corollary:example5}
$S^{a,b}$ is decreasing in $a$ and increasing in $b$.
\end{corollary}
Thus, we follow the same method as in Example 4 to determine $c^{a,b}_{j,n}$, but we need to note 
that in this example $a>0$ and $b>0$ instead of $a\geq 0$ and $b\geq 0$ of Example 4.  Consequently, here we define
$b\geq\epsilon$, for $\epsilon>0$, the set $A_{\epsilon}$ given by 
$ A_{\epsilon}=\left\{a:0\leq a\leq 1\right\}\cup\left\{a:a\geq 1+\epsilon\right\}$
and 
\begin{equation}
\tilde S=\underset{a\in A_{\epsilon}}{\inf}\underset{b\geq \epsilon}{\sup}~\left\{S^{a,b}:a-b>1\right\}.
\label{eq:tilde_S2}
\end{equation}
In this case, Corollary \ref{corollary:example5} and the convergence
criterion $a-b>1$ ensure that $\tilde S$ is attained at $a_0=1+\epsilon$ and $b_0=\epsilon$.
As before, we set $\epsilon=10^{-10}$.
The rest of the arguments leading to the choice of $c^{a,b}_{j,n}$ remains the same as in Example 4, 
and hence in this example $c^{a,b}_{j,n}$ has the form 
\begin{equation}
c^{a,b}_{j,n}=\left\{\begin{array}{ccc} u^{a,b}_{j,n}, &\mbox{if}~u^{a,b}_{j,n}>0;\\
S^{a_0,b_0}_{j,n}, & \mbox{otherwise},\end{array}\right.
\label{eq:example4_c_supp}
\end{equation}
with $a_0=1+10^{-10}$, $b_0=10^{-10}$, where $S^{a_0,b_0}_{j,n}$ is decreasing in $j$ as before.

Figure \ref{fig:example5} depicts the results of our Bayesian analysis of the series (\ref{eq:example5}) for
various values of $a$ and $b$. All the results are in accordance with those of \ctn{Bou12}.
\begin{figure}
\centering
\subfigure [Convergence: $a=2,b=1$.]{ \label{fig:example5_a_2_b_1}
\includegraphics[width=6cm,height=5cm]{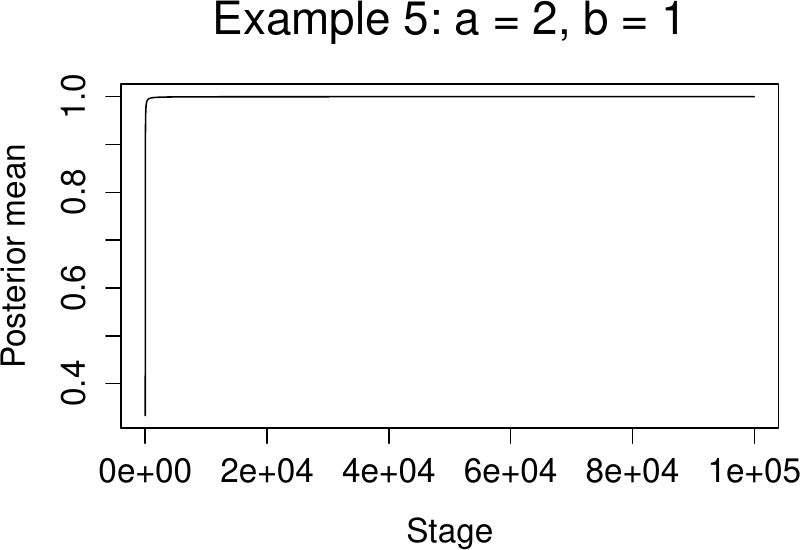}}
\hspace{2mm}
\subfigure [Convergence: $a=1+20^{-10},b=10^{-10}$.]{ \label{fig:example5_a12_b01}
\includegraphics[width=6cm,height=5cm]{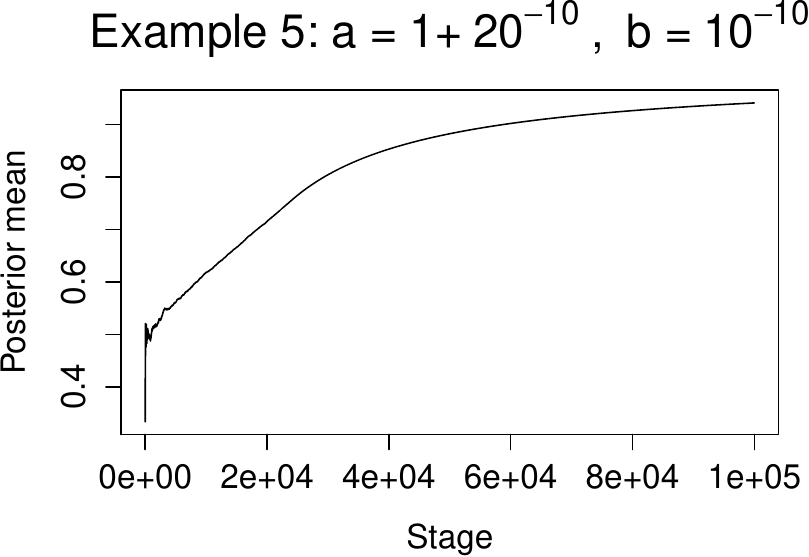}}\\
\hspace{2mm}
\subfigure [Convergence: $a=1+30^{-10},b=20^{-10}$.]{ \label{fig:example5_a13_b02}
\includegraphics[width=6cm,height=5cm]{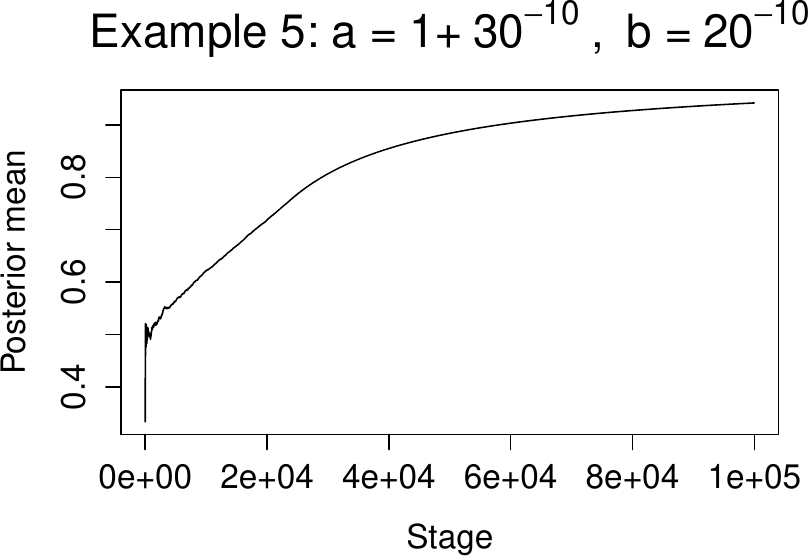}}
\hspace{2mm}
\subfigure [Divergence: $a=1/2,b=1/2$.]{ \label{fig:example5_a_1_2_b_1_2}
\includegraphics[width=6cm,height=5cm]{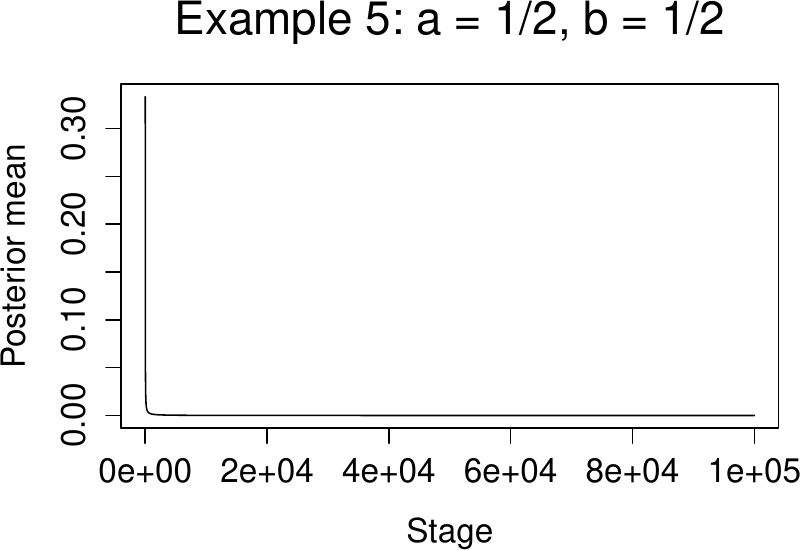}}\\
\subfigure [Divergence: $a=\frac{1}{2}\left(1-10^{-11}\right),
b=\frac{1}{2}\left(1-10^{-11}\right)$.]{ \label{fig:example5_a+b_less_1}
\includegraphics[width=6cm,height=5cm]{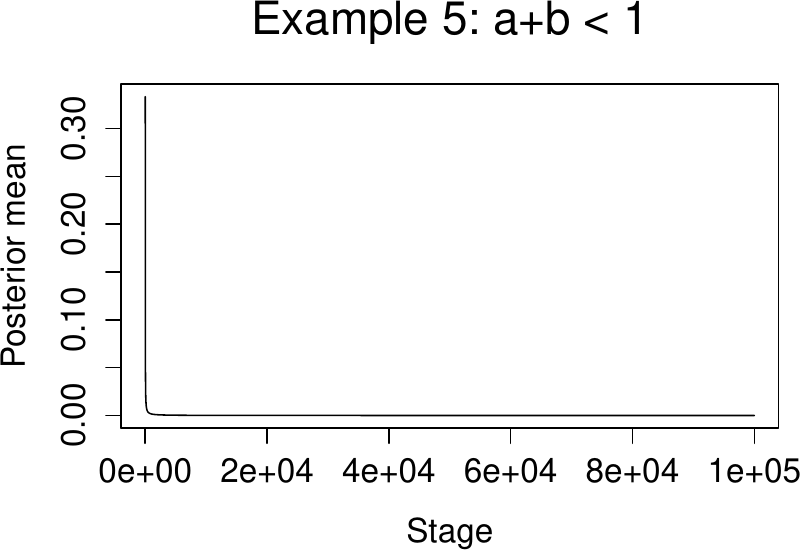}}
\caption{Example 5: The series (\ref{eq:example5}) converges for $(a=2,b=1)$, $(a=1+20^{-10},b=10^{-10})$,
$(a=1+30^{-10},b=20^{-10})$ and diverges for $(a=1/2,b=1/2)$ and 
$\left(a=\frac{1}{2}\left(1-10^{-11}\right),b=\frac{1}{2}\left(1-10^{-11}\right)\right)$.}
\label{fig:example5}
\end{figure}

\subsection{Example 6}
\label{subsec:example6}
We now investigate whether or not the following series converges:
\begin{equation}
S=\sum_{i=1}^{\infty}\frac{1}{i^3|\sin i|}.
\label{eq:example6_S}
\end{equation}
This series is a special case of the generalized form of the Flint Hills series (see
\ctn{Pickover02} and \ctn{Alek11}).

For our purpose, we first embed the above series into 
\begin{equation}
S^{a,b}=\sum_{i=1}^{\infty}\frac{i^{b-3}}{a+|\sin i|},
\label{eq:example6}
\end{equation}
where $b\in\mathbb R$ and $|a|\leq\eta$, for some $\eta>0$, specified according to our
purpose. Note that, $S=S^{0,0}$, and we set $\eta=10^{-10}$ for our investigation of (\ref{eq:example6_S}).

Note that for any fixed $a\neq 0$, $S^{a,b}$ converges if $b<2$ and diverges if $b\geq 2$.
Since $S^{a,b}$ increases in $b$
it follows that the equality in
\begin{equation}
\tilde S=\sup\left\{S^{a,b}:a=\epsilon,~b\leq 2-\epsilon \right\}
\label{eq:sup_A_example6}
\end{equation}
is attained at $(a_0,b_0)=(\epsilon,2-\epsilon)$.

Arguments in keeping with those in the previous examples
lead to the following choice of the
upper bound for $S^{a,b}_{j,n}$, which we again denote by $c^{a,b}_{j,n}$:
\begin{equation}
c^{a,b}_{j,n}=\left\{\begin{array}{ccc} u^{a,b}_{j,n}, &\mbox{if}~b<2;\\
v^{a,b}_{j,n}, & \mbox{otherwise},\end{array}\right.
\label{eq:example6_c}
\end{equation}
where 
\begin{align}
u^{a,b}_{j,n}&=S^{a_0,b_0}_{j,n}+\frac{(|a|-b+2-2\epsilon+10^{-5})}{\log(j+1)}; 
\label{eq:u_example6}\\
v^{a,b}_{j,n}&=S^{a_0,b_0}_{j,n}+\frac{(|a|-b+2-2\epsilon-10^{-5})}{\log(j+1)}. 
\label{eq:v_example6}
\end{align}
It can be easily verified that the upper bound is decreasing in $j$.
Notice that we add the term $10^{-5}$ when $b<2$ so that our Bayesian method favours convergence and subtract 
the same when $b\geq 2$ to facilitate detection of divergence. Since convergence or divergence of $S^{a,b}$ 
does not depend upon $a\in [-\eta,\eta]\setminus\left\{0\right\}$, 
we use $|a|$ in (\ref{eq:u_example6}) and (\ref{eq:v_example6}). 

Setting $\epsilon=10^{-10}$, Figures \ref{fig:example6a} and \ref{fig:example6b} 
depict convergence and divergence of $S^{a,b}$ for various
values of $a$ and $b$. In particular, panel (e) of Figure \ref{fig:example6b} 
shows that our main interest, the series $S$, given by (\ref{eq:example6_S}), converges.
\begin{figure}
\centering
\subfigure [Convergence: $a=-10^{-10},b=2-10^{-10}$.]{ \label{fig:example6_a_minus_e_b_2_minus_e}
\includegraphics[width=6cm,height=5cm]{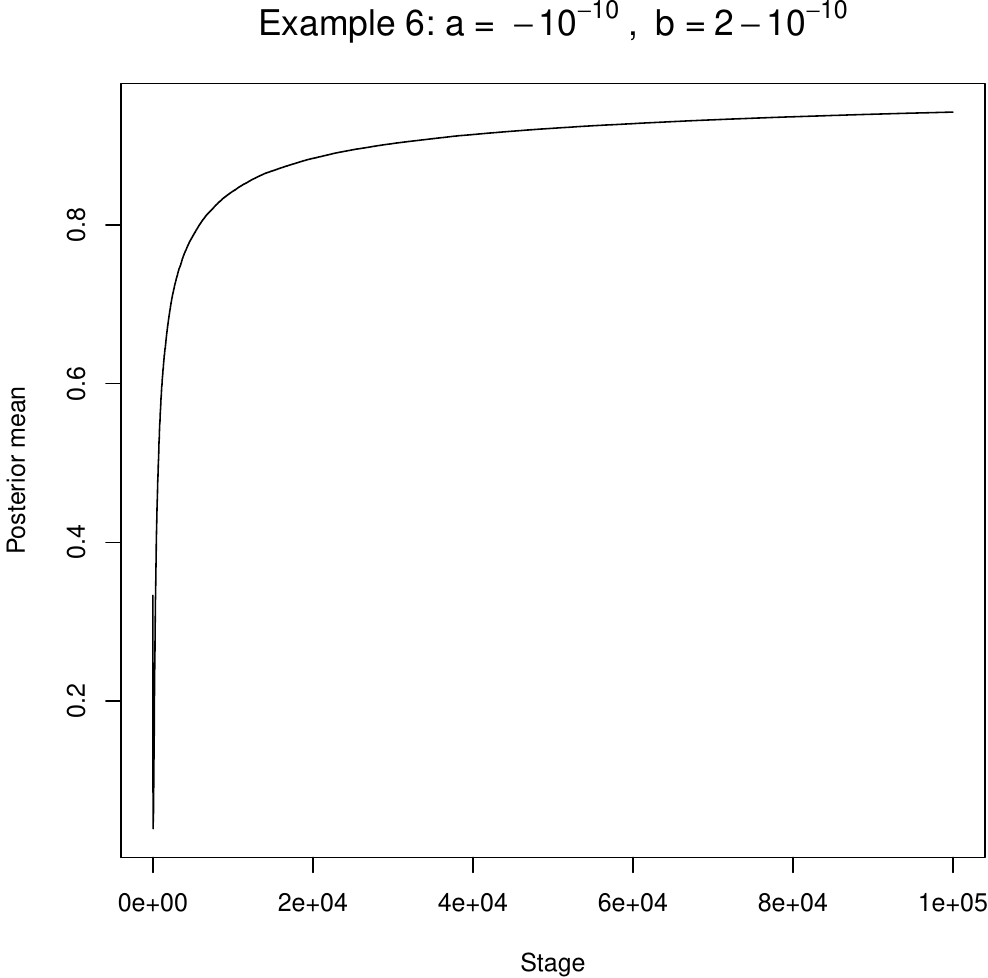}}
\hspace{2mm}
\subfigure [Divergence: $a=-10^{-10},b=2+10^{-10}$.]{ \label{fig:example6_a_minus_e_b_2_plus_e}
\includegraphics[width=6cm,height=5cm]{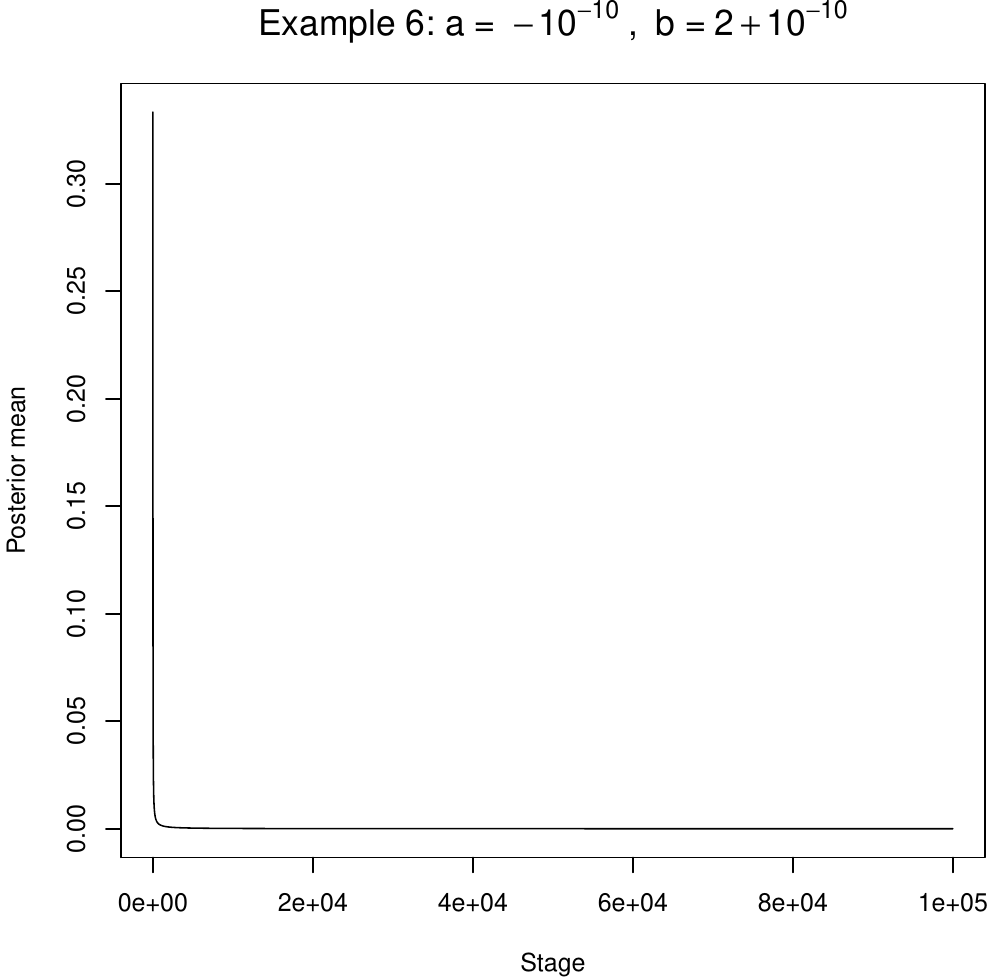}}\\
\subfigure [Convergence: $a=10^{-10},b=2-10^{-10}$.]{ \label{fig:example6_a_plus_e_b_2_minus_e}
\includegraphics[width=6cm,height=5cm]{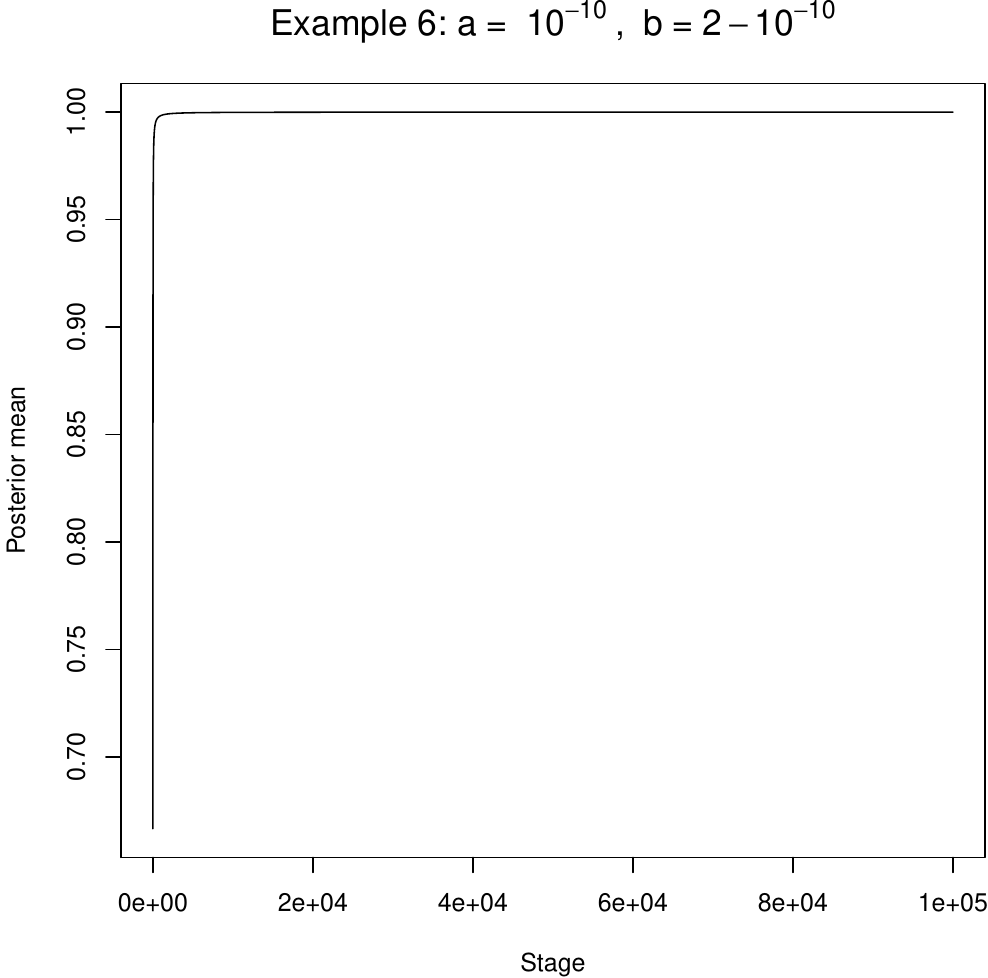}}
\hspace{2mm}
\subfigure [Divergence: $a=10^{-10},b=2+10^{-10}$.]{ \label{fig:example6_a_plus_e_b_2_plus_e}
\includegraphics[width=6cm,height=5cm]{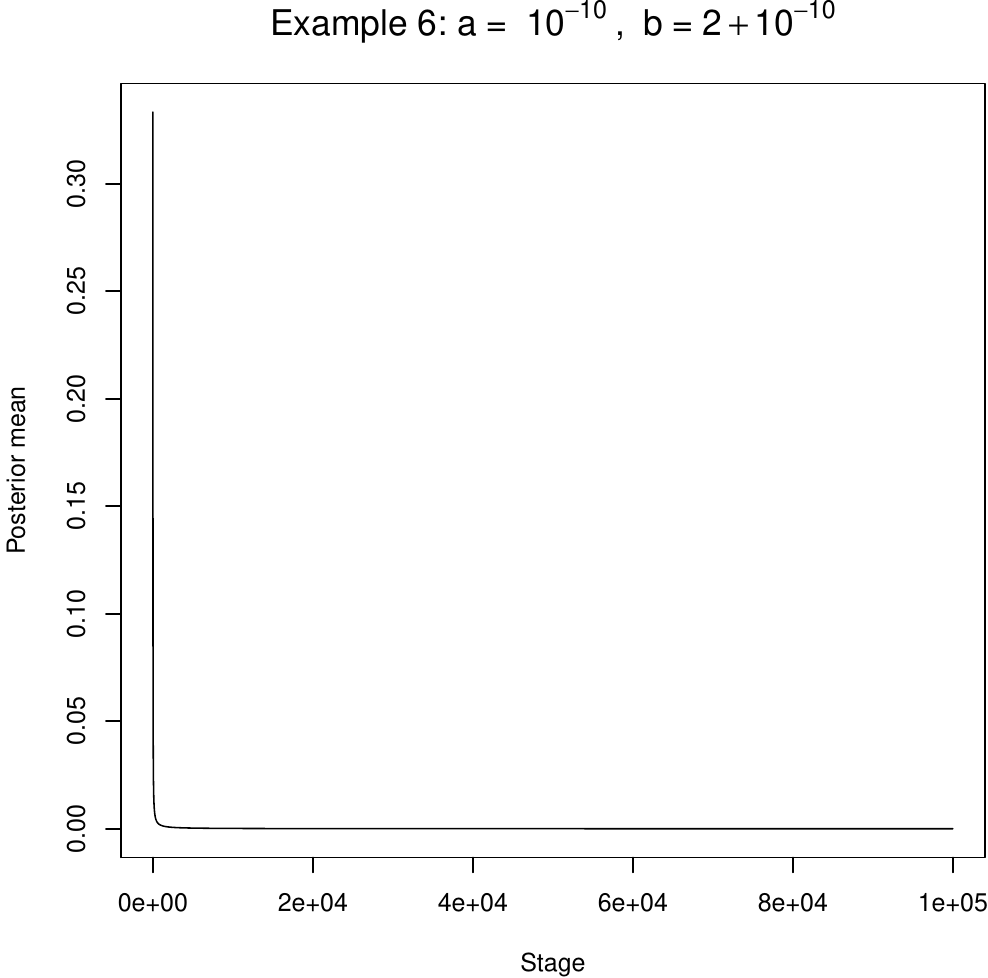}}
\caption{Example 6: The series (\ref{eq:example6}) converges for $(a=-10^{-10},b=2-10^{-10})$, 
$(a=10^{-10},b=2-10^{-10})$, and diverges for $(a=-10^{-10},b=2+10^{-10})$,
$(a=10^{-10},b=2+10^{-10})$.}
\label{fig:example6a}
\end{figure}

\begin{figure}
\centering
\subfigure [Convergence: $a=-10^{-10},b=-10^{-10}$.]{ \label{fig:example6_a_minus_e_b_minus_e}
\includegraphics[width=6cm,height=5cm]{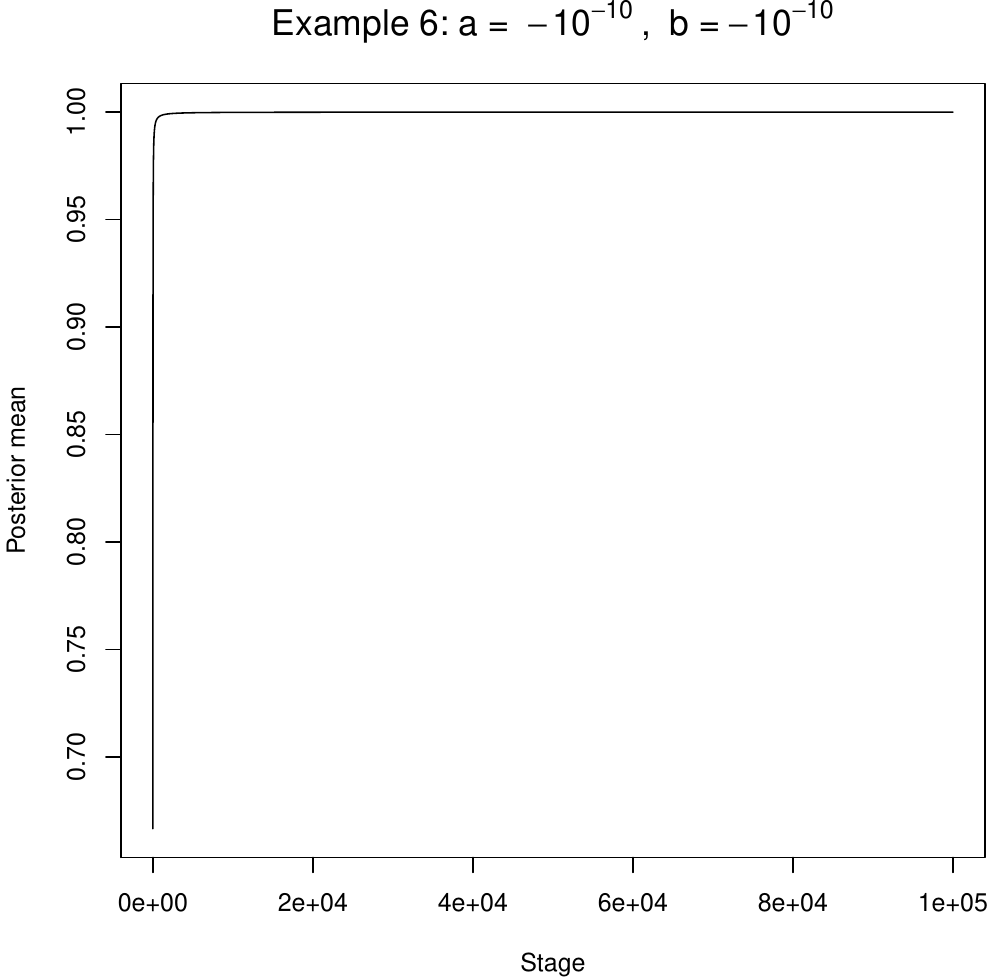}}
\hspace{2mm}
\subfigure [Convergence: $a=-10^{-10},b=10^{-10}$.]{ \label{fig:example6_a_minus_e_b_plus_e}
\includegraphics[width=6cm,height=5cm]{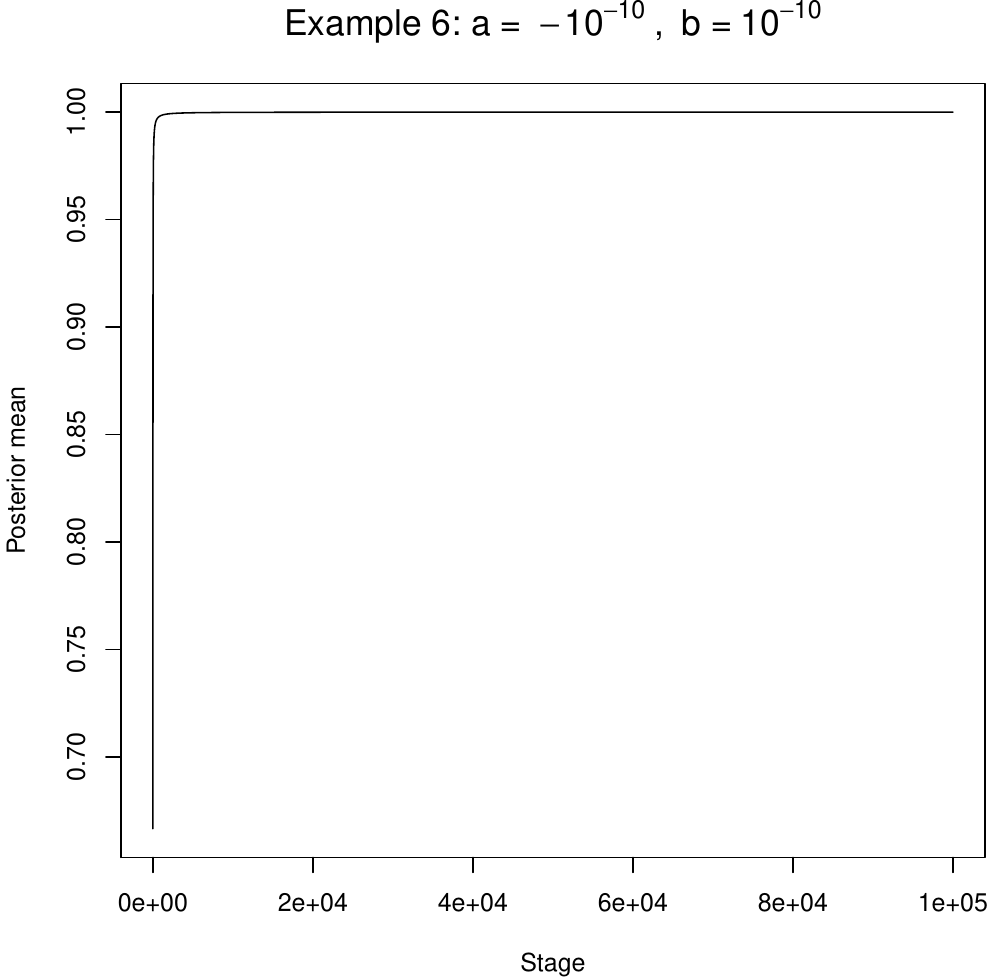}}\\
\subfigure [Convergence: $a=10^{-10},b=-10^{-10}$.]{ \label{fig:example6_a_plus_e_b_minus_e}
\includegraphics[width=6cm,height=5cm]{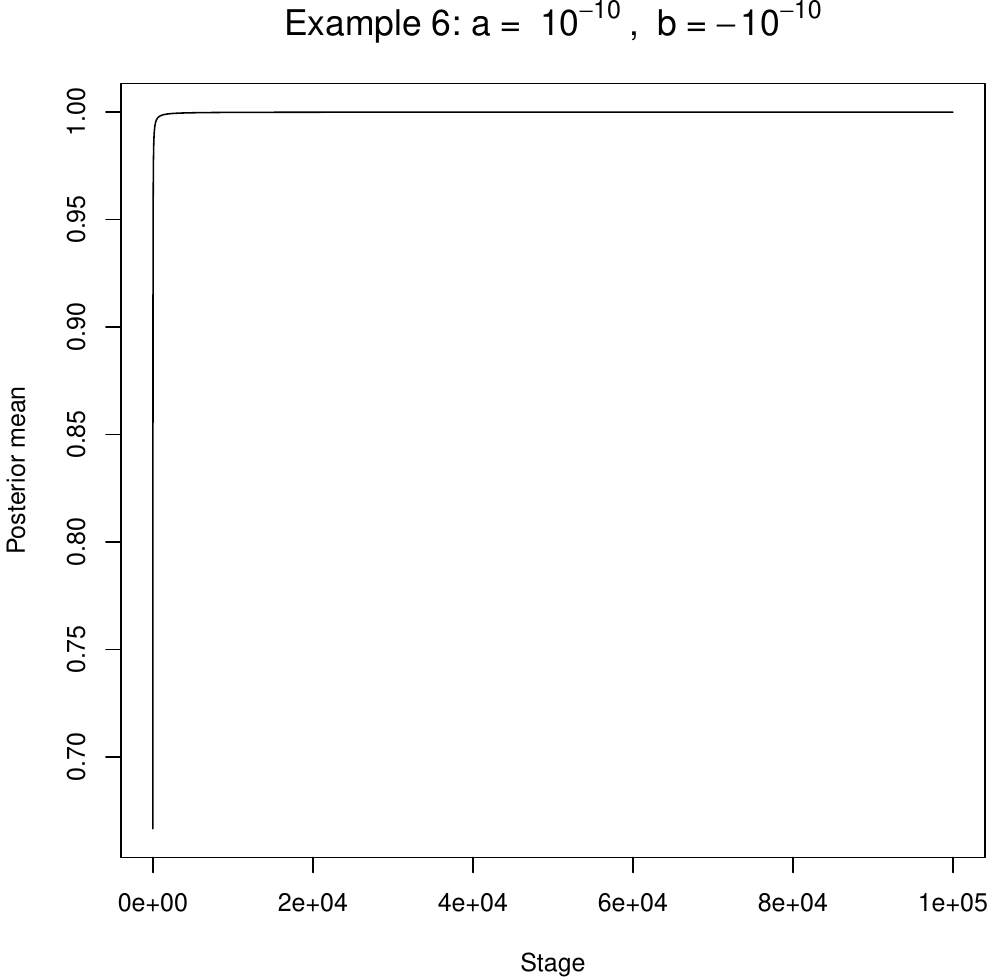}}
\hspace{2mm}
\subfigure [Convergence: $a=10^{-10},b=10^{-10}$.]{ \label{fig:example6_a_plus_e_b_plus_e}
\includegraphics[width=6cm,height=5cm]{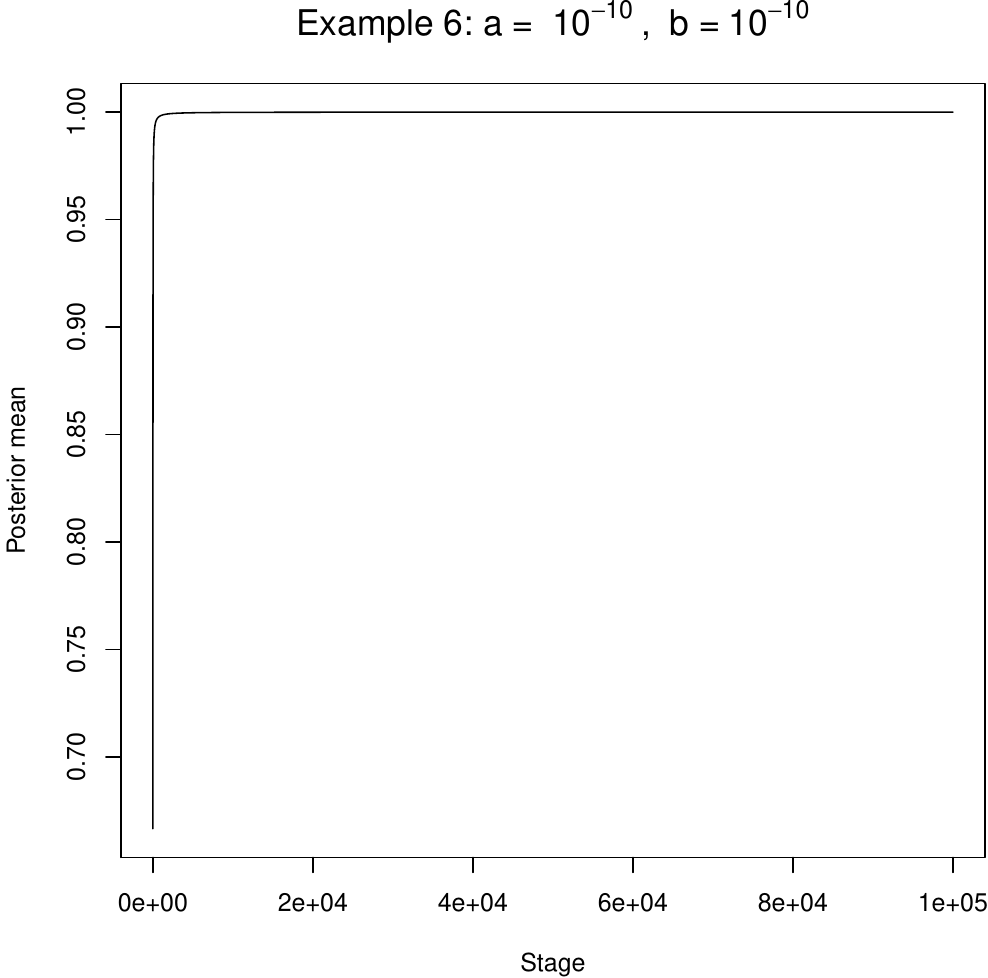}}\\
\subfigure [Convergence: $a=0,b=0$.]{ \label{fig:example6_a_0_b_0}
\includegraphics[width=6cm,height=5cm]{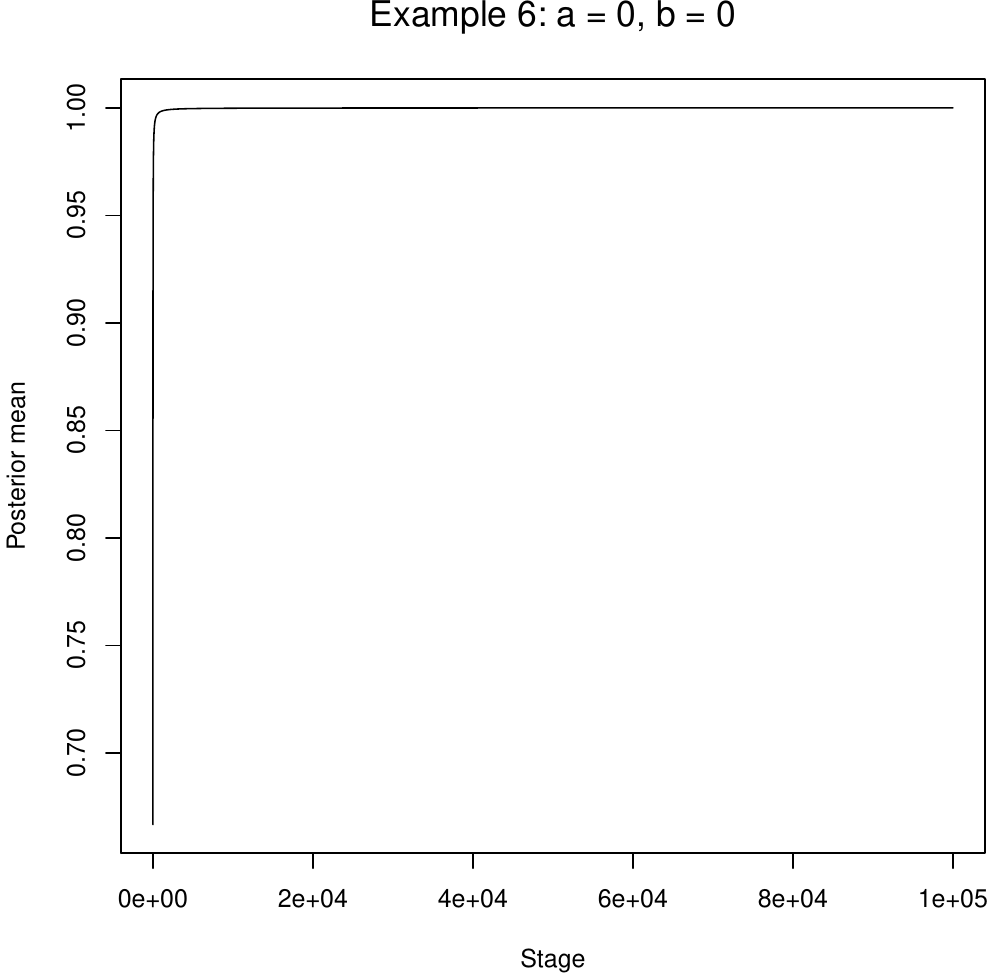}}
\caption{Example 6: The series (\ref{eq:example6}) converges for $(a=-10^{-10},b=-10^{-10})$, 
$(a=-10^{-10},b=10^{-10})$, $(a=10^{-10},b=-10^{-10})$, $(a=10^{-10},b=10^{-10})$, and
$(a=0,b=0)$.}
\label{fig:example6b}
\end{figure}

\subsection{Example 7}
\label{subsec:example7}
We now consider
\begin{equation}
S=\sum_{i=1}^{\infty}\frac{|\sin~i|^i}{i}.
\label{eq:example7}
\end{equation}
We embed this series into
\begin{equation}
S^{a,b}=\sum_{i=1}^{\infty}\frac{|\sin~a\pi i|^i}{i^b},
\label{eq:example7_embedding}
\end{equation}
where $a\in\mathbb R$ and $b\geq 1$.
The above series converges if $b>1$, for all $a\in\mathbb R$.
But for $b=1$, it is easy to see that 
the series diverges if $a=\ell/2m$, where
$\ell$ and $m$ are odd integers.

Letting $a_0=\pi^{-1}$ and $b_0=1+\epsilon$, with $\epsilon=10^{-10}$,
we set the following upper bound that is decreasing in $j$:
\begin{equation}
c^{a,b}_{j,n}=S^{a_0,b_0}_{j,n}+\frac{\epsilon}{j}.
\label{eq:example7_upper_bound}
\end{equation}
Thus, $c^{a,b}_{j,n}$ corresponds to a convergent series which is also sufficiently close
to divergence. Addition of the term $\frac{\epsilon}{j}$ provides further protection from 
erroneous conclusions regarding divergence.

Panel(a) of Figure \ref{fig:example7} demonstrates that the series of our interest, given by
(\ref{eq:example7}), diverges. Panel (b) confirms that for $a=5/(2\times 7)$ and $b=1$, the series
indeed diverges, as it should.
\begin{figure}
\centering
\subfigure [Divergence: $a=\pi^{-1},b=1$.]{ \label{fig:example7_a_pi_inv_b_1}
\includegraphics[width=6cm,height=5cm]{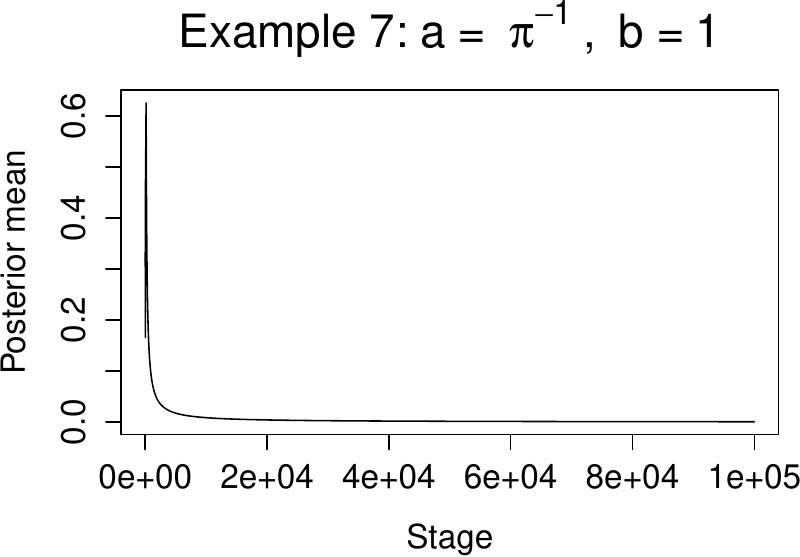}}
\hspace{2mm}
\subfigure [Divergence: $a=5/(2\times 7),b=1$.]{ \label{fig:example7_a_odd1_b_1}
\includegraphics[width=6cm,height=5cm]{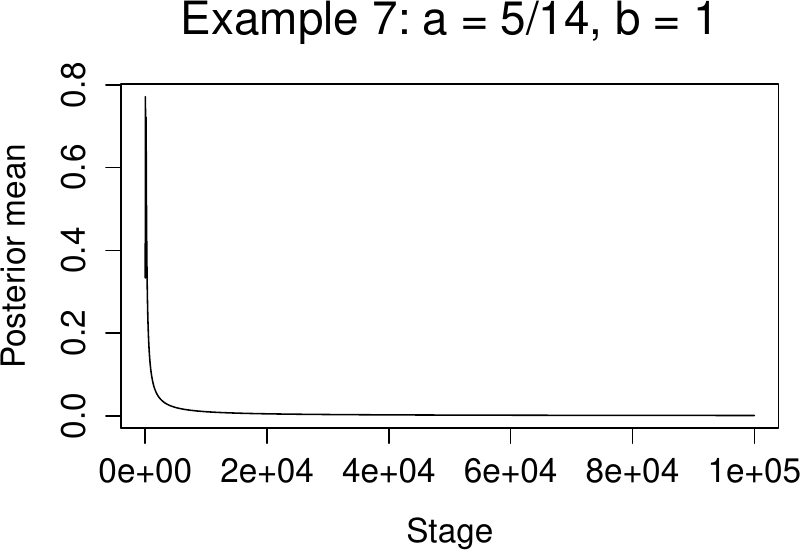}}
\caption{Example 7: The series (\ref{eq:example7_embedding}) diverges for $(a=\pi^{-1},b=1)$, 
$(a=5/7,b=1)$.} 
\label{fig:example7}
\end{figure}

\section{Oscillatory series with multiple limit points}
\label{sec:osc_mult}

In this section we assume that the sequence $\left\{S_{1,n}\right\}_{n=1}^{\infty}$ has multiple
limit points, including the possibility that the number of limit points is countably infinite.

\subsection{Finite number of limit points}
\label{eq:finite_limit_points}
Let us assume that there are $M~(>1)$ limit points of the sequence $\left\{S_{1,n}\right\}_{n=1}^{\infty}$. 
Then there exist sequences
$\{c_{m,j}\}_{j=1}^{\infty}$; $m=0,\ldots,M$, such that
$\left\{(c_{m-1,j},c_{m,j}];~m=1,\ldots,M\right\}$ 
partition the real line $\mathbb R$
for every $j\geq 1$ and that there exists $j_0\geq 1$ such that for all $j\geq j_0$,
the interval $(c_{m-1,j},c_{m,j}]$ contains at most one limit point of the sequence $\left\{S_{1,n}\right\}_{n=1}^{\infty}$,
for every $m=1,\ldots,M$. With these sequences we define
\begin{equation}
Y_{j}=m~~\mbox{if}~~c_{m-1,j}<S_{1,j}\leq c_{m,j};~m=1,2,\ldots,M,
\label{eq:y_finite}
\end{equation}

Recall that in Section 4 of our main manuscript 
we allowed the sequence $\{c_j\}_{j=1}^{\infty}$
to depend upon the underlying series $S_{1,\infty}$. 
Likewise, here also we allow the quantities $c_{0,j},c_{1,j},\ldots,c_{M,j}$ to depend upon 
$S_{1,\infty}$. In other words,
for $\omega\in\mathfrak S$, for $m=0,1,2,\ldots,M$, and $j=1,2,3,\ldots$, 
$c_{m,j}=c_{m,j}(\omega)$ corresponds to $S_{1,\infty}(\omega)$.

Note that unlike our ideas appropriate for non-oscillating series, here do not consider blocks
of partial sums, 
$S_{j,n_j}=\sum_{i=\sum_{k=0}^{j-1}n_k+1}^{\sum_{k=0}^jn_k}X_i$,
but $S_{1j}=\sum_{i=1}^jX_i$. In other words,
for Bayesian analysis of non-oscillating series we compute sums of $n_j$ terms in each iteration, 
whereas for oscillating series
we keep adding a single term at every iteration. Thus, computationally, the latter is a lot simpler.

We assume that
\begin{equation}
\left(\mathbb I(Y_{j}=1),\ldots,\mathbb I(Y_{j}=M)\right)
\sim Multinomial\left(1,p_{1,j},\ldots,p_{M,j}\right),
\label{eq:multinomial}
\end{equation}
where $p_{m,j}$ can be interpreted as the probability that $S_{1,j}\in (c_{m-1,j},c_{m,j}]$.
As $j\rightarrow\infty$ it is expected that $c_{m-1,j}$ and $c_{m,j}$ will converge to appropriate constants
depending upon $m$, and that $p_{m,j}$ will tend to the correct proportion of the limit point  
indexed by $m$. Indeed, let $\left\{p_{m,0};~m=1,\ldots,M\right\}$ denote the actual proportions
of the limit points indexed by $\left\{1,\ldots,M\right\}$, as $j\rightarrow\infty$.

Following the same principle discussed in Section 3 of our main manuscript, 
and extending the Beta prior to the Dirichlet prior, at the $k$-th stage we arrive at the
following posterior of $\left\{p_{m,k}:m=1,\ldots,M\right\}$:
\begin{equation}
\pi\left(p_{1,k},\ldots,p_{M,k}|y_{k}\right)
\equiv Dirichlet\left(\sum_{j=1}^k\frac{1}{j^2}+\sum_{j=1}^k\mathbb I\left(y_{j}=1\right),\ldots,
\sum_{j=1}^k\frac{1}{j^2}+\sum_{j=1}^k\mathbb I\left(y_{j}=M\right)\right).
\label{eq:posterior_dirichlet}
\end{equation}
The posterior mean and posterior variance of $p_{m,k}$, for $m=1,\ldots,M$, are given by:
\begin{align}
E\left(p_{m,k}|y_{k}\right)&=
\frac{\sum_{j=1}^k\frac{1}{j^2}+\sum_{j=1}^k\mathbb I\left(y_{j}=m\right)}
{M\sum_{j=1}^k\frac{1}{j^2}+k};
\label{eq:mean_dirichlet}\\
Var\left(p_{m,k}|y_{k}\right)&=
\frac{\left(\sum_{j=1}^k\frac{1}{j^2}+\sum_{j=1}^k\mathbb I\left(y_{j}=m\right)\right)
\left((M-1)\sum_{j=1}^k\frac{1}{j^2}+k-\sum_{j=1}^k\mathbb I\left(y_{j}=m\right)\right)}
{\left(M\sum_{j=1}^k\frac{1}{j^2}+k\right)^2\left(M\sum_{j=1}^k\frac{1}{j^2}+k+1\right)}.
\label{eq:var_dirichlet}
\end{align}
Let $k=M\tilde k$, where $\tilde k\rightarrow\infty$. Then,
from (\ref{eq:mean_dirichlet}) and (\ref{eq:var_dirichlet}) it is easily seen, using 
$\frac{\sum_{j=1}^k\mathbb I\left(y_{j}(\omega)=m\right)}{k}\rightarrow p_{m,0}$ as $k\rightarrow\infty$,
that,
\begin{align}
E\left(p_{m,k}|y_{k}\right)&\rightarrow p_{m,0},~~\mbox{and}
\label{eq:mean_dirichlet_convergence}\\
Var\left(p_{m,k}|y_{k}\right)&= O\left(\frac{1}{k}\right)\rightarrow 0,
\label{eq:var_dirichlet_convergence}
\end{align}
as $k\rightarrow\infty$.

We can now characterize the $m$ limit points of
$S_{1,\infty}(\omega)$ in terms of the limits of the marginal posterior
probabilities of $p_{m,k}$, denoted by $\pi_m\left(\cdot|y_k(\omega)\right)$, 
as $k\rightarrow\infty$.
\begin{theorem}
\label{theorem:finite_limit_points}
For $\omega\in\mathfrak S\cap \mathfrak N^c$, where $\mathfrak N$ has zero probability measure, 
$\left\{S_{1,n}(\omega)\right\}_{n=1}^{\infty}$ has $M~(>1)$ limit points almost surely if and only if 
\begin{itemize}
\item[(1)] There exist sequences $\{c_{m,j}(\omega)\}_{j=1}^{\infty}$; $m=0,\ldots,M$, such that 
$(c_{m-1,j}(\omega),c_{m,j}(\omega)]$ partition the real line $\mathbb R$ for every $j\geq 1$ and $m=1,\ldots,M$.
\item[(2)] There exists $j_0(\omega)\geq 1$ such that for all $j\geq j_0(\omega)$, for $m=1,\ldots,M$, 
$(c_{m-1,j}(\omega),c_{m,j}(\omega)]$ contains at most one
limit point of $\{S_{1,n}(\omega)\}_{n=1}^{\infty}$.
\item[(3)] With $Y_j$ defined as in (\ref{eq:y_finite}),
\begin{equation}
\pi_m\left(\mathcal N_{p_{m,0}}|y_{k}(\omega)\right)\rightarrow 1,
\label{eq:consistency_at_p_m_0}
\end{equation}
as $k\rightarrow\infty$. 
In the above, $\mathcal N_{p_{m,0}}$ is any neighborhood of $p_{m,0}$, with $p_{m,0}$ satisfying
$0<p_{m,0}<1$ for $m=1,\ldots,M$ such that $\sum_{m=1}^Mp_{m,0}=1$.
\end{itemize}
\end{theorem}
\begin{proof}
For $\omega\in\mathfrak S\cap\mathfrak N^c$, where $\mathfrak N$ has zero probability measure, 
let $S_{1,\infty}(\omega)$ be oscillatory with $M$ limit points having proportions
$\left\{p_{m,0};~m=1,\ldots,M\right\}$. Conditions (1) and (2) then clearly hold.
Then with our definition of $Y_j$ provided in (\ref{eq:y_finite}), 
the results (\ref{eq:mean_dirichlet_convergence}) and (\ref{eq:var_dirichlet_convergence}) hold 
with $k=M\tilde k$, where $\tilde k\rightarrow\infty$.
Now let $\mathcal N_{p_{m,0}}$ be any neighborhood of $p_{m,0}$.
Let $\epsilon>0$ be sufficiently small so that 
$\mathcal N_{p_{m,0}}\supseteq\left\{|p_{m,k}-p_{m,0}|<\epsilon\right\}$. Then by Chebychev's inequality,
using (\ref{eq:mean_dirichlet_convergence}) and (\ref{eq:var_dirichlet_convergence}), it is seen that
$\pi_m\left(\mathcal N_{p_{m,0}}|y_k(\omega)\right)\rightarrow 1$, as $k\rightarrow\infty$.
Thus, (\ref{eq:consistency_at_p_m_0}) holds. In fact, more generally, condition (3) holds.

Now assume that conditions (1), (2), (3) hold. 
Then $\pi_m\left(|p_{m,k}-p_{m,0}|<\epsilon|y_k(\omega)\right)\rightarrow 1$, as $k\rightarrow\infty$.
Combining this with Chebychev's inequality it follows that 
(\ref{eq:mean_dirichlet_convergence}) and (\ref{eq:var_dirichlet_convergence}) hold with
$0<p_{m,0}<1$ for $m=1,\ldots,M$ such that $\sum_{m=1}^Mp_{m,0}=1$. If $\left\{S_{1,n}(\omega)\right\}_{n=1}^{\infty}$ has
less than $M$ limit points, then at least one $p_{m,0}=0$, providing a contradiction.
Hence $\left\{S_{1,n}(\omega)\right\}_{n=1}^{\infty}$ must have $M$ limit points.
\end{proof}

\subsection{Choice of $c_{0,j},\ldots,c_{M,j}$ for a given series}
\label{subsec:choice_c}

Let us define, for $j=1,2,\ldots,k$,
\begin{align}
\tilde p_{\ell,j}&=\left\{\begin{array}{ccc}
0 & \mbox{if} & \ell=0;\\
E\left(p_{\ell,j}|y_{j}\right) & \mbox{if} & \ell=1,2,\ldots,M.
\end{array}\right. 
\label{eq:osc_recursive_postmean1}
\end{align}
We also define, for $\ell=1,2,\ldots,M$,
\begin{equation}
\tilde p_{\ell,0}=E\left(p_{\ell,1}\right),
\label{eq:osc_priormean}
\end{equation}
the prior mean at the first stage, before observing any data.

We then set $c_{0,j}\equiv 0$ for all $j=1,2,\ldots,k$, and, for $m\geq 1$, define
\begin{equation}
c_{m,j}=\log\left[\frac{\left(\sum_{\ell=1}^m\tilde p_{\ell,{j-1}}\right)^{1/\rho(\theta)}}
{1-\left(\sum_{\ell=1}^m\tilde p_{\ell,{j-1}}\right)^{1/\rho(\theta)}}\right],
\label{eq:c_finite}
\end{equation}
for $j=1,2,\ldots,k$.
Thus, the inequality $c_{m-1,j}<S_{1,j}\leq c_{m,j}$ in (\ref{eq:y_finite}) is equivalent to
\begin{equation}
\sum_{\ell=1}^{m-1}\tilde p_{\ell,k}<\left(\frac{\exp\left(S_{1,j}\right)}{1+\exp\left(S_{1,j}\right)}\right)^{\rho(\theta)}
\leq \sum_{\ell=1}^m\tilde p_{\ell,k},
\label{eq:c_finite2}
\end{equation}
where $\rho(\theta)$ is some relevant power depending upon the set of parameters $\theta$ of the given series,
responsible for appropriately
inflating or contracting the quantity $\frac{\exp\left(S_{1,j}\right)}{1+\exp\left(S_{1,j}\right)}$ 
for properly diagnosing the limit points. Thus, given the series $S_{1,\infty}(\omega)$, $\theta=\theta(\omega)$
is allowed to depend upon the underlying series.
If $\left(\frac{\exp\left(S_{1,j}\right)}{1+\exp\left(S_{1,j}\right)}\right)^{\rho(\theta)}\geq 1$, we set $Y_j=M$.
By (\ref{eq:consistency_at_p_m_0}), for large $k$, $\tilde p_{\ell,k}$ and $S_{1,j}$ adaptively adjust
themselves so that the correct proportions of the limit points are achieved in the long run.

\subsection{Infinite number of limit points}
\label{subsec:infinite_limit_points}

We now assume that the number of limits points of $\left\{S_{1,n}(\omega)\right\}_{n=1}^{\infty}$ is countably infinite,
and that $\left\{p_{m,0};m=1,2,3,\ldots\right\}$, where $0\leq p_{m,0}\leq 1$ and $\sum_{m=1}^{\infty}p_{m,0}=1$,
are the true proportions of the limit points.

Now we define
\begin{equation}
Y_{j}=m~~\mbox{if}~~c_{m-1,j}<S_{1,j}\leq c_{m,j};~m=1,2,\ldots,\infty,
\label{eq:y_infinite}
\end{equation}
where the sequences $\left\{c_{m,j}\right\}_{j=1}^{\infty}$; $m\geq 1$, are such that 
$(c_{m-1,j},c_{m,j}]$; $m\geq 1$, partition $\mathbb R$ for every $j\geq 1$, and that there
exists $j_0\geq 1$ such that for all $j\geq j_0$, these intervals contain at most one limit point of
$\left\{S_{1,n}\right\}_{n=1}^{\infty}$.

Let $\mathcal X=\left\{1,2,\ldots\right\}$ and let $\mathcal B\left(\mathcal X\right)$ denote the Borel
$\sigma$-field on $\mathcal X$ (assuming every singleton of $\mathcal X$ is an open set). Let $\mathcal P$
denote the set of probability measures on $\mathcal X$. Then, at the $j$-th stage,
\begin{equation}
[Y_j|P_j]\sim P_j,
\label{eq:Y_DP}
\end{equation}
where $P_j\in\mathcal P$. We assume that
$P_j$ is the following Dirichlet process (see \ctn{Ferguson73}):
\begin{equation}
P_j\sim DP\left(\frac{1}{j^2}G\right),
\label{eq:DP}
\end{equation}
where, the probability measure $G$ is such that, for every $j\geq 1$,
\begin{equation}
G\left(Y_j=m\right)=\frac{1}{2^m}.
\label{eq:G}
\end{equation}
It then follows using the same previous principles that, at the $k$-th stage, 
the posterior of $P_k$ is again a Dirichlet process, given by
\begin{equation}
[P_k|y_k]\sim DP\left(\sum_{j=1}^k\frac{1}{j^2}G+\sum_{j=1}^k\delta_{y_j}\right),
\label{eq:posterior_DP}
\end{equation}
where $\delta_{y_j}$ denotes point mass at $y_j$.
It follows from (\ref{eq:posterior_DP}) that
\begin{align}
E\left(p_{m,k}|y_{k}\right)&=
\frac{\frac{1}{2^m}\sum_{j=1}^k\frac{1}{j^2}+\sum_{j=1}^k\mathbb I\left(y_{j}=m\right)}
{\sum_{j=1}^k\frac{1}{j^2}+k};
\label{eq:mean_DP}\\
Var\left(p_{m,k}|y_{k}\right)&=
\frac{\left(\sum_{j=1}^k\frac{1}{j^2}+\sum_{j=1}^k\mathbb I\left(y_{j}=m\right)\right)
\left((1-\frac{1}{2^m})\sum_{j=1}^k\frac{1}{j^2}+k-\sum_{j=1}^k\mathbb I\left(y_{j}=m\right)\right)}
{\left(\sum_{j=1}^k\frac{1}{j^2}+k\right)^2\left(\sum_{j=1}^k\frac{1}{j^2}+k+1\right)}.
\label{eq:var_DP}
\end{align}
As before, it easily follows from (\ref{eq:mean_DP}) and (\ref{eq:var_DP}) that
for $m=1,2,3,\ldots$, 
\begin{align}
E\left(p_{m,k}|y_{k}\right)&\rightarrow p_{m,0},~~\mbox{and}
\label{eq:mean_DP_convergence}\\
Var\left(p_{m,k}|y_{k}\right)&= O\left(\frac{1}{k}\right)\rightarrow 0,
\label{eq:var_DP_convergence}
\end{align}
almost surely, as $k\rightarrow\infty$.

The theorem below characterizes countable number of limit points of
$S_{1,\infty}$ in terms of the limit of the marginal posterior
probabilities of $p_{m,k}$, 
as $k\rightarrow\infty$. 
\begin{theorem}
\label{theorem:infinite_limit_points}
For $\omega\in\mathfrak S\cap \mathfrak N^c$, where $\mathfrak N$ has zero probability measure, 
$\left\{S_{1,n}(\omega)\right\}_{n=1}^{\infty}$ has countable limit points almost surely if and only if 
\begin{itemize}
\item[(1)] There exist sequences $\{c_{m,j}(\omega)\}_{j=1}^{\infty}$; $m=0,1,2\ldots$, such that 
$(c_{m-1,j}(\omega),c_{m,j}(\omega)]$ partition the real line $\mathbb R$ for every $j\geq 1$ and $m\geq 1$.
\item[(2)] There exists $j_0(\omega)\geq 1$ such that for all $j\geq j_0(\omega)$,  
$(c_{m-1,j}(\omega),c_{m,j}(\omega)]$ contains at most one
limit point of $\{S_{1,n}(\omega)\}_{n=1}^{\infty}$, for every $m\geq 1$.
\item[(3)] With $Y_j$ defined as in (\ref{eq:y_infinite}),
\begin{equation}
\pi_m\left(\mathcal N_{p_{m,0}}|y_{k}(\omega)\right)\rightarrow 1,
\label{eq:consistency_at_p_m_0_DP}
\end{equation}
as $k\rightarrow\infty$. 
In the above, $\mathcal N_{p_{m,0}}$ is any neighborhood of $p_{m,0}$, with $p_{m,0}$ satisfying
$0\leq p_{m,0}\leq 1$ for $m=1,2,\ldots$ such that $\sum_{m=1}^{\infty}p_{m,0}=1$, with at most finite number
of $m$ such that $p_{m,0}=0$.
\end{itemize}
\end{theorem}
\begin{proof}
Follows using the same ideas as the proof of Theorem \ref{theorem:finite_limit_points}.
\end{proof}

As regards the choice of the quantities $c_{m,j}$, we simply extend the construction detailed 
in Section \ref{subsec:choice_c} by only letting $M\rightarrow\infty$, and with obvious replacement
of the posterior means with those associated with the posterior Dirichlet process.

It is useful to remark that our theory with countably infinite number of limit points is readily
applicable to 
situations where the number of limit points
is finite but unknown. In such cases, only a finite number of the probabilities $\left\{p_{m,j};~m=1,2,3\ldots\right\}$ 
will have posterior probabilities around positive quantities, while the rest will concentrate around zero.
For known finite number of limit points, it is only required to specify $G$ such that it 
gives positive mass to only a specific finite set.

\subsection{Characterization of convergence and divergence with our approach on limit points}
\label{subsec:discussion_limpoints}
Note that for convergent series, $\pi_m\left(\mathcal N_1|y_k(\omega)\right)\rightarrow 1$ as $k\rightarrow\infty$
for smaller values of $m$, while for divergent series with $S_{1,\infty}(\omega)=\infty$
or $S_{1,\infty}(\omega)=-\infty$, $\pi_m\left(\mathcal N_1|y_k(\omega)\right)\rightarrow 1$ as $k\rightarrow\infty$
for much larger values of $m$ and the smallest value of $m$, respectively.
We formalize these statements below as the following theorems.

\begin{theorem}
\label{theorem:divergence_limpoints}
Let there be $M$ number of possible limit points of $S_{1,\infty}(\omega)$, where
$M$ may be infinite. 
Then for any $\omega\in\mathfrak S\cap\mathfrak N^c$, where $\mathfrak N$ has zero probability measure,
$S_{1,\infty}(\omega)=\infty$ if and only if, 
for any sequences $\{c_{m,j}(\omega)\}_{j=1}^{\infty}$;
$m=1,2,\ldots,M$, such that $(c_{m-1,j}(\omega),c_{m,j}(\omega)]$; $m=1,\ldots,M$, partitions the real line $\mathbb R$
for every $j\geq 1$, it holds that
\begin{equation}
\pi_{m,k}\left(\mathcal N_{1}|y_k(\omega)\right)\rightarrow 1,
\label{eq:post_limpoints1}
\end{equation}
as $k\rightarrow\infty$ and $m\rightarrow M$. 
\end{theorem}
\begin{proof}
For $\omega\in\mathfrak S\cap\mathfrak N^c$, where $\mathfrak N$ has zero probability measure, 
let $S_{1,\infty}(\omega)=\infty$. 
Then as $k\rightarrow\infty$, 
\begin{equation}
\left(\frac{\exp\left(S_{1,k}(\omega)\right)}{1+\exp\left(S_{1,k}(\omega)\right)}\right)^{\rho(\theta(\omega))}\rightarrow 1.
\label{eq:S_divergence1}
\end{equation}
In other words, for any fixed $M~(>1)$, $y_k(\omega)\rightarrow M$, as $k\rightarrow\infty$. Hence, as
$k\rightarrow\infty$ and $m\rightarrow M$, 
it easily follows using the same techniques as before, that (\ref{eq:post_limpoints1}) holds. 
Consequently, for infinite number of limit points, (\ref{eq:post_limpoints1}) holds as $m\rightarrow\infty$.

Now assume that (\ref{eq:post_limpoints1}) holds. It then follows from the formula of the posterior
mean that $y_k(\omega)\rightarrow M$, as $k\rightarrow\infty$, for fixed $M$.
Hence, (\ref{eq:S_divergence1}) holds, from which it follows that $S_{1,\infty}(\omega)=\infty$.
\end{proof}

\begin{theorem}
\label{theorem:divergence_limpoints_negative}
Let there be $M$ number of possible limit points of $S_{1,\infty}(\omega)$, where
$M$ may be infinite. 
Then for any $\omega\in\mathfrak S\cap\mathfrak N^c$, where $\mathfrak N$ has zero probability measure,
$S_{1,\infty}(\omega)=-\infty$ almost surely if and only if 
for any sequences $\{c_{m,j}(\omega)\}_{j=1}^{\infty}$;
$m=1,2,\ldots,M$, such that $(c_{m-1,j}(\omega),c_{m,j}(\omega)]$; $m=1,\ldots,M$, partitions the real line $\mathbb R$
for every $j\geq 1$, it holds that
\begin{equation}
\pi_{m,k}\left(\mathcal N_{1}|y_k(\omega)\right)\rightarrow 1,
\label{eq:post_limpoints1_negative}
\end{equation}
as $k\rightarrow\infty$ and $m\rightarrow 1$. 
\end{theorem}
\begin{proof}
For $\omega\in\mathfrak S\cap\mathfrak N^c$, where $\mathfrak N$ has zero probability measure, 
let $S_{1,\infty}(\omega)=-\infty$. 
Then as $k\rightarrow\infty$, 
\begin{equation}
\left(\frac{\exp\left(S_{1,k}(\omega)\right)}{1+\exp\left(S_{1,k}(\omega)\right)}\right)^{\rho(\theta(\omega))}\rightarrow 0.
\label{eq:S_divergence1_negative}
\end{equation}
In other words, for any fixed $M~(>1)$, $y_k(\omega)\rightarrow 1$, as $k\rightarrow\infty$. Hence, as
$k\rightarrow\infty$ and $m\rightarrow 1$, it is easily seen that (\ref{eq:post_limpoints1_negative}) holds. 

Also, if (\ref{eq:post_limpoints1_negative}) holds, then it follows from the formula of the posterior
mean that $y_k(\omega)\rightarrow 1$, as $k\rightarrow\infty$.
Hence, (\ref{eq:S_divergence1_negative}) holds, from which it follows that $S_{1,\infty}(\omega)=-\infty$.
\end{proof}

\begin{theorem}
\label{theorem:convergence_limpoints}
For $\omega\in\mathfrak S\cap\mathfrak N^c$, where $\mathfrak N$ has zero probability measure,
$S_{1,\infty}(\omega)$ is convergent if and only if 
for any sequences $\{c_{m,j}(\omega)\}_{j=1}^{\infty}$;
$m=1,2,\ldots,M$, such that $(c_{m-1,j}(\omega),c_{m,j}(\omega)]$; $m=1,\ldots,M$, partitions the real line $\mathbb R$
for every $j\geq 1$, it holds 
for some finite $m_0(\omega)\geq 1$, that
\begin{equation}
\pi_{m_0(\omega),k}\left(\mathcal N_{1}|y_k(\omega)\right)\rightarrow 1,
\label{eq:post_limpoints2}
\end{equation}
as $k\rightarrow\infty$. 
\end{theorem}
\begin{proof}
Let $S_{1,\infty}(\omega)$ be convergent. 
Then as $k\rightarrow\infty$, 
\begin{equation}
\left(\frac{\exp\left(S_{1,k}(\omega)\right)}{1+\exp\left(S_{1,k}(\omega)\right)}\right)^{\rho(\theta(\omega))}\rightarrow 
c(\omega),
\label{eq:S_convergence1}
\end{equation}
for some constant $0\leq c(\omega)<1$. Hence, there exists some finite $m_0(\omega)\geq 1$ such that 
$y_k(\omega)\rightarrow m_0(\omega)$, as $k\rightarrow\infty$.  
Using the same techniques as before, it is seen that that (\ref{eq:post_limpoints2}) holds. 

Now assume that (\ref{eq:post_limpoints2}) holds. It then follows from the formula of the posterior
mean, that $y_k(\omega)\rightarrow m_0(\omega)$, as $k\rightarrow\infty$.
Hence, (\ref{eq:S_convergence1}) holds, from which it follows that $S_{1,\infty}(\omega)$ is convergent.
\end{proof}

According to Theorems \ref{theorem:divergence_limpoints_negative} and \ref{theorem:convergence_limpoints},
$m$ tends to $1$ and a finite quantity greater than or equal to $1$, accordingly as the series diverges
to $-\infty$ or converges. If the finite quantity in the latter case turns out to be $1$, then it is not
possible to distinguish between convergence and divergence to $-\infty$ by this method. However, 
Theorem 4.1 of our main manuscript 
can be usefully exploited in this case. If this method based on oscillating
series yields $m=1$, then we suggest checking for convergence using Theorem 4.1, 
which would then help us confirm if the series is truly convergent. 

\subsection{A rule of thumb for diagnosis of convergence, divergence and oscillations}
\label{subsec:thumb_rule}

Based on the above theorems we propose the following rule of thumb for detecting convergence and divergence
when $M$ is finite:
if $\frac{m}{M}> 0.9$ such that $\pi_{m,k}\left(\mathcal N_{1}|y_k(\omega)\right)\rightarrow 1$ 
as $k\rightarrow\infty$, then declare the series as divergent to $\infty$. If 
$0.1<\frac{m}{M}\leq 0.9$ such that $\pi_{m,k}\left(\mathcal N_{1}|y_k(\omega)\right)\rightarrow 1$, then declare the
series as convergent. On the other hand, if $\frac{m}{M}\leq 0.1$, use Theorem 4.1 
to check for convergence; in the case of negative result, declare the series as divergent to $-\infty$.

If, instead, there exist $m_{\ell};~\ell=1,\ldots,L$ ($L>1$) such that 
$\pi_{m_{\ell},k}\left(\mathcal N_{p_{m_{\ell},0}}|y_k(\omega)\right)\rightarrow 1$ as
$k\rightarrow\infty$, where $0<p_{m_{\ell},0}<1$ for $\ell=1,\ldots,L$ and $\sum_{\ell=1}^Lp_{m_{\ell},0}=1$, then 
say that the sequence $\left\{S_{1,n}(\omega)\right\}_{n=1}^{\infty}$ has $L$ limit points.

\section{Illustration of our Bayesian theory on oscillation}
\label{sec:osc_examples}
We first consider a simple oscillatory series to illustrate our Bayesian idea on detection of limit points
(Section \ref{subsec:osc_series1}).
Next, in Section \ref{subsec:Bayesian_limpoints}, we illustrate our theory on limit points with Example 5, 
arguably the most complex series in our
set of examples (other than Riemann Hypothesis) and in Section \ref{sec:Bayesian_limpoints_RH},
validate our result on Riemann Hypothesis with our Bayesian limit point theory. 

\subsection{Illustration with a simple oscillatory series}
\label{subsec:osc_series1}
Let us re-consider the series $S_{1,\infty}(\omega)=\sum_{i=1}^{\infty}\left(-1\right)^{i-1}$,
which we already introduced after Theorem 4.2 of our main manuscript. 
We consider the theory based on Dirichlet process
developed in Section \ref{subsec:infinite_limit_points}, assuming for the sake
of illustrations that $G$ is concentrated on $M$ values, with $G\left(Y_j=m\right)=\frac{1}{M}$;
$m=1,2,\ldots,M$. We set $M=10$ and $K=10^5$ for our experiments.
With $\rho(\theta)=2$, the results are depicted
in Figure \ref{fig:osc_series1}. Two explicit limit points, with proportions $0.5$ each, are 
correctly recognized. The limit points are obviously $0$ and $1$ for this example.
Implementation takes just a fraction of a second, even on an ordinary 32-bit laptop.
\begin{figure}
\centering
\subfigure [First limit point: The posterior of $p_{5,k}$ converges to $0.5$ as $k\rightarrow\infty$]
{\label{fig:series1_1}
\includegraphics[width=6cm,height=5cm]{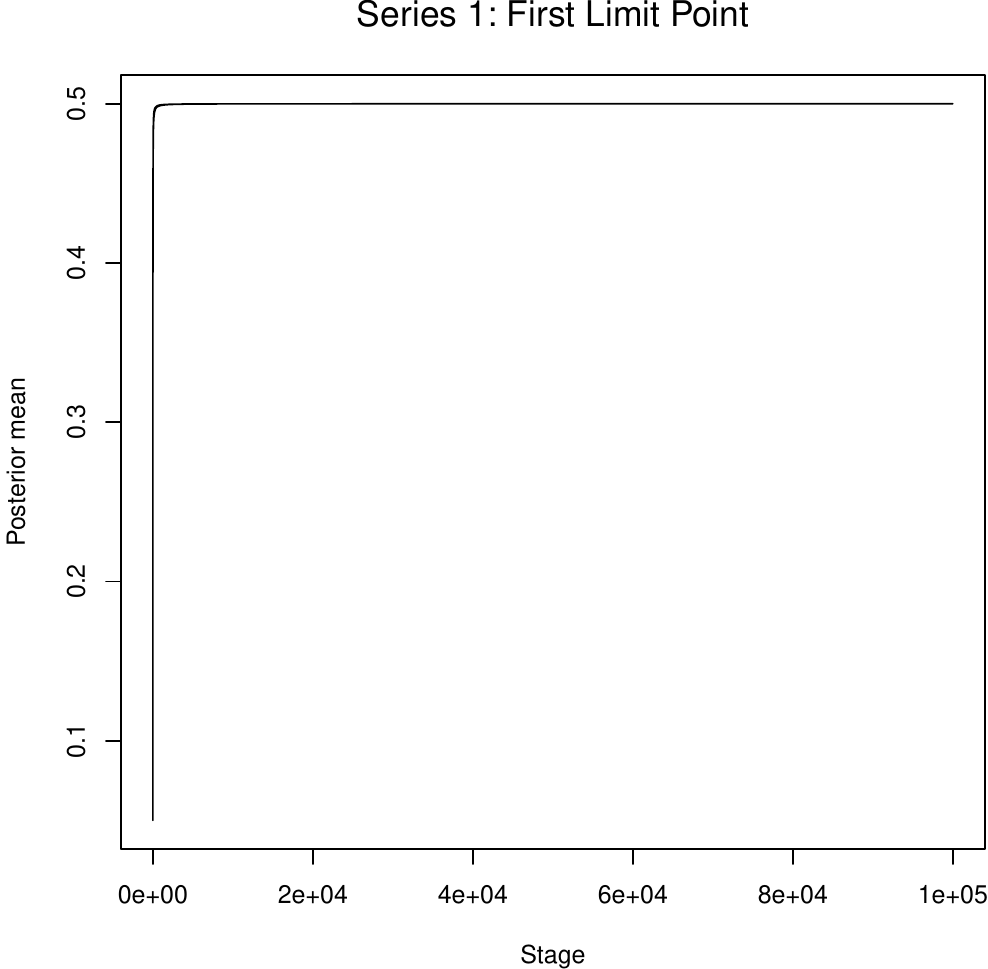}}
\hspace{2mm}
\subfigure [Second limit point: The posterior of $p_{6,k}$ converges to $0.5$ as $k\rightarrow\infty$.]
{ \label{fig:series1_2}
\includegraphics[width=6cm,height=5cm]{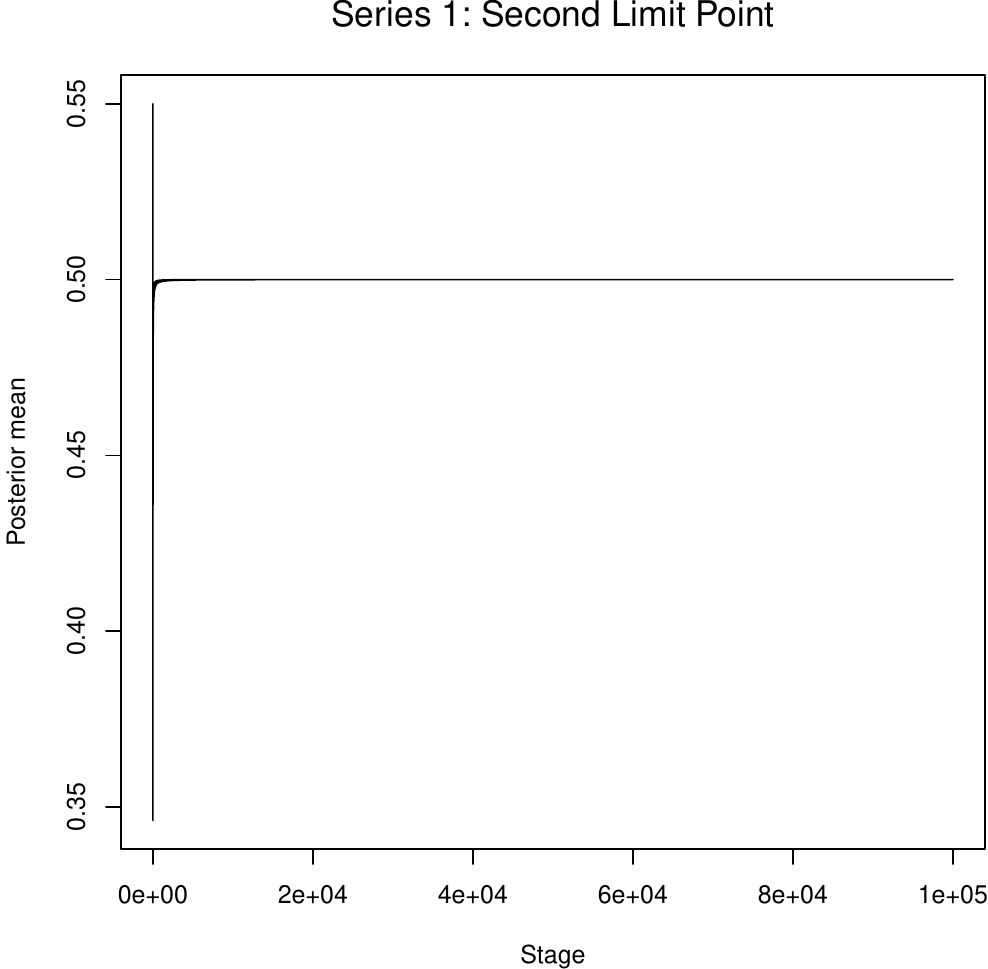}}
\caption{Illustration of the Dirichlet process based theory on the first oscillating series:
two limit points, each with proportion $0.5$, are captured.} 
\label{fig:osc_series1}
\end{figure}

\subsection{Illustration of the Bayesian limit point theory with Example 5}
\label{subsec:Bayesian_limpoints}

Since there is at most one limit point in the cases that we investigated, application
of our ideas to these cases must be able to re-confirm this. 
As before we consider the theory based on Dirichlet process
with $G\left(Y_j=m\right)=\frac{1}{M}$;
$m=1,2,\ldots,M$, where we set $M=10$.
Thus, by our rule of thumb, divergence is to be declared only if 
$\pi_{m=10,k}\left(\mathcal N_{1}|y_k\right)\rightarrow 1$, as $k\rightarrow\infty$. 

As regards implementation, notice that here there is no scope for parallelization since at the $j$-th step
only $y_j$ is added to the existing $S_{1,j-1}$ to form $S_{1,j}=S_{1,j-1}+y_j$. As such, on our VMware,
using a single processor, only about two seconds are required for $10^5$ iterations associated with the series
(\ref{eq:example5}), for various values of $a~(>0)$ and $b~(>0)$.

\subsubsection{Choice of $\rho(\theta)$ in 
$\left(\frac{\exp\left(S_{1,k}\right)}{1+\exp\left(S_{1,k}\right)}\right)^{\rho(\theta)}$}
\label{subsubsec:rho_choice}
In our example, $\theta=(a,b)$. We choose, for $j\geq 1$, 
\begin{equation}
\tilde\rho(\theta)=a-b+\epsilon, 
\label{eq:rho_choice}
\end{equation}
and set
\begin{equation}
\left(\frac{\exp\left(S_{1,j}\right)}{1+\exp\left(S_{1,j}\right)}\right)^{\rho(\theta)}
=\min\left\{1,\left(\frac{\exp\left(S_{1,j}\right)}{1+\exp\left(S_{1,j}\right)}\right)^{\tilde\rho(\theta)}\right\}
\label{eq:S_limpoints}
\end{equation}
Recall that the series (\ref{eq:example5}), defined for $a>0$ and $b>0$, converges for $a-b>1$ and diverges for $a+b<1$. 
In keeping with this result,
(\ref{eq:S_limpoints}) decreases as $(a-b)$ increases, so that the chance of correctly diagnosing convergence 
increases. Moreover, if both $a$ and $b$ are between 0 and 1 such that $a+b<1$, then (\ref{eq:S_limpoints}) tends
to be inflated, thereby increasing the chance of correctly detecting divergence.
The term $\epsilon$ in (\ref{eq:S_limpoints}) prevents the power from becoming zero when $a=b$. It is important
to note here that for $a+b=1$ convergence or divergence is not guaranteed, but if $\epsilon=0$ in (\ref{eq:S_limpoints}), 
then $a=b$ would trivially indicate divergence, even if the series is actually convergent. A positive
value of $\epsilon$ provides protection from such erroneous decision.
Note that if $a<b-\epsilon$, the convergence criterion $a-b>1$ is not met but the divergence criterion
$a+b<1$ may still be satisfied. Thus, for such instances, greater weight in favour of divergence is indicated.
In our illustration, we set $\epsilon=10^{-10}$.

\subsubsection{Results}
\label{subsubsec:results}
Figure \ref{fig:limpoints_example5} shows the results of our Bayesian analysis of the series (\ref{eq:example5})
based on our Dirichlet process model. Based on the rule of thumb proposed in Section \ref{subsec:thumb_rule}
all the results are in agreement with the results based on Figure \ref{fig:example5}.
\begin{figure}
\centering
\subfigure [Convergence: $a=2,b=1$. The posterior of $p_{6,k}$ converges to 1 as $k\rightarrow\infty$]
{\label{fig:limpoints_a_2_b_1}
\includegraphics[width=6cm,height=5cm]{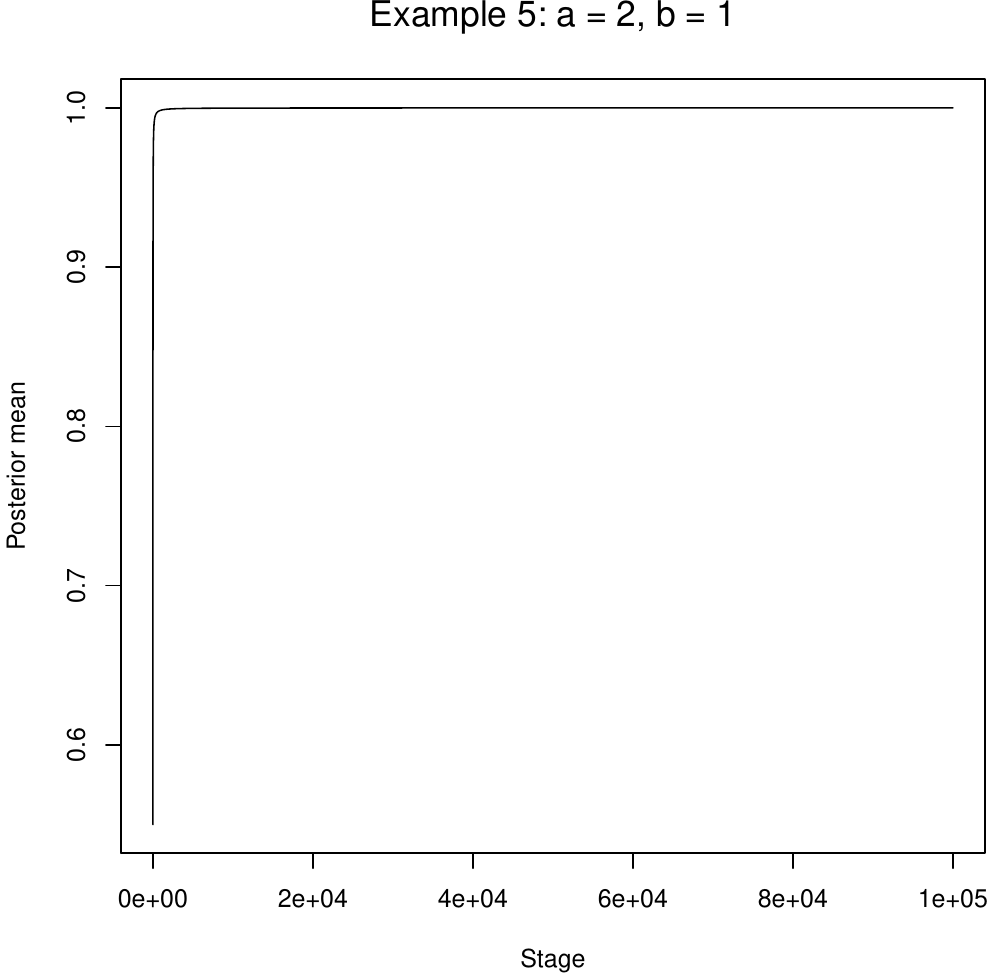}}
\hspace{2mm}
\subfigure [Convergence: $a=1+20^{-10},b=10^{-10}$.
The posterior of $p_{6,k}$ converges to 1 as $k\rightarrow\infty$.]
{ \label{fig:limpoints_a12_b01}
\includegraphics[width=6cm,height=5cm]{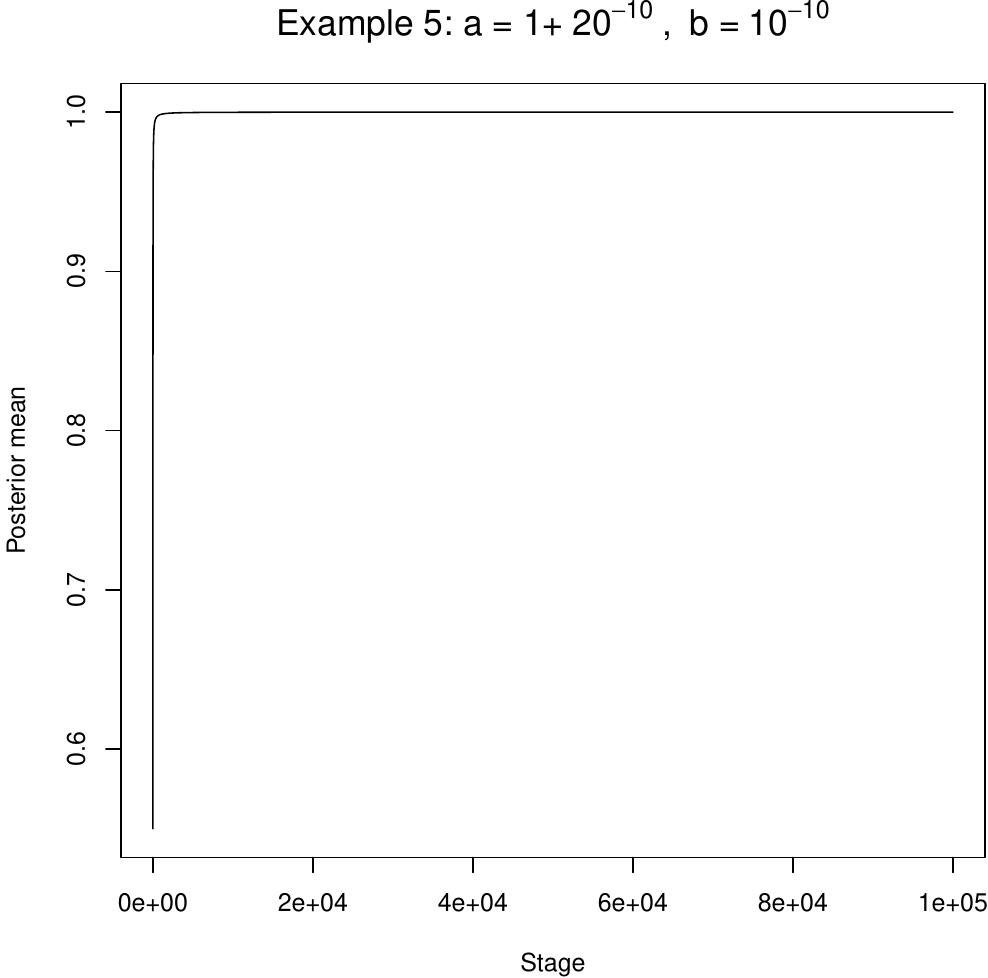}}\\
\hspace{2mm}
\subfigure [Convergence: $a=1+30^{-10},b=20^{-10}$.
The posterior of $p_{6,k}$ converges to 1 as $k\rightarrow\infty$.]
{\label{fig:limpoints_a13_b02}
\includegraphics[width=6cm,height=5cm]{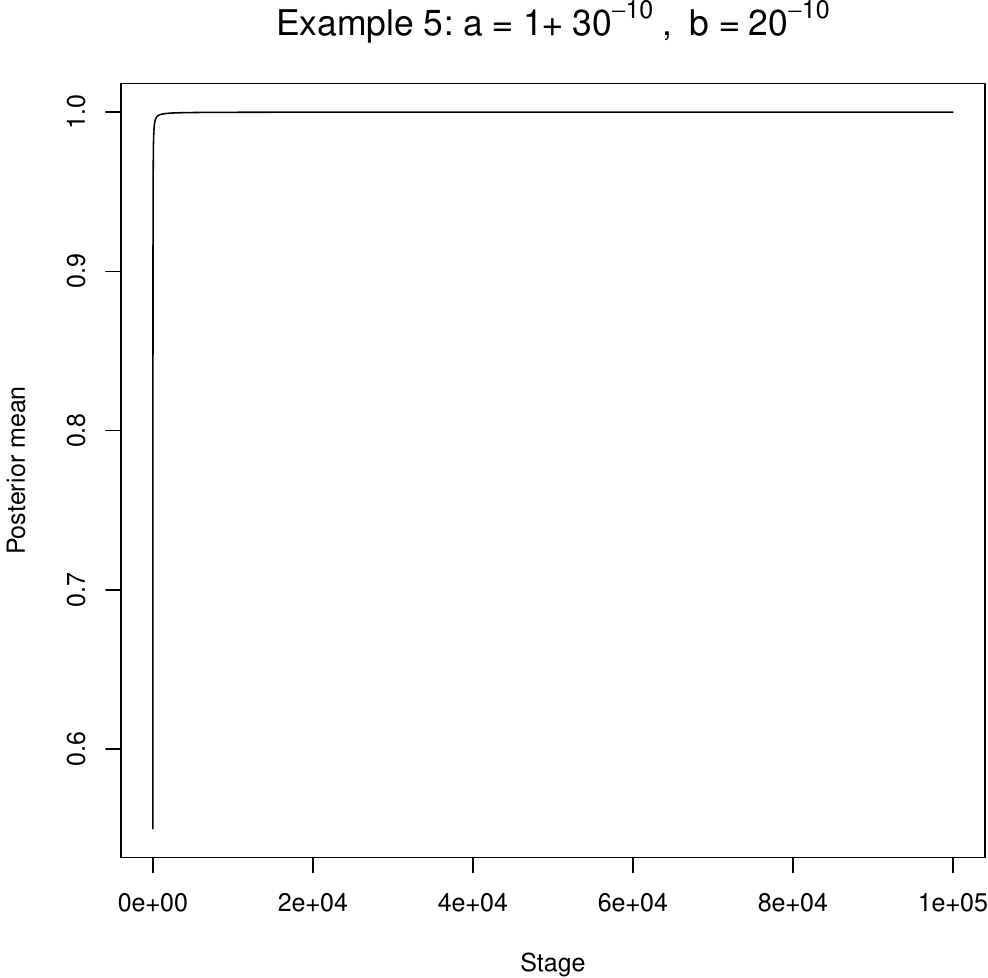}}
\hspace{2mm}
\subfigure [Divergence: $a=1/2,b=1/2$.
The posterior of $p_{10,k}$ converges to 1 as $k\rightarrow\infty$.]
{ \label{fig:limpoints_a_1_2_b_1_2}
\includegraphics[width=6cm,height=5cm]{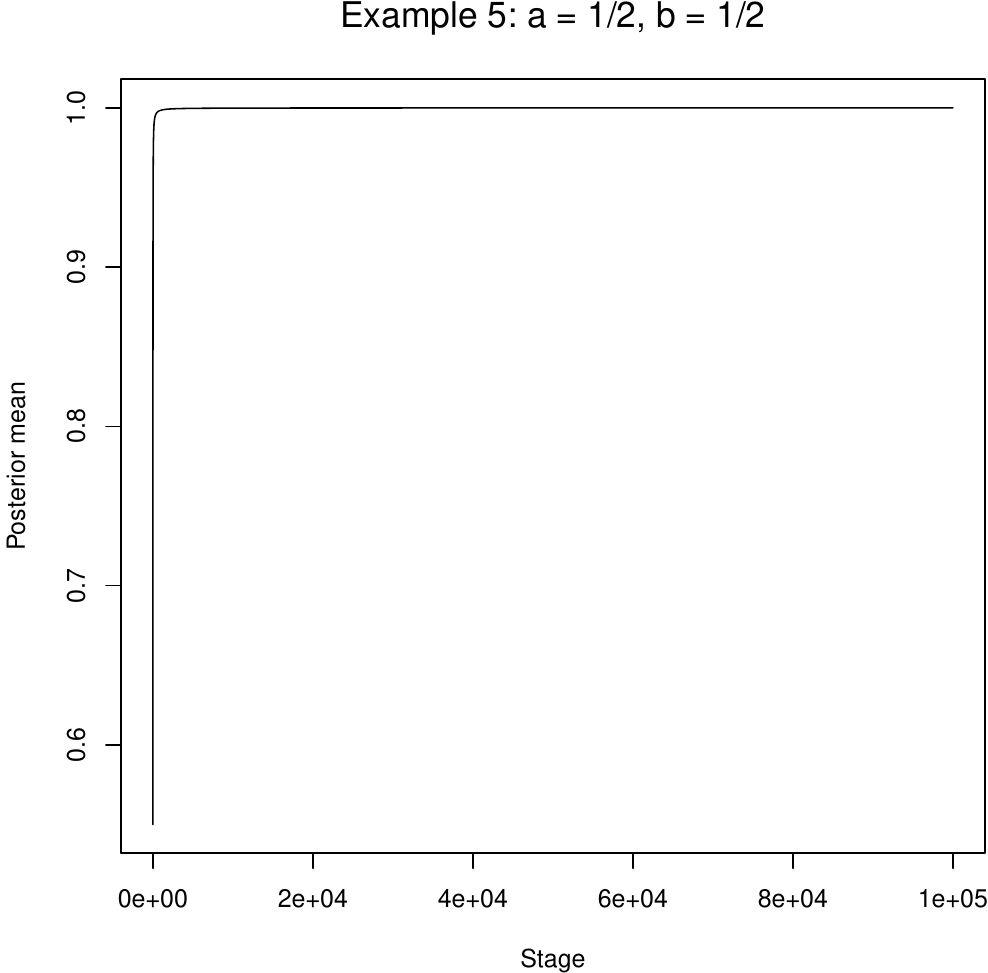}}\\
\subfigure [Divergence: $a=\frac{1}{2}\left(1-10^{-11}\right),
b=\frac{1}{2}\left(1-10^{-11}\right)$.
The posterior of $p_{10,k}$ converges to 1 as $k\rightarrow\infty$.]
{\label{fig:limpoints_a+b_less_1}
\includegraphics[width=6cm,height=5cm]{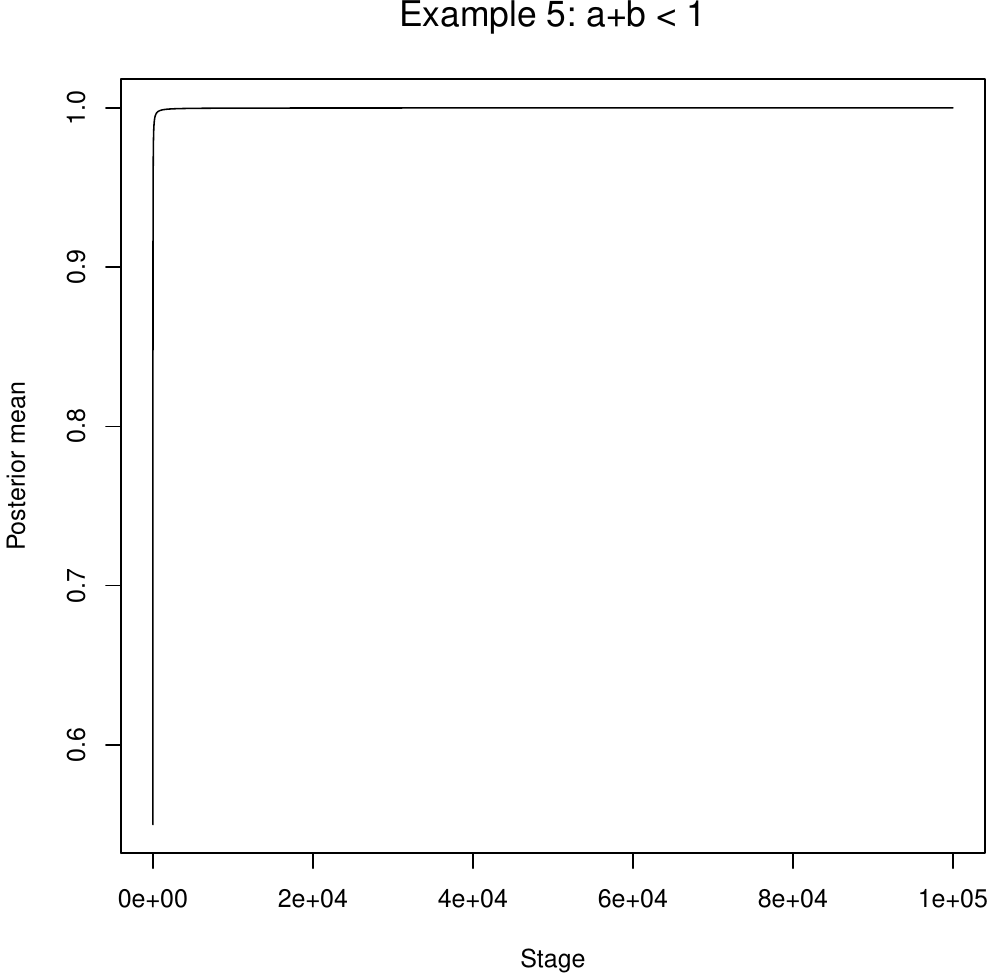}}
\caption{Illustration of the Dirichlet process based theory with Example 5: 
For $(a=2,b=1)$ in the series (\ref{eq:example5}), $\frac{m}{M}=\frac{6}{10}<0.9$, indicating convergence, 
for $(a=1+20^{-10},b=10^{-10})$, $\frac{m}{M}=\frac{6}{10}<0.9$, indicating convergence, for
$(a=1+30^{-10},b=20^{-10})$, $\frac{m}{M}=\frac{6}{10}<0.9$, indicating convergence,
for $(a=1/2,b=1/2)$, $\frac{m}{M}=\frac{10}{10}>0.9$, indicating divergence,  
and for $\left(a=\frac{1}{2}\left(1-10^{-11}\right),b=\frac{1}{2}\left(1-10^{-11}\right)\right)$,
$\frac{m}{M}=\frac{10}{10}>0.9$, indicating divergence.}
\label{fig:limpoints_example5}
\end{figure}

\section{Application of the Bayesian multiple limit points theory to Riemann Hypothesis}
\label{sec:Bayesian_limpoints_RH}
To strengthen our result on Riemann Hypothesis presented in Section 6 of our main manuscript 
we consider application of our Bayesian multiple limit points theory to Riemann Hypothesis.

\subsection{Choice of $\rho(\theta)$ in 
$\left(\frac{\exp\left(S_{1,k}\right)}{1+\exp\left(S_{1,k}\right)}\right)^{\rho(\theta)}$}
\label{subsec:rho_choice_RH}
For Riemann Hypothesis, $\theta=a$; we choose, for $j\geq 1$, 
\begin{equation}
\tilde\rho(\theta)=a^6. 
\label{eq:rho_choice_RH}
\end{equation}
The reason for such choice with a relatively large power is to allow discrimination between 
$\left(\frac{\exp\left(S_{1,k}\right)}{1+\exp\left(S_{1,k}\right)}\right)^{\rho(\theta)}$ for close values of $a$. However,
substantially large powers of $a$ are not appropriate because that would make the aforementioned
term too small to enable detection of divergence. In fact, we have chosen the power after much experimentation.
Implementation of our methods takes about 2 seconds on our VMWare, with $10^5$ iterations.

\subsection{Results}
\label{subsec:results_RH_limpoints}

The results of application of our ideas on multiple limit points are depicted in Figures 
\ref{fig:RH_DP_1}, \ref{fig:RH_DP_2} and \ref{fig:RH_DP_3}. The values of $m/M$ and the
thumb rule proposed in Section \ref{subsec:thumb_rule} show that all the results are
consistent with those obtained in Section 6. 
For $a=2$ and $a=3$ we obtained $m/M=0.1$, but the existing theory and our results 
reported in Section 6 
confirm that the series is convergent, and not oscillating, for these values. 
There seems to be a slight discrepancy
only regarding the location of the change point of convergence. 
In this case, unlike $a=0.72$ as obtained in Section 6, 
we obtained $a=0.7$ as the change point (see panel (b) of Figure \ref{fig:RH_DP_2}).

This (perhaps) negligible difference notwithstanding, both of our methods are remarkably
in agreement with each other, emphasizing our point that Riemann Hypothesis can not be
completely supported.

\begin{figure}
\centering
\subfigure [Divergence: $a=0.1$, $\frac{m}{M}=\frac{10}{10}$.]{ \label{fig:RH_DP_a_01}
\includegraphics[width=6cm,height=5cm]{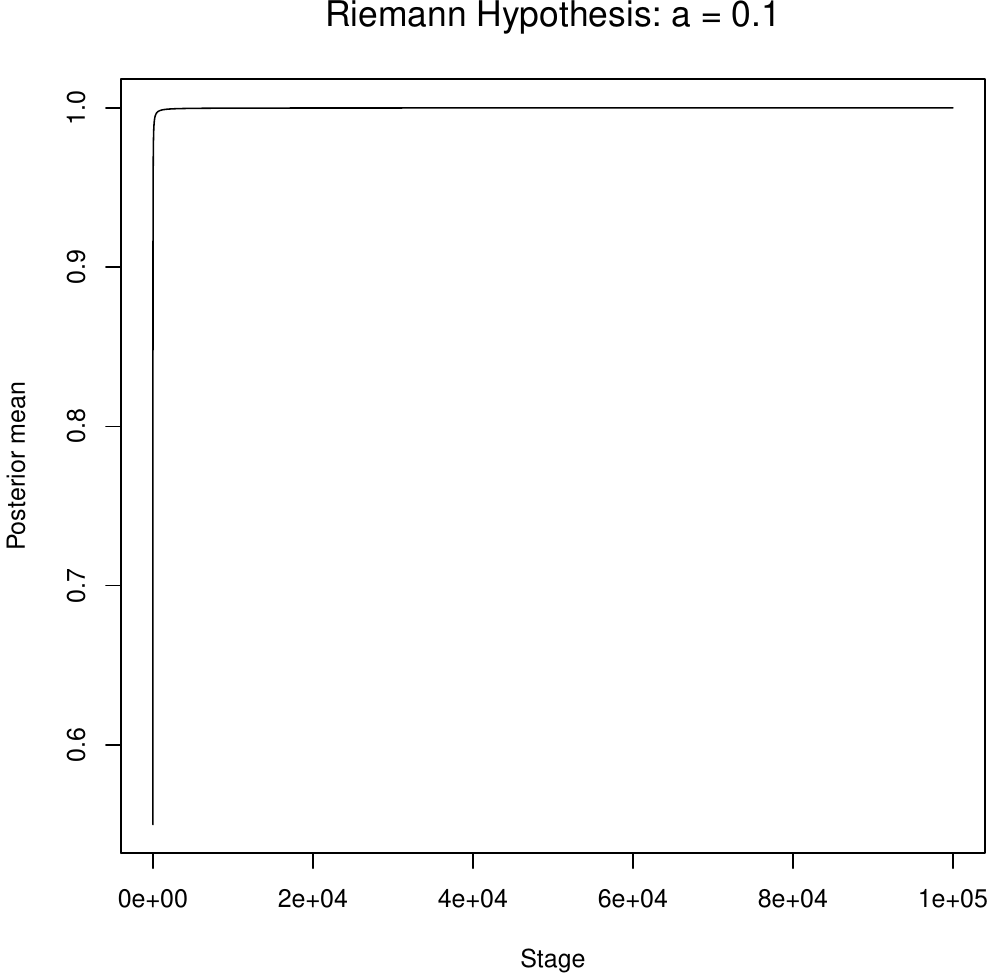}}
\hspace{2mm}
\subfigure [Divergence: $a=0.2$, $\frac{m}{M}=\frac{10}{10}$.]{ \label{fig:RH_DP_a_02}
\includegraphics[width=6cm,height=5cm]{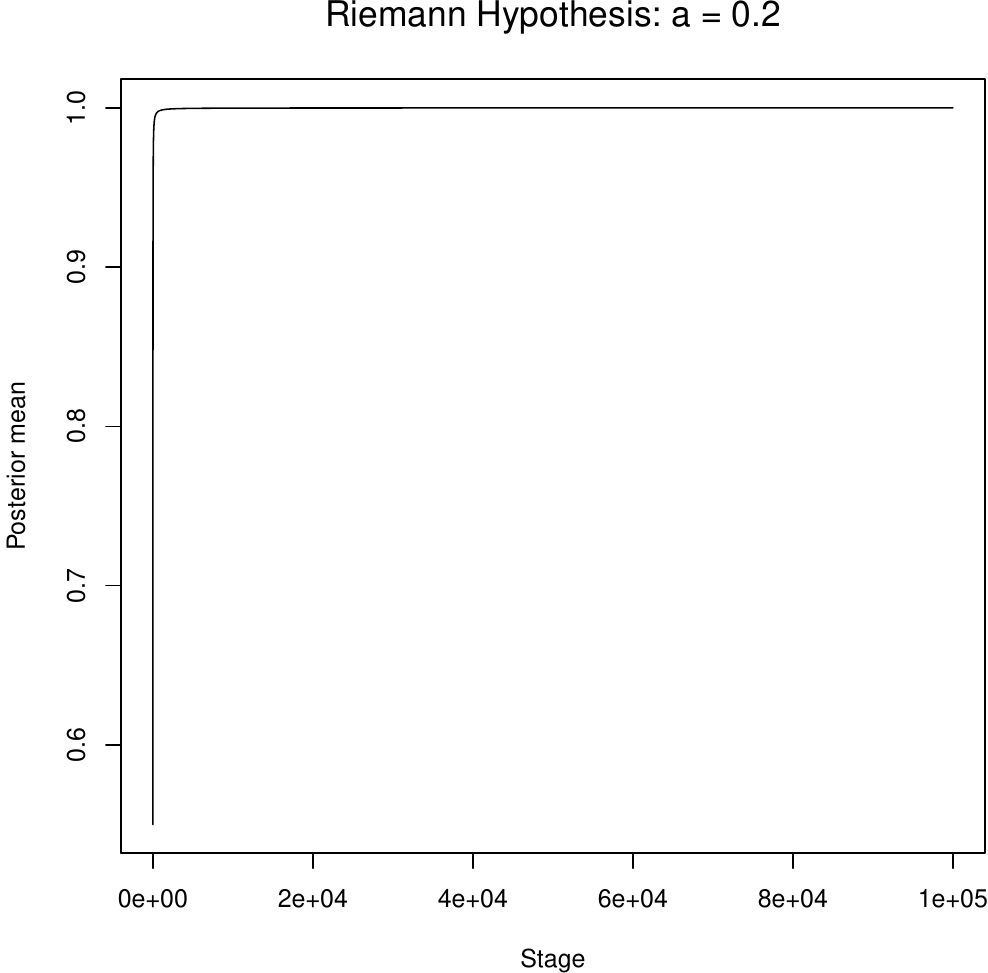}}\\
\subfigure [Divergence: $a=0.3$, $\frac{m}{M}=\frac{10}{10}$.]{ \label{fig:RH_DP_a_03}
\includegraphics[width=6cm,height=5cm]{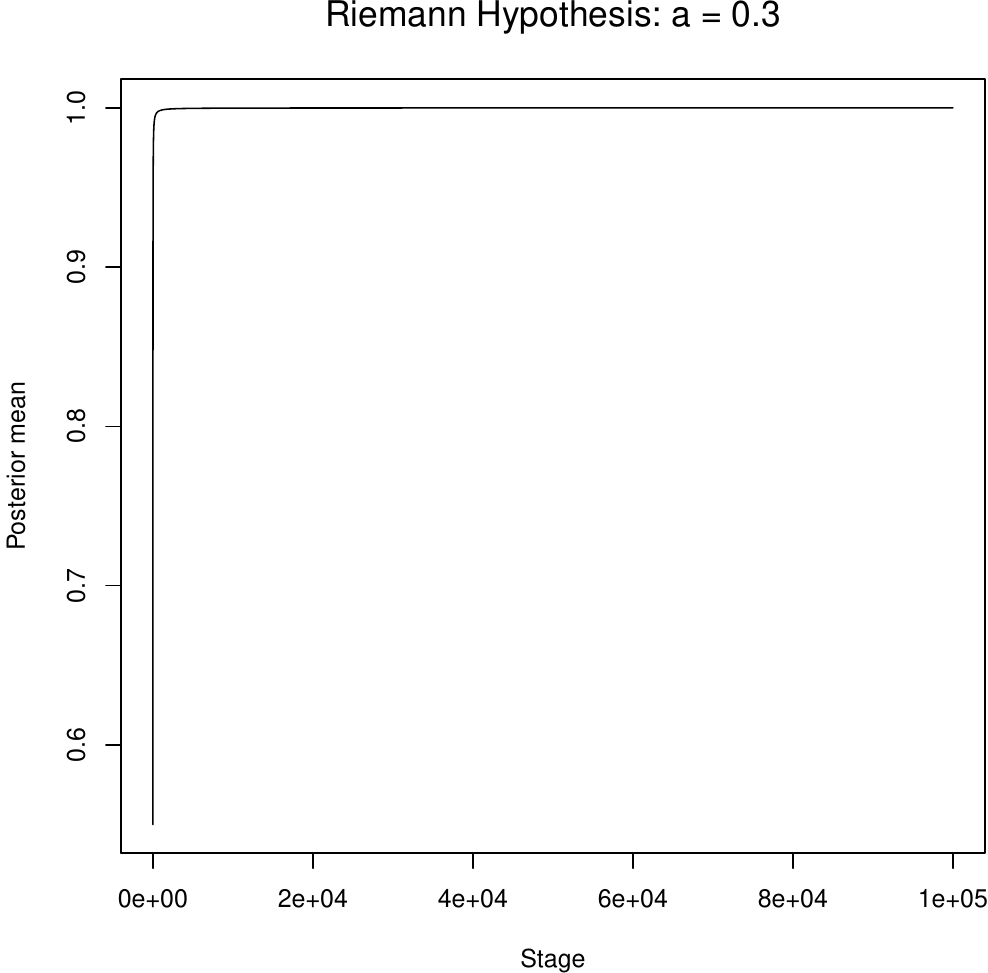}}
\hspace{2mm}
\subfigure [Divergence: $a=0.4$, $\frac{m}{M}=\frac{10}{10}$.]{ \label{fig:RH_DP_a_04}
\includegraphics[width=6cm,height=5cm]{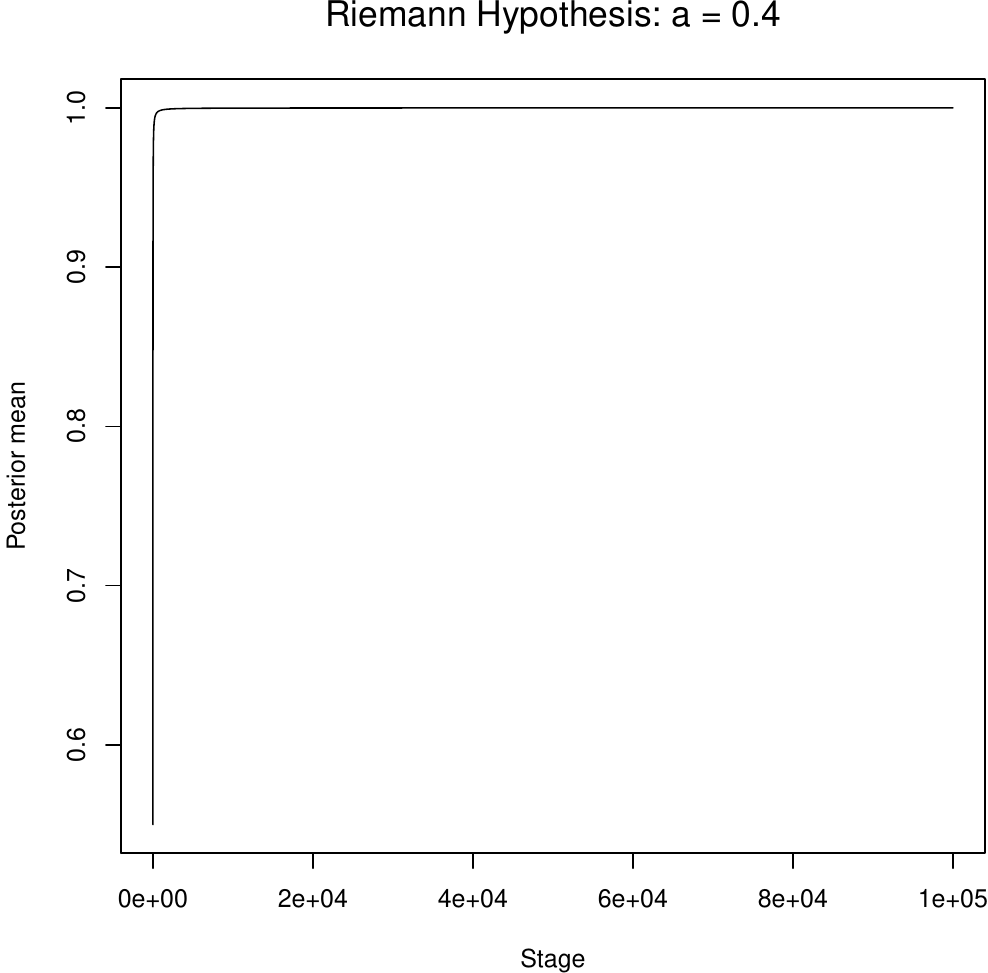}}\\
\subfigure [Divergence: $a=0.5$, $\frac{m}{M}=\frac{10}{10}$.]{ \label{fig:RH_DP_a_05}
\includegraphics[width=6cm,height=5cm]{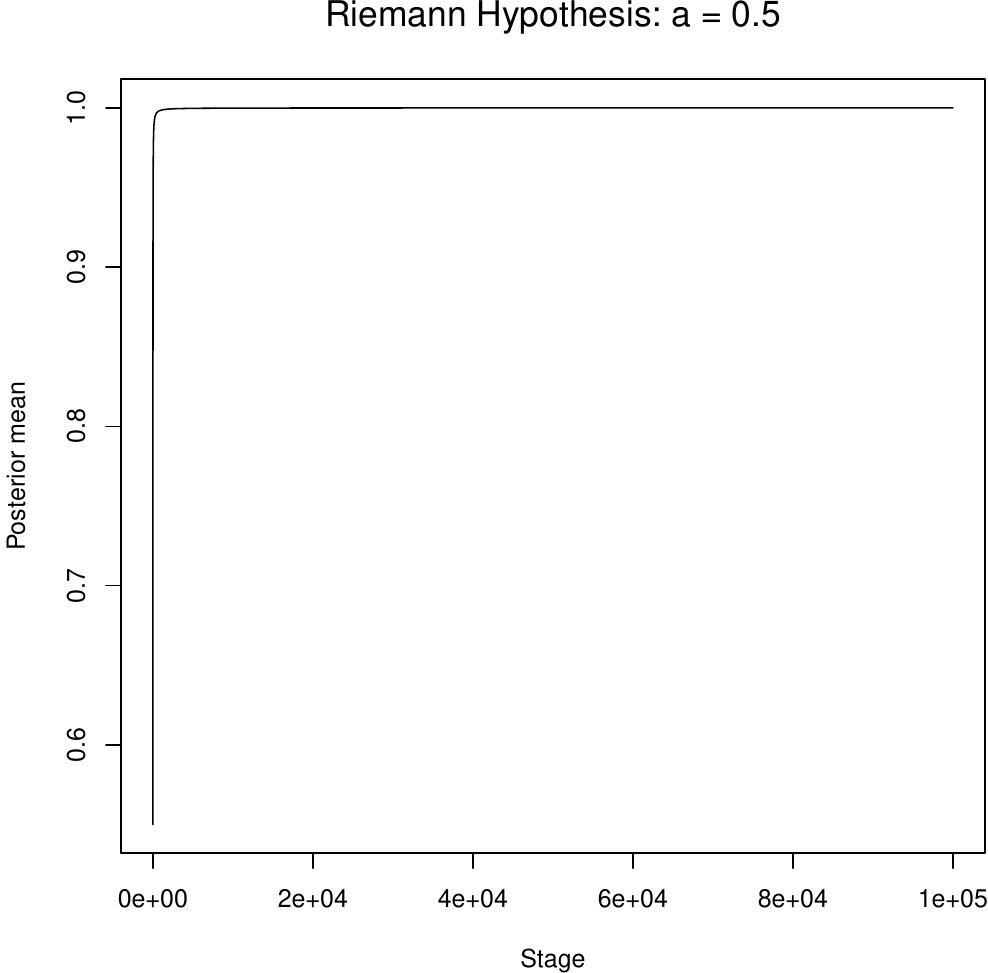}}
\hspace{2mm}
\subfigure [Divergence: $a=0.6$, $\frac{m}{M}=\frac{10}{10}$.]{ \label{fig:RH_DP_a_06}
\includegraphics[width=6cm,height=5cm]{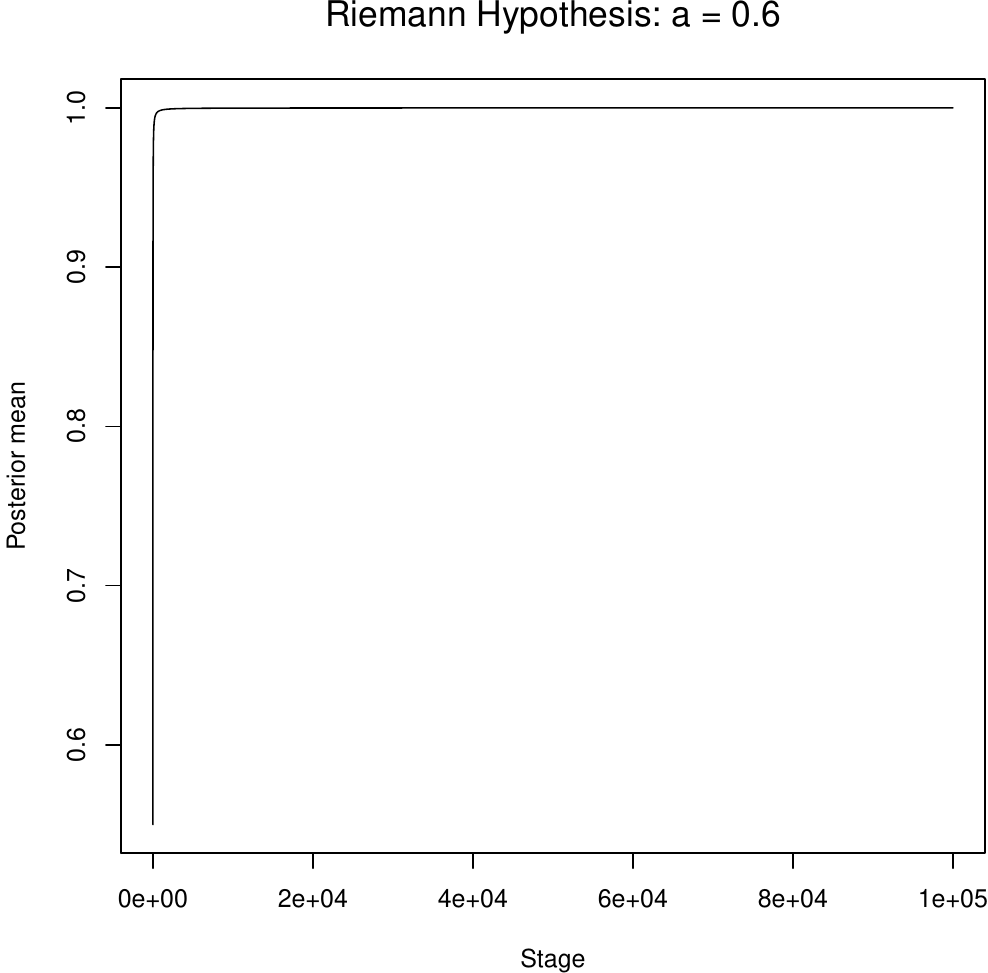}}
\caption{Riemann Hypothesis based on Bayesian multiple limit points theory: Divergence for 
$a=0.1$, $0.2$, $0.3$, $0.4$, $0.5$, $0.6$.}
\label{fig:RH_DP_1}
\end{figure}

\begin{figure}
\centering
\subfigure [Convergence: $a=0.7$, $\frac{m}{M}=\frac{9}{10}$.]{ \label{fig:RH_DP_a_07}
\includegraphics[width=6cm,height=5cm]{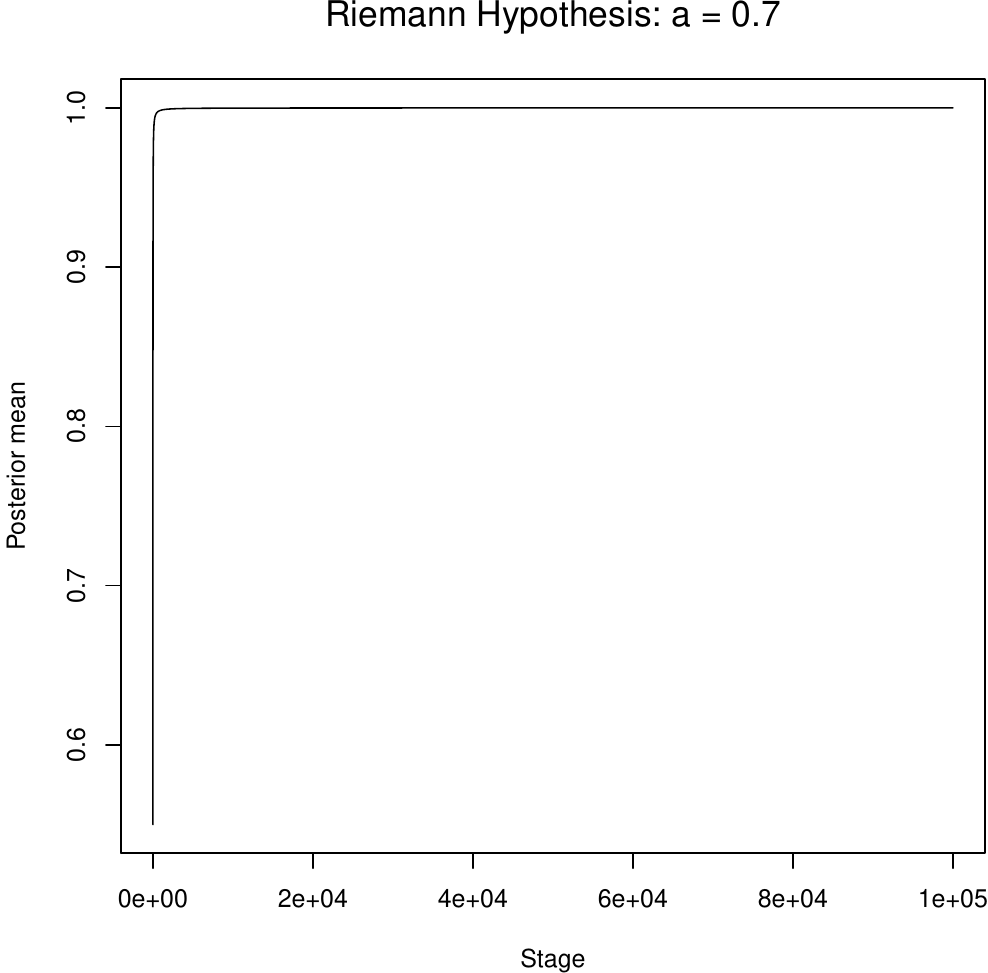}}
\hspace{2mm}
\subfigure [Convergence: $a=0.74$, $\frac{m}{M}=\frac{9}{10}$.]{ \label{fig:RH_DP_a_074}
\includegraphics[width=6cm,height=5cm]{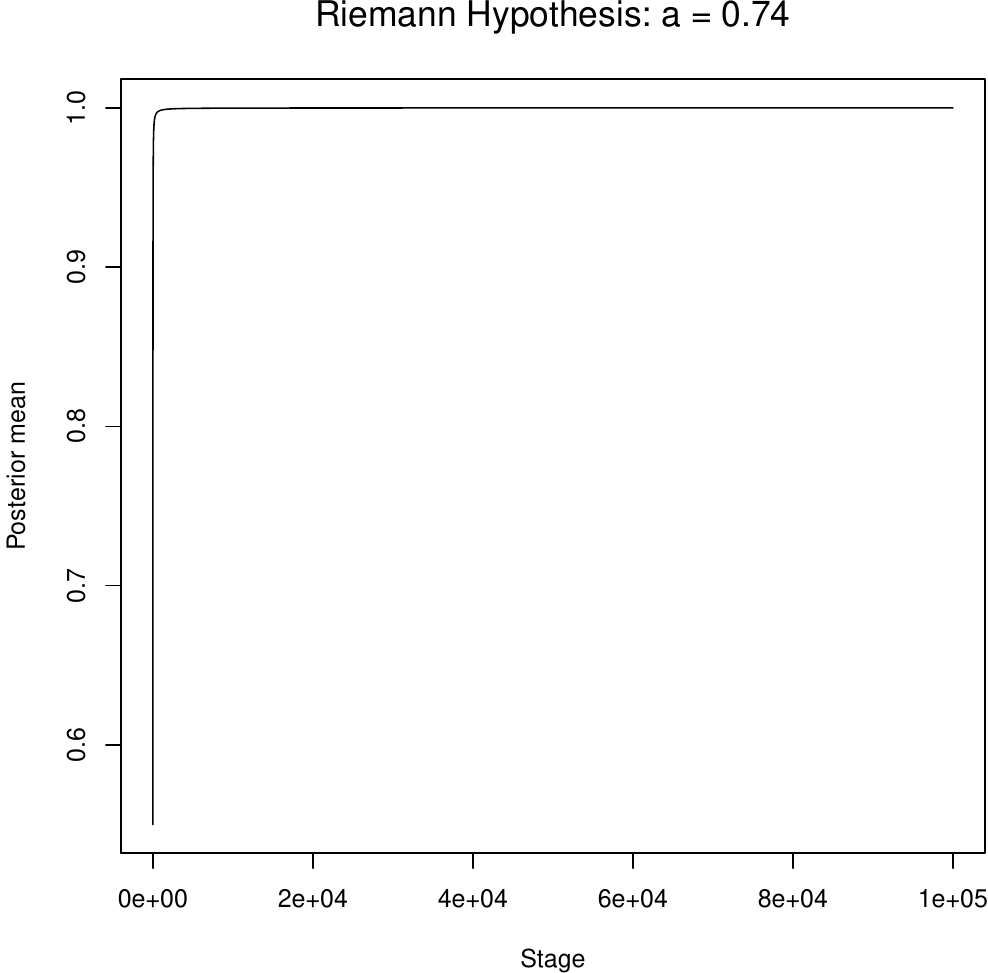}}\\
\subfigure [Convergence: $a=0.8$, $\frac{m}{M}=\frac{8}{10}$.]{ \label{fig:RH_DP_a_08}
\includegraphics[width=6cm,height=5cm]{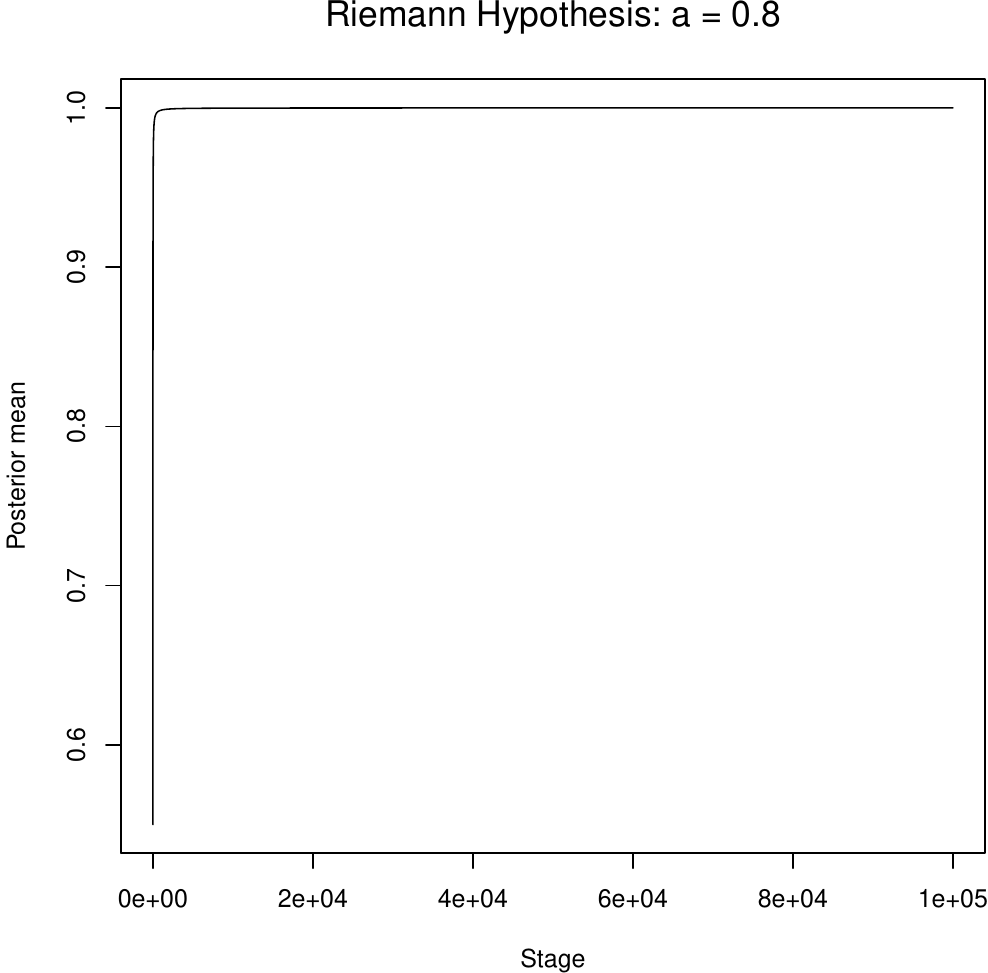}}
\hspace{2mm}
\subfigure [Convergence: $a=0.9$, $\frac{m}{M}=\frac{7}{10}$.]{ \label{fig:RH_DP_a_09}
\includegraphics[width=6cm,height=5cm]{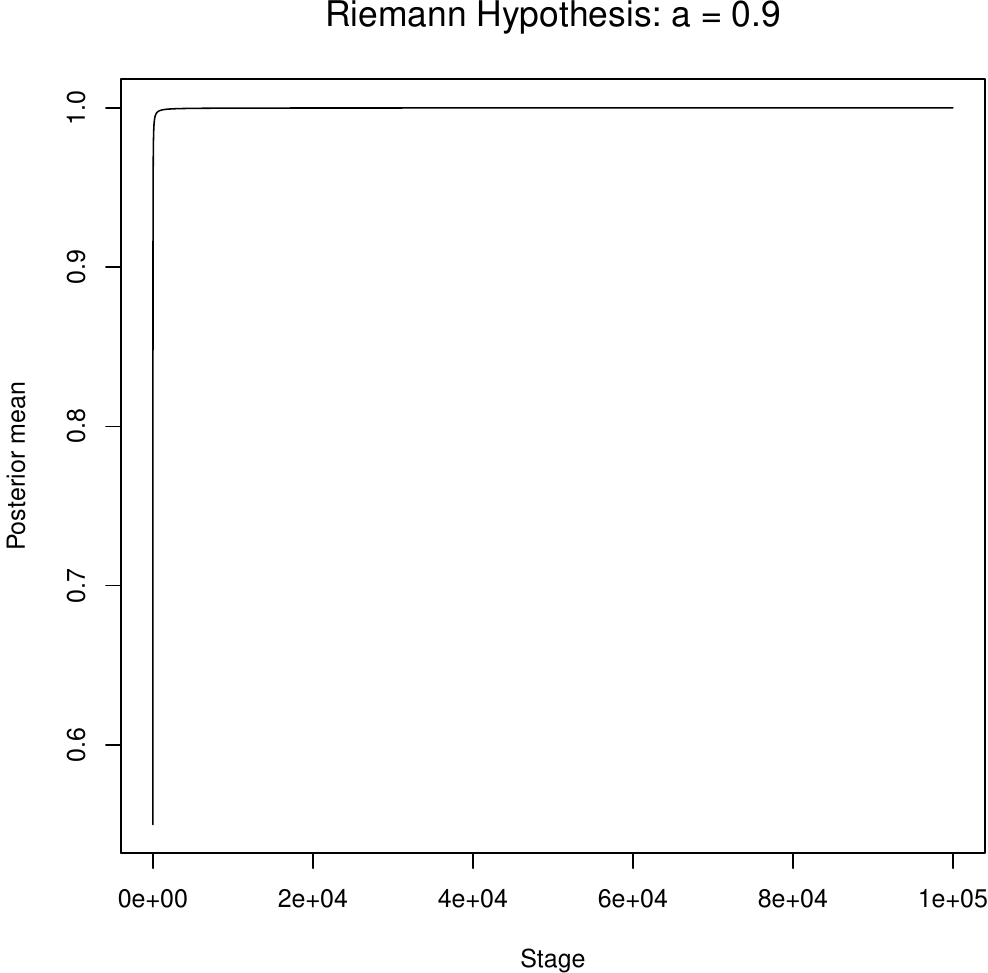}}\\
\subfigure [Convergence: $a=1.0$, $\frac{m}{M}=\frac{5}{10}$.]{ \label{fig:RH_DP_a_1}
\includegraphics[width=6cm,height=5cm]{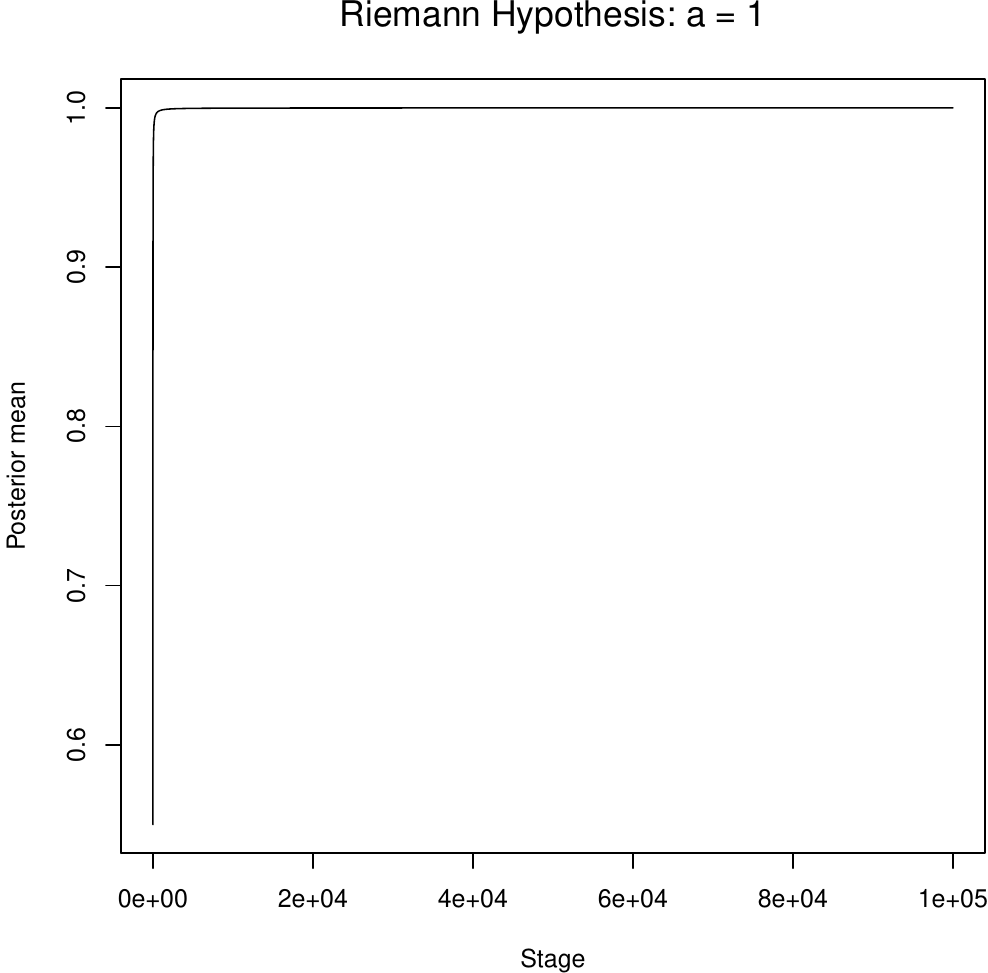}}
\hspace{2mm}
\subfigure [Convergence: $a=1+10^{-10}$, $\frac{m}{M}=\frac{5}{10}$.]{ \label{fig:RH_DP_a_1_e}
\includegraphics[width=6cm,height=5cm]{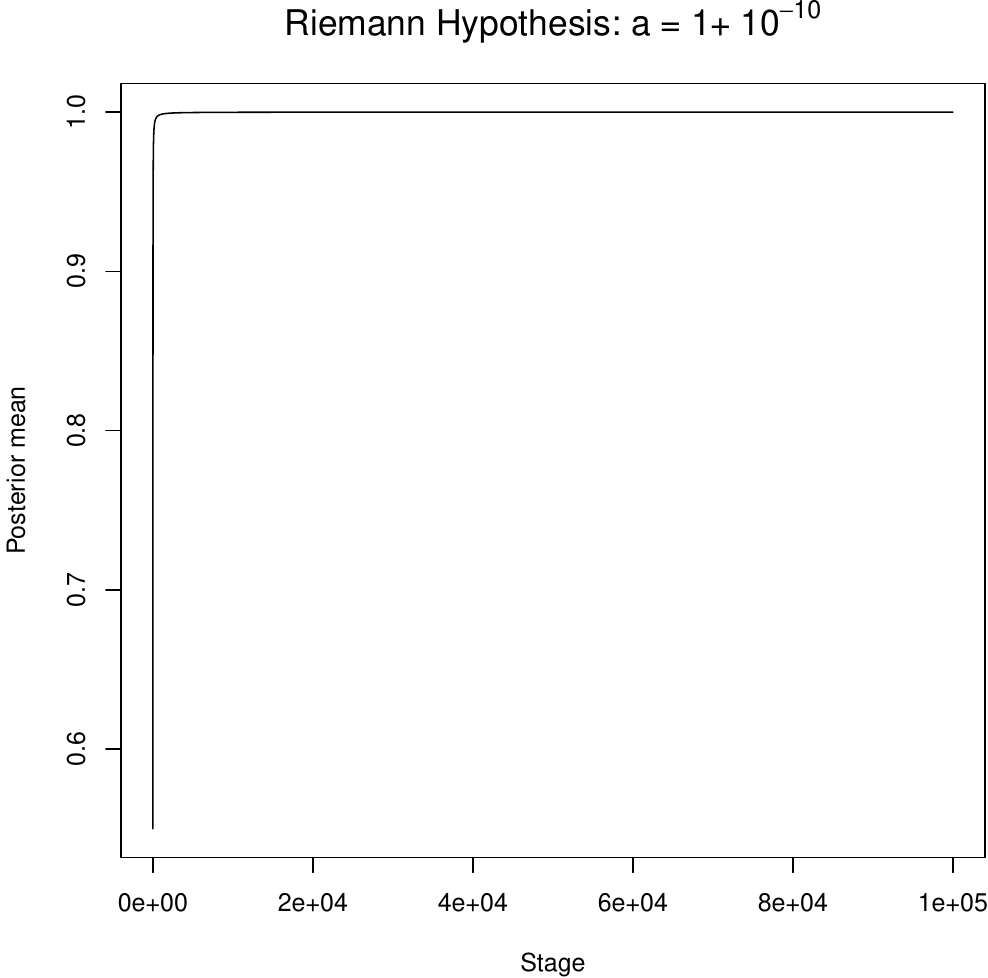}}
\caption{Riemann Hypothesis based on Bayesian multiple limit points theory: Divergence for 
$a=0.7$ but convergence for $a=0.74$, $0.8$, $0.9$, $1$, $1+10^{-10}$.}
\label{fig:RH_DP_2}
\end{figure}

\begin{figure}
\centering
\subfigure [Convergence: $a=2$, $\frac{m}{M}=\frac{1}{10}$.]{ \label{fig:RH_DP_a_2}
\includegraphics[width=6cm,height=5cm]{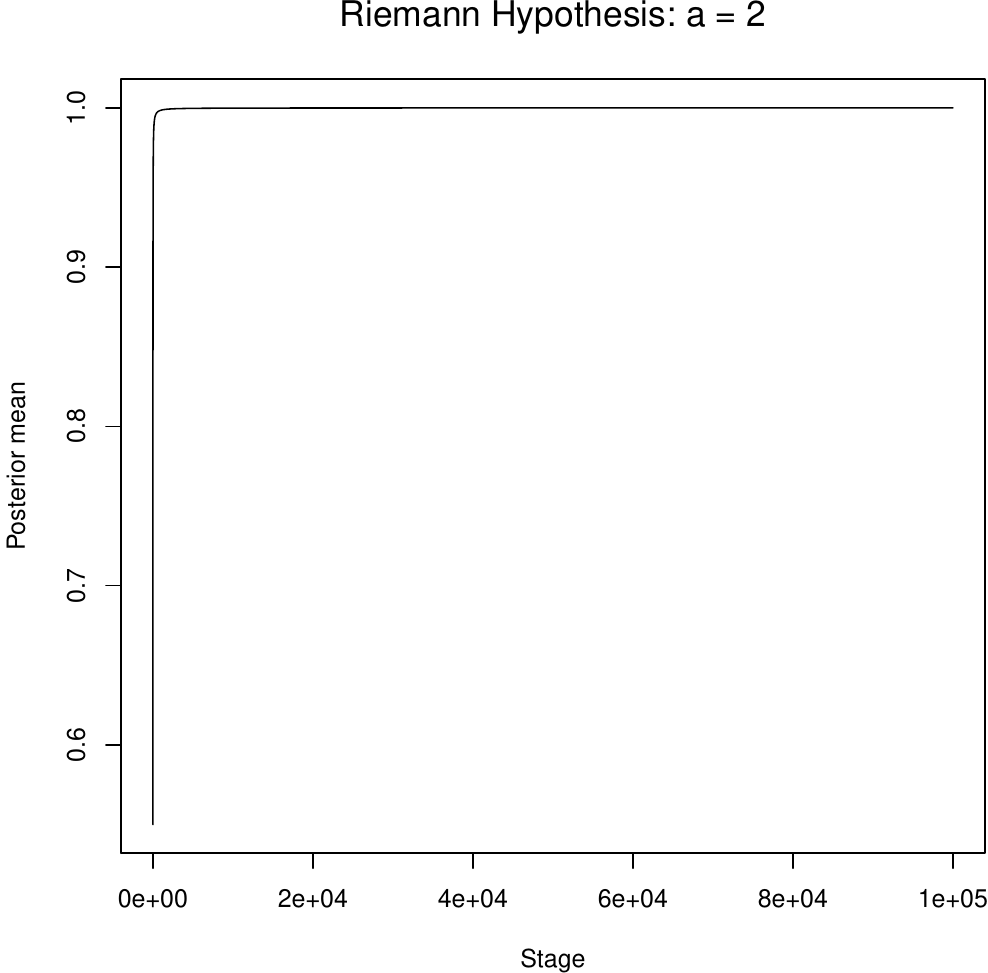}}
\hspace{2mm}
\subfigure [Convergence: $a=3$, $\frac{m}{M}=\frac{1}{10}$.]{ \label{fig:RH_DP_a_3}
\includegraphics[width=6cm,height=5cm]{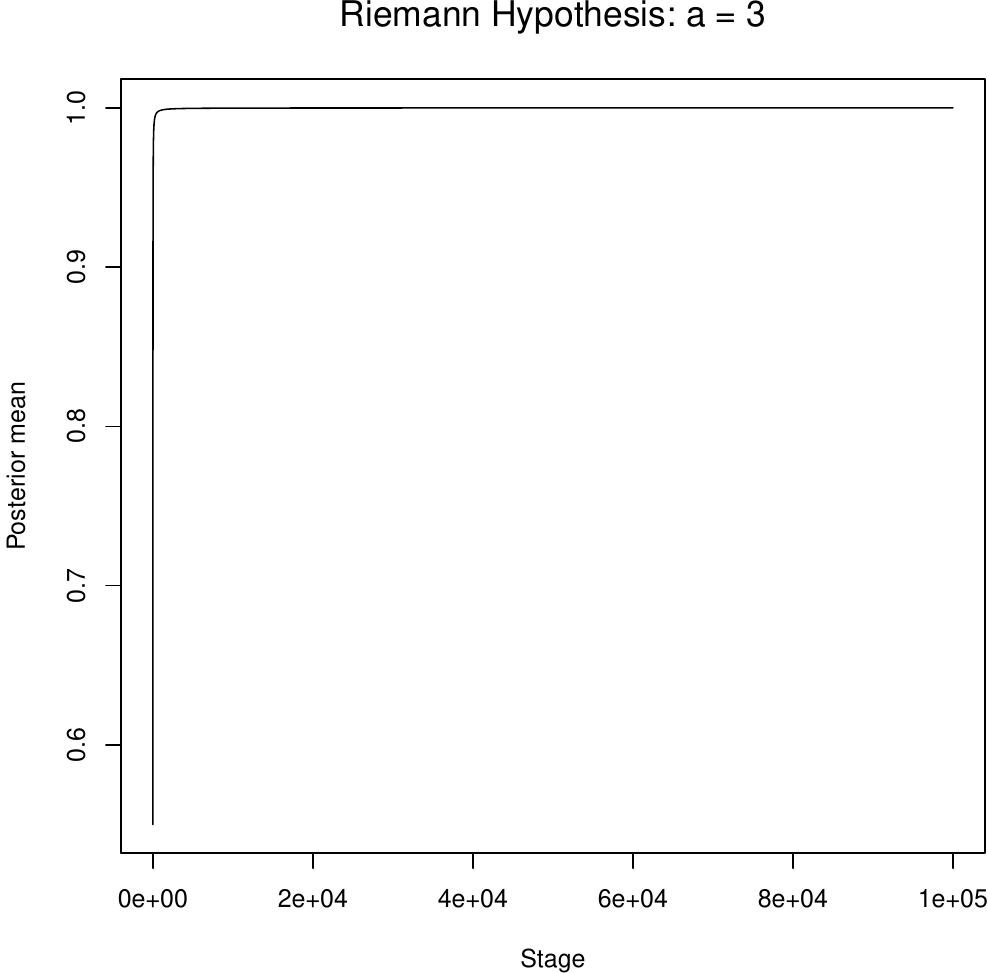}}
\caption{Riemann Hypothesis based on Bayesian multiple limit points theory: Convergence for 
$a=2$, $3$.}
\label{fig:RH_DP_3}
\end{figure}

\section{Characterization of Riemann Hypothesis based on Bernoulli numbers}
\label{sec:Bernoulli_RH}
Characterization of Riemann Hypothesis by convergence of inifinte sums associated with Bernoulli numbers
are provided in \ctn{Carey03} (unpublished, according to our knowledge).
In particular, it has been shown that
Riemann hypothesis is true if and only if the following series is convergent:
\begin{equation}
\tilde S_1=\sum_{m=1}^{\infty}\frac{\pi (4m+3)}{2^{4m+1}}
\sum_{k=0}^m(-1)^k\frac{{2m+1\choose k}{4m+2-2k\choose 2m+1}}{2m+2-2k}
\log\left(\frac{\left(2\pi\right)^{2m+2-2k}\left|B_{2m+2-2k}\right|}{2(2m+2-2k)^2(2m-2k)!}\right),
\label{eq:bernoulli_1}
\end{equation}
where $\left\{B_n;~n=0,1,\ldots\right\}$ are Bernoulli numbers characterized by their generating function
$\sum_{n=0}^{\infty}B_nx^n/n!=x/\left(\exp(x)-1\right)$. The Bernoulli numbers are related to the
Riemann zeta function by (see, for example \ctn{Sury03}) 
\begin{equation}
B_{2m}=(-1)^{m-1}\frac{2(2m)!}{(2\pi)^{2m}}\zeta(2m).
\label{eq:Riemann_Bernoulli}
\end{equation}
\ctn{Carey03} further showed that convergence of the related series
\begin{equation}
\tilde S_2=\sum_{m=1}^{\infty}\frac{\pi (4m+3)}{2^{4m+1}}
\sum_{k=0}^m(-1)^k\frac{{2m+1\choose k}{4m+2-2k\choose 2m+1}}{2m+2-2k}
\log\left((2m+1-2k)\frac{\left|B_{2m+2-2k}\right|}{\left|B_{2m+4-2k}\right|}\right),
\label{eq:bernoulli_2}
\end{equation}
is also equivalent to the assertion that Riemann hypothesis is correct.
However, the terms of both the series (\ref{eq:bernoulli_1}) and (\ref{eq:bernoulli_2}) tend to explode very quickly. 
Stirlings's approximation of the factorials involved in the summands facilitates computation of larger number
of summands compared to the original terms. In this context, note that Stirling's approximation
applied to the factorials in (\ref{eq:Riemann_Bernoulli}), along with the approximation 
$\zeta(2m)\sim 1$, as $m\rightarrow\infty$, 
lead the following asymptotic form of $B_{2m}$ as  as $m\rightarrow\infty$:
\begin{equation}
B_{2m}\sim (-1)^{m-1}4\sqrt{\pi m}\left(\frac{m}{\pi e}\right)^{2m}.
\label{eq:bernoulli_asymp}
\end{equation}
Figure \ref{fig:RH_Bernoulli} shows the logarithms of the first few terms $a_m$ of the above two series,
based on the actual terms $a_m$ and the Stirling-approximated $a_m$ (ignoring a multiplicative constant); 
the rest of the terms become
too large to be reliably computed, even with Stirling's approximation. 
The bottomline that emerges from
(\ref{fig:RH_Bernoulli}) is that the series $\tilde S_1$ and $\tilde S_2$ appear to be clearly divergent,
providing some support to our result on Riemann hypothesis.
\begin{figure}
\centering
\subfigure [Actual terms of series $\tilde S_1$.]{ \label{fig:bernoulli_1_actual}
\includegraphics[width=6cm,height=5cm]{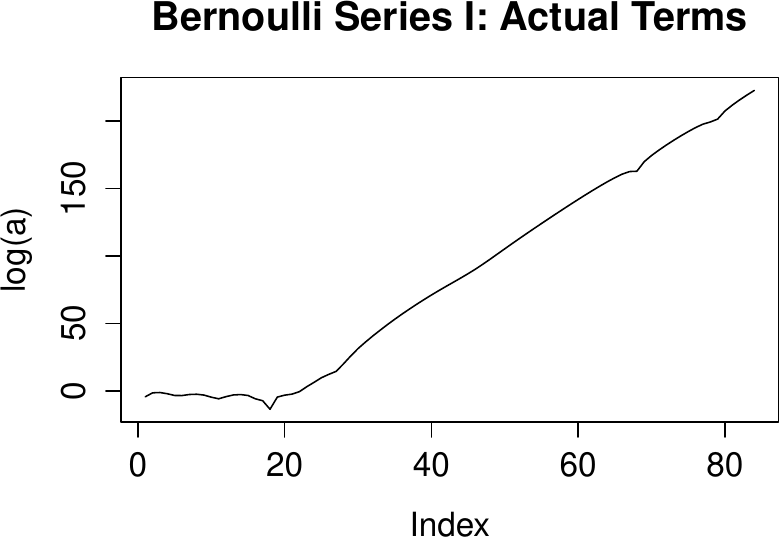}}
\hspace{2mm}
\subfigure [Stirling based terms of series $\tilde S_1$.]{ \label{fig:bernoulli_1_stirling}
\includegraphics[width=6cm,height=5cm]{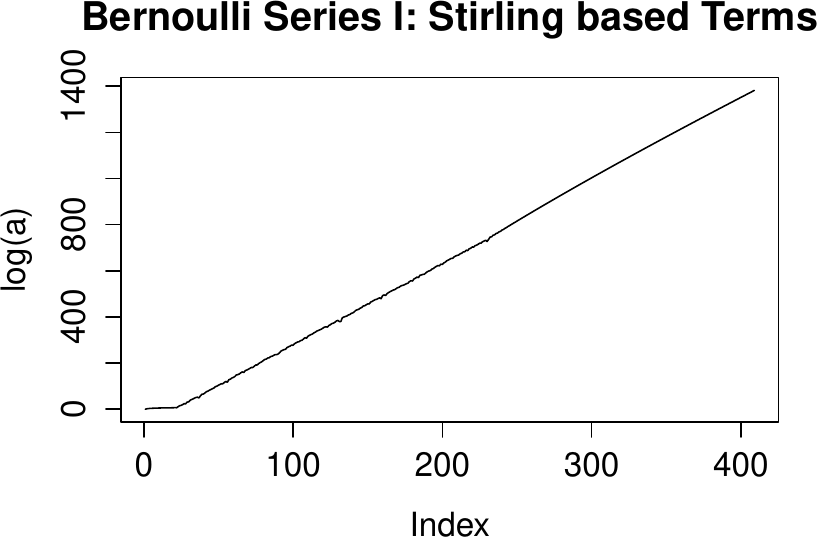}}\\
\subfigure [Actual terms of series $\tilde S_2$.]{ \label{fig:bernoulli_2_actual}
\includegraphics[width=6cm,height=5cm]{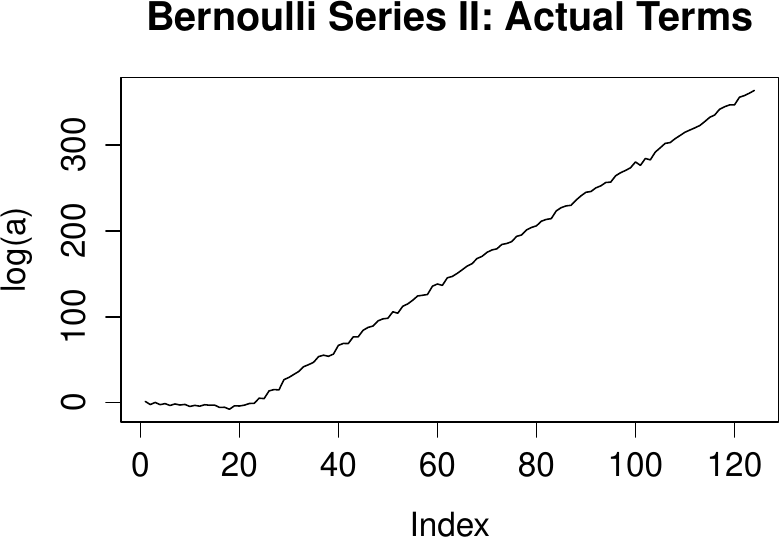}}
\hspace{2mm}
\subfigure [Stirling based terms of series $\tilde S_2$.]{ \label{fig:bernoulli_2_stirling}
\includegraphics[width=6cm,height=5cm]{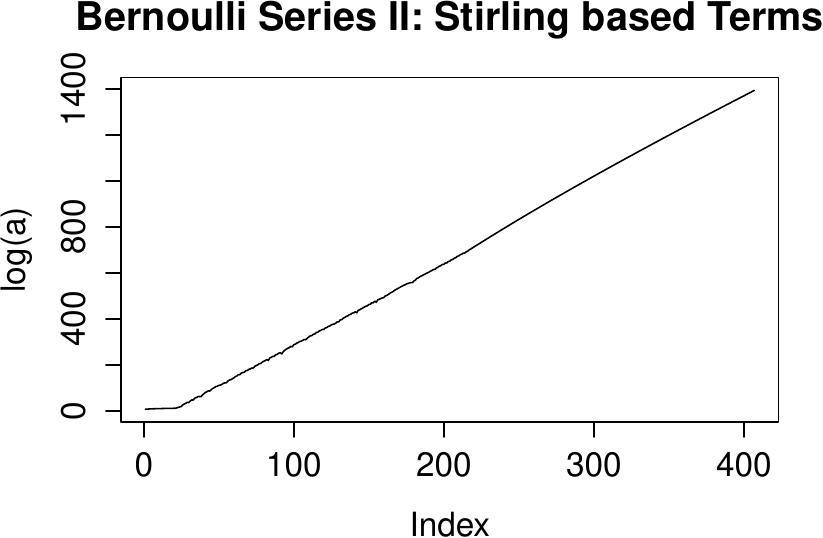}}
\caption{Actual and Stirling-approximated terms $a_m$ of the series $\tilde S_1$ and $\tilde S_2$.}
\label{fig:RH_Bernoulli}
\end{figure}

\newpage

\normalsize
\bibliographystyle{natbib}
\bibliography{irmcmc}

\end{document}